\documentclass[12pt]{amsart}

\usepackage{mathtools}
\usepackage{stmaryrd} 
\usepackage{amssymb} 

\usepackage[dvipsnames]{xcolor}

\usepackage{soul}
\usepackage{blindtext} 
\usepackage[linktocpage, pdfborder={0 0 .3}]{hyperref}
\usepackage{cleveref} 
\usepackage{soul}
\usepackage{overpic, float}
\setstcolor{red}
\usepackage[normalem]{ulem}
\usepackage{eucal} 
\usepackage{stmaryrd} 

\usepackage{tikz-cd}
\usepackage{tikz}
\usetikzlibrary{shapes.geometric, arrows}
\usetikzlibrary{fit,calc,positioning}
\usetikzlibrary{tikzmark,quotes} 
\usetikzlibrary{arrows.meta}
\usepackage{enumitem} 

\usepackage{contour}
\contourlength{1pt}

\newcommand{\R}{\mathbb{R}}
\newcommand{\C}{\mathbb{C}}
\newcommand{\Z}{\mathbb{Z}}
\newcommand{\F}{\mathbb{F}}
\renewcommand{\H}{\mathbb{H}}
\newcommand{\TT}{\mathcal{T}}
\newcommand{\BB}{\mathcal{B}}
\newcommand{\PP}{\mathcal{P}}
\newcommand{\GG}{\mathcal{G}}
\newcommand{\RR}{\mathcal{R}}
\newcommand{\CP}{\mathbb{C}{\rm P}}
\newcommand{\ML}{\mathcal{ML}}
\newcommand{\E}{\mathcal{E}}
\newcommand{\QF}{\mathcal{QF}}
\newcommand{\LL}{\mathcal{L}}

\newcommand{\PSL}{\mathrm{PSL}}
\newcommand{\SL}{\mathrm{SL}}

\newcommand{\QD}{\mathrm{QD}}
\newcommand{\MCG}{\mathrm{MCG}}
\newcommand{\PML}{\mathrm{PML}}
\newcommand{\PMF}{\mathrm{PMF}}

\newcommand{\Hol}{\operatorname{Hol}}
\newcommand{\Ad}{\operatorname{Ad}}
\renewcommand{\Im}{\operatorname{Im}}

\newcommand{\Area}{\operatorname{Area}}
\newcommand{\tr}{\operatorname{tr}}
\newcommand{\length}{\operatorname{length}}
\newcommand{\Isom}{\operatorname{Isom}}

\newcommand{\ep}{\epsilon}
\newcommand{\del}{\delta}
\newcommand{\gam}{\gamma}
\newcommand{\Gam}{\Gamma}
\newcommand{\Lam}{\Lambda}
\newcommand{\lam}{\lambda}

\newtheorem{lemma}{Lemma}[section]
\newtheorem{claim}[lemma]{Claim}
\newtheorem{proposition}[lemma]{Proposition}
\newtheorem{theorem}[lemma]{Theorem}
\newtheorem{question}[lemma]{Question}
\newtheorem{remark}[lemma]{Remark}

\newtheorem{corollary}[lemma]{Corollary}
\newtheorem{definition}[lemma]{Definition}

\newtheorem{thm}[]{Theorem}

\newcommand{\col}{\colon}
\newcommand{\sub}{\subset}
\newcommand{\minus}{\setminus}
\newcommand{\ti}{\tilde}
\newcommand{\til}{\tilde}

\newcommand{\bdr}{\partial}

\DeclareRobustCommand{\rchi}{{\mathpalette\irchi\relax}}
\newcommand{\irchi}[2]{\raisebox{\depth}{$#1\chi$}} 

 \newcommand{\Label}[1]{\label{#1}\textcolor{green}{\tiny #1} }
\renewcommand{\Label}[1]{\label{#1}}

\usepackage[normalem]{ulem}
\usepackage{xcolor}
\makeatletter
\def\squigred{\bgroup \markoverwith{\textcolor{red}{\lower3.5\p@\hbox{\sixly \char58}}}\ULon}
\makeatother

\makeatletter
\def\squigblue{\bgroup \markoverwith{\textcolor{blue}{\lower3.5\p@\hbox{\sixly \char58}}}\ULon}
\makeatother

\definecolor{dblue}{cmyk}{1,.5, 0,.1}
\definecolor{arsenic}{rgb}{0.23, 0.27, 0.49}
\newcommand{\note}[1]{\marginpar{\textcolor{dblue}{\tiny #1}}}

 \definecolor{calpolypomonagreen}{rgb}{0.12, 0.3, 0.17}
 
  \definecolor{darkbyzantium}{rgb}{0.6, 0.3, 0.4}
 \definecolor{azure}{rgb}{0.0, 0.5, 1.0}
 \definecolor{cittingcolor}{cmyk}{60,0,10,0}
  
 \hypersetup{pdfborder=0 0 .3}
	\hypersetup{
    colorlinks=true,
   linkcolor=darkbyzantium,
    filecolor=magenta,      
    urlcolor=NavyBlue,
    citecolor = calpolypomonagreen,
    pdftitle={Overleaf Example},
    pdfpagemode=FullScreen,
    }

 \usepackage{xparse}

\ExplSyntaxOn
\NewDocumentCommand{\dslash}{O{}}
 {
  \str_case:nn { #1 }
   {
    {}{/\mkern-6mu/}
    {\big}{\big/\mkern-7mu\big/}
    {\Big}{\Big/\mkern-10mu\Big/}
    {\bigg}{\bigg/\mkern-14mu\bigg/}
    {\Bigg}{\Bigg/\mkern-18mu\Bigg/}
   }
 }
\ExplSyntaxOff

\newcommand{\Qed}[1]{\nopagebreak[4]{\tiny \hfill
\fbox{\ref{#1}} \linebreak }\pagebreak[2]}

\usepackage{tikz-cd}
\tikzset{
  symbol/.style={
    draw=none,
    every to/.append style={
      edge node={node [sloped, allow upside down, auto=false]{$#1$}}}
  }
}

\newcommand{\function}[5]{%
  \begin{tikzcd}[
    column sep=2em,
    row sep=1ex,
    ampersand replacement=\&
  ]
  #1\colon \&[-3em]
  #2\vphantom{#3} \arrow[r] \&
  #3\vphantom{#2} \\
  \&
  #4\vphantom{#5} \arrow[u,symbol=\in] \arrow[r,mapsto] \&
  #5\vphantom{#4} \arrow[u,symbol=\in]
  \end{tikzcd}%
}

\usepackage{fancybox}

\makeatletter
\def\uwave{\bgroup \markoverwith{\lower3.5\p@\hbox{\sixly \textcolor{blue}{\char58}}}\ULon}
\font\sixly=lasy6 
\makeatother

\setstcolor{red}

\begin{document}

 \title[\today]{Bending Teichm\"uller spaces and character varieties}

\author{Shinpei Baba}
\address{University of Osaka}

\email{sb.sci@osaka-u.ac.jp}

\maketitle
\begin{abstract}  
   We consider the mapping $b_L\col\TT \to \rchi$ from the Fricke-Teichmüller space $\TT$ 
  into the $\PSL_2\C$-character variety $\rchi$ of the surface, obtained by bending Fuchsian representations along a fixed measured lamination $L$.
   We prove that this mapping is an equivariant symplectic real-analytic embedding, and, for almost all measured laminations, proper.

 We also show that this ``bending map''  $b_L\col \TT \to \rchi$ extends continuously almost-everywhere to the canonical inclusion map from the Thurston boundary of $\TT$ into the Morgan-Shalen boundary of $\rchi$.
 
Moreover we ``complexify" this bending map in a geometric manner.
 Namely, we symplectically embed this real-analytic subvariety $\Im b_L$ into the product variety $\rchi \times \rchi$ by the diagonal mapping twisted by complex conjugation.  
Then we construct a closed $\C$-symplectic complex-analytic subvariety of $\rchi \times \rchi$ containing  $\Im b_L$ as a half-dimensional real-analytic subvariety.

   \end{abstract}

\setcounter{tocdepth}{1}

\tableofcontents

\section{Introduction}
Thurston discovered the bent hyperbolic surfaces $\tau$ on the boundary of the convex core of a (geometrically finite) hyperbolic three-dimensional manifold (\cite{Thurston-78}). 
The intrinsic metric of the convex surface is hyperbolic, and the surface is bent along a measured lamination, where the bending angles correspond to the transversal measure of the lamination. 
Such bent surfaces are particularly useful for capturing the global properties of the hyperbolic manifold. 

Lifting the convex surface $\tau$ to the universal cover $\H^3$ of the hyperbolic three-manifold,  we obtain an equivariant bending $\H^2 \to \H^3$ which preserves the (intrinsic) hyperbolic metric of the surface. 
Then, this bending map is equivariant via a holonomy representation of a surface group into $\PSL_2\C$.
Moreover, if $\tau$ is $\pi_1$-injective (equivalently, incompressible) in the ambient hyperbolic 3-manifold, then the bending map $\H^2 \to \H^3$ is a proper embedding.

In this paper, we utilize this bending construction in a new generalized manner and construct similar equivariant geometry-preserving mappings, in fact, at the level of associated deformation spaces. 

\subsection{Holonomy varieties}
Let $Y$ be a marked Riemann surface structure on a closed oriented surface $S$ of genus $g$ at least two.
Let  $\QD(Y)$ denote the space of the holomorphic quadratic differentials on $Y$, which is a complex vector space of dimension $3g-3$. 
Then $\QD(Y)$ is identified with the space $\PP_Y$ of all $\CP^1$-structures on $Y$, and this correspondence yields the {\it Schwarzian parameterization} of $\CP^1$-structures (see \cite{Dumas-08} for example).

Let $$\Hol\col \PP \to \rchi$$ be the holonomy map from the deformation space $\PP$ of all $\CP^1$-structures on $S$ to the $\PSL_2\C$-character variety $\rchi$ of $S$.
Recall that the character variety $\rchi$ is an affine algebraic variety.
 Its smooth part has Goldman's complex symplectic structure invariant under the action of the mapping class group; see  \cite{Goldman84SymplecticNatureOfFundamentalGroupsOfSurfaces}.
Many interesting properties of this mapping, associated with the Schwarzian parametrization,  have been discovered, and in particular the following holds.
\begin{theorem}
\label{PoincareHolonomyVariety}
The restriction of the holonomy map to $\PP_Y \cong \QD(Y)$ is a proper Lagrangian complex-analytic embedding into $\rchi$. 
\end{theorem}
The injectivity of \Cref{PoincareHolonomyVariety} is due to Poincaré \cite{Poincare884}; 
the properness is due to Kapovich \cite{Kapovich-95} (see \cite{Gallo-Kapovich-Marden} for the full proof; see also  \cite{Dumas18HolonomyLimit, Tanigawa99});
the Lagrangian property is proven by Kawai \cite{Kawai96SymplecticNature}. 
On the other hand, the entire holonomy map $\Hol\col \PP \to \rchi$ of $\CP^1$-structures is neither injective nor proper (see \cite{Hejhal-75}).

 By \Cref{PoincareHolonomyVariety}, for every marked Riemann surface structure $Y$, the vector space $\QD(Y) \cong \C^{3g-3}$ is properly embedded onto a half-dimensional smooth subvariety of $\rchi$.
 We call this image, associated with the Schwarzian parametrization, the {\sf Poincaré holonomy variety} of $Y$. 
In particular, the holonomy variety of $Y$ contains the Bers slice of $Y$ as a bounded pseudo-convex domain. 

   The Morgan-Shalen compactification of the character variety $\rchi$ consists of certain $\pi_1(S)$-actions of metric trees (\cite{Culler-Shalen-83, MorganShalen84}). 
   Dumas investigated the asymptotic behavior of the proper mapping $\Hol \vert \PP(X)$.  
   Namely, she showed that  $\Hol \vert \PP(X)$ extends to the ray compactification of the vector space $\QD(X)$ almost everywhere in a natural manner.

\begin{theorem}[Corollary E in \cite{Dumas18HolonomyLimit}]\Label{DumasHolonomyLimit}
Let  $q \in \QD(X) \minus \{0\}$ be a generic direction.
Let $V$ be the vertical measured foliation of $q$, and let $\ti{V}$ be the pull-back measured foliation of $V$ to the universal cover $\ti{X}$.
Then  $\Hol(t q)$ converges to the $\pi_1(S)$-action on the metric tree dual to $\ti{V}$ as $t \to \infty$. 

Moreover, $\Hol | \PP_X$ continuously extends the full measure set of the ray-compactification boundary $\bdr \QD(X)$ to the mapping to the Morgan-Shalen boundary of $\rchi$ in a natural manner.
\end{theorem}

\subsection{Real bending varieties}
Recall that $\CP^1$ is the ideal boundary of the hyperbolic three-space $\H^3$, and the automorphism group $\PSL_2\C$ of $\CP^1$ is identified with the group of orientation-preserving isometries of $\H^3$.
Utilizing this correspondence in a sophisticated manner, Thurston gave another parametrization of $\PP$, so that $\CP^1$-structures correspond to equivariant pleated surfaces in $\H^3$ (\S \ref{sRealBendingMap}). 
In this paper, we first yield an analogue of \Cref{PoincareHolonomyVariety} by specific slices in the Thurston parametrization of $\CP^1$-structures.

In fact, Tanigawa \cite{Tanigawa-97}, Wolf-Scannell \cite{ScannellWolf02GraftingMapOfTeichmullerSpace}, Dumas-Wolf \cite{Dumas-Wolf08} considered the $\CP^1$-structures with a fixed bending measured lamination and analyzed their conformal structures. 
In this paper, as in the holonomy variety,  we instead consider the holonomy representations of such $\CP^1$-structures.

For a measured lamination $L$ on a hyperbolic surface $\tau$,  we obtain an equivariant pleated surface in $\H^3$ by bending the universal cover of $\tau$, the hyperbolic plane $\H^2$, along the inverse-image $\ti{L}$ of $L$ in $\H^2$, and the pleated surface $\ti\tau \cong \H^2 \to \H^3$ is equivariant via a representation $\pi_1(S) \to \PSL_2\C$. (See \S \ref{sBendingDeformation} for details.)
Let $\TT$ be the space of marked hyperbolic structures on $S$, the Fricke-Teichmüller space; then $\TT$ is diffeomorphic to $\R^{6g-6}$ as a smooth manifold.
The Weil-Petersson form gives a symplectic structure on $\TT$, and Goldman extended it to a complex symplectic structure on the smooth part of $\rchi$ (\cite{Goldman84SymplecticNatureOfFundamentalGroupsOfSurfaces}).
For a measured lamination $L$ on $S$, let $b_L\col \TT \to \rchi$ be the map taking $\tau \in \TT$ to the holonomy representation $\pi_1(S) \to \PSL_2\C$ of the pleated surface given by $\tau$ and $L$.

This mapping is closely related to the Thurston parametrization of $\PP$ (\Cref{ThurstonParametrization}), and
the following theorem is an analogue of \Cref{PoincareHolonomyVariety} in the Thurston parametrization.
\begin{thm}[Theorems \ref{BendingTeich}, \ref{RealSymplectic}, Lemma \ref{RealBendingEquivariant}]\Label{BendingTeichmullerSpaces}
Let $L$ be an arbitrary measured lamination on $S$. 
Then, the bending map $b_L\col \TT \to \rchi$ is a real-analytic symplectic embedding, and it is equivariant by the subgroup of the mapping class group $\GG_L$ of $S$ preserving $L$.  

Moreover,  $b_L$ is proper if and only if $L$ contains no periodic leaf of weight $\pi$ modulo $2\pi$. 
\end{thm}

This preservation of the symplectic structure of $\TT$ by $b_L$ resembles the preservation of the (intrinsic) hyperbolic metric by the bending map $\H^2 \to \H^3$, and the equivariant property is also analogous.
Moreover, by Theorem \ref{BendingTeichmullerSpaces},  the real bending map $b_L$ is a proper mapping for almost all measured laminations $L$.
In addition, for exceptional laminations, we explicitly characterize the non-properness in the Fenchel-Nielsen coordinates (\Cref{NonpropernessOfBendingTeichmullerSpaces}). 

 Depending on $L \in \ML$, the stabilizer $\GG_L$ can be a large subgroup and, on the other hand, it can be the trivial subgroup of the mapping class group $\MCG$ (\Cref{rStabilizer}).

 We next consider the asymptotic behavior of $b_L\col \TT \to \rchi$.
  Namely, we give an analogue of \Cref{DumasHolonomyLimit} for the real bending map $b_L$.
  Recall that the Thurston boundary of the Teichmüller space is canonically embedded in the Morgan-Shalen boundary (see \cite[\S11.16]{Kapovich-01}). 
 In this paper,    the ``boundary map'' of $b_L$ is the identity for almost all the points.    
\begin{thm}\Label{BoundaryMap}
Let $V \in \PMF = \bdr_{th}\TT$ be a measured foliation such that every singular leaf is a tripod, i.e. a union of three rays with a common endpoint. 
For every $L \in \ML$ and every sequence $\tau_i \in \TT$ converging to $V$, the bent representation $b_L(\tau_i)$ converges to the $\pi_1(S)$-action on the dual metric tree of $\tilde{V}$ as $i \to \infty$.
(Theorem \ref{GenericConvergenceToThurstonBoundary}.)
\end{thm}
Note that a full-measure set of measured foliations satisfies the assumption that every singular leaf is a tripod.

\subsection{Complex bending varieties}\Label{sComplexBendingVarieties}

   Historically, a real-analytic deformation determined by a measured lamination or a measured foliation (an equivalent object) often has a significant complexification: 
 A Teichmüller geodesic in the Teichmüller space $\TT$ is determined by a measured foliation on a Riemann surface, and its complexification is a Teichmüller disk in $\TT$.
A measured lamination on a hyperbolic surface yields a real-analytic earthquake line in $\TT$ (\cite{Thurston86, Kerckhoff85Earthquakes}), and an earthquake disk is its complexification (\cite{McMullen98ComplexEarthquakesTeichmullerTheory}). 
    
We aim to geometrically complexify the real-analytic embedding $b_L \col \TT \to \rchi$ in \Cref{BendingTeichmullerSpaces}, and obtain a complex-analytic mapping from a closed complex-analytic variety.
It is plausible that such complexifications of the real bending varieties $\Im b_L$ in a common analytic space will lead us to discover intersecting properties of the original real-analytic varieties.

   We first explain the domain of the complexified bending map. 
 Given a representation $\rho\col \pi_1(S) \to \PSL_2\C$, if a holonomy $\rho(\ell) \in \PSL_2\C$ along a loop $\ell$ is either hyperbolic or elliptic,  then one can certainly bend $\rho$ along $\ell$ as the axis of $\rho(\ell)$ gives the axis of bending deformation.
 However, it is {\it not} clear at all if one can define bending if $\rho(\ell)$ is parabolic or the identity.

Therefore, given a weighted multiloop $M$ on $S$, we introduce an appropriate closed analytic set $X_M$ consisting of certain (double) framed representations, so that the framing determines the bending axes even when the holonomy along some loops of $M$ is trivial (\S \ref{sFramedCharacterVarieties}). 
In fact, this modification of $\rchi$ essentially occurs only in a complex-analytic subvariety of $\rchi$ disjoint from $\TT$: Namely, when specific subvarieties are removed from $X_M$ and $\rchi$,  the map forgetting the framing induces a finite-to-one holomorphic covering map from $X_M$ to $\rchi$ (see \S \ref{sForgetfulMap}).
In particular, there is a canonical embedding of the Fricke-Teichmüller space $\TT$ into $X_M$ as a real-analytic smooth subvariety. 
In addition, we can pull back the complex symplectic structure on $\rchi$ to $X_M$ minus a subvariety.
 
We next explain the target space. 
Notice that the Fricke-Teichmüller space  $\TT$ is a component of the real slice of the character variety $\rchi$. 
 Moreover, the real bending map $b_L\col \TT \to \rchi$ is in the complex affine variety $\rchi$ (i.e. its tangent spaces contain no complex lines).
Therefore, it is necessary to enlarge the ambient space in order to obtain nontrivial and different complexifications for different bending laminations.

 When the $\PSL_2\C$ Lie algebra $\mathfrak{psl}_2\C$ is regarded as a real Lie algebra, its complexification is isomorphic to $\mathfrak{psl}_2\C \oplus (\mathfrak{psl}_2\C)^\ast$, where $\ast$ denotes complex conjugation.
 Thus,  for a representation $\rho\col \pi_1(S) \to \PSL_2\C$,  we consider the diagonal representation $\pi_1(S) \to \PSL_2\C \times \PSL_2\C$ twisted by conjugation of matrix entries, defined by $\gam \mapsto (\rho(\gam), \rho(\gam)^{\ast})$.
Then, given a representation framed along loops of $M$,  we can appropriately bend it along the axes determined by their framings, where the bending happens in the space of representations into $\PSL_2\C \times \PSL_2\C$. 
Then we obtain the {\sf complex bending map} $B_M \col X_M \to \rchi \times \rchi$. (See \S \ref{sComplexifiedBending} for details.) 
 Let 
     $$\Delta^\ast = \{ (\rho_1, \rho_2) \col \pi_1(S) \to \rchi \times \rchi \mid \rho_1 = \rho_2^\ast\},$$
     the anti-holomorphic diagonal in $\rchi \times \rchi$.
     Define $\psi\col \rchi \to \Delta^\ast \sub \rchi \times \rchi$  by $\rho \mapsto (\rho, \rho^\ast).$
     Let $\omega$ be Goldman's complex symplectic structure on $\rchi$. 
     Then, the average of pull-back complex symplectic structures   $\frac{1}{2} ({\rm pr}_1^\ast \omega + {\rm pr}_2^\ast \omega )$ is a complex symplectic structure on $\rchi \times \rchi$, where ${\rm pr_1}$ and ${\rm pr_2}$ are projections $\rchi \times \rchi \to \rchi$ to the first factor and the second factor, respectively. 
     Then the diagonal embedding $\rchi \to \rchi \times \rchi$ preserves the $\C$-symplectic structure.

   Each hyperbolic surface $\tau \in \TT$ corresponds to a discrete faithful representation $\pi_1(S) \to \PSL_2,\R$ whose image consists of hyperbolic elements except the identity. 
   Choose the orientation of each loop $m$ of $M$ ({\it oriented multiloop}).
   Then, the hyperbolic element $\rho(m)$ has a unique repelling fixed point and attracting fixed point on $\CP^1$, and we have a canonical embedding $\iota_M\col \TT \to X_M$ by adding the information of the fixed points.
  
 \begin{thm}[complexified bending maps along multiloops]\Label{Complexification}
Let $M$ be a weighted oriented multiloop on $S$, i.e. a measured lamination only with periodic leaves. 
Then  $B_M \col X_M \to \rchi \times \rchi$ is a complex-analytic mapping, such that 
\begin{enumerate}
\item the restriction of $B_M$ of $\TT$ is a real-analytic embedding into $\Delta^\ast$;
\item  $\psi \circ b_M \col \TT \to \rchi \times \rchi$ coincides with $B_M \circ \iota_M$ to $\TT$ (\Cref{fMultiloopComplexificationDiagram});
\item $B_M$ is complex symplectic in the complement of a proper subvariety of $X_M$; 
\item $B_M$ is equivariant by the action of the subgroup of the mapping class group preserving $M$. 
\end{enumerate}
\end{thm}
\begin{figure}
\begin{tikzcd}
X_M \arrow{r}{B_M}\arrow{r} & \rchi \times \rchi\\
\TT  \arrow{r}{b_M}  \arrow[u, hook, "\iota_M"
]  & \rchi
 \arrow[u, hook, "\psi"
] 
  \end{tikzcd}
 \caption{The commutative diagram describing the complexification $B_M$ of the real-analytic bending map $b_M$.}\label{fMultiloopComplexificationDiagram}
\end{figure}

(The complex-analyticity is proven in \Cref{BendingIsAnalytic}.
 For (1), see \Cref{ComplexifyingTwistedDiagonal}.  
 For (2), see Proposition \ref{ComplexifyingTwistedDiagonal}.
 For (3), see \Cref{Symp};
 For (4), see \Cref{EquivariantCBending}.) 
The removed subvariety in (3) consists of framed representations such that at least one loop of $M$ has trivial holonomy. 

Moreover, the properness of \Cref{BendingTeichmullerSpaces} is also carried over to complexified bending maps for a dense subset of $\ML$.
\begin{thm}\Label{BendingCharacterVarieties}
If $\ell$ is a non-separating oriented loop with weight not equal to $\pi$ modulo $2\pi$,
then, the bending map 
$B_\ell\col X_{\ell} \to \rchi \times \rchi$ is a proper mapping. (\Cref{ProperCase}.)
\end{thm}

Therefore, under the assumption of \Cref{BendingCharacterVarieties},
the image of  $B_\ell$  is a closed analytic subvariety in $\rchi \times \rchi$ ({\sf complex bending variety}).
Thus, via $\psi$, $\Im b_\ell$ is properly embedded in the real-analytic subvariety of the closed analytic set $\Im B_\ell$,  and $\psi$ preserves the $\R$-symplectic structure of $\Im b_\ell$.

On the other hand,
the complex bending map $B_M$ is {\it not} proper or injective in general (see \Cref{NoninjectiveNonproper}).
 However, $B_M$ is injective and proper ``almost everywhere'':
If an analytic subset is removed from the domain $X_M$ and a subvariety is removed from the target $\rchi \times \rchi$, then $B_M$ becomes injective and proper (\Cref{InjectivityAlomstEverywhere}, \Cref{PropernessAwayFromSubvariety}).

Next, we consider this complexification of a general bending map $b_L$. 
A {\it quasi-Fuchsian representation} $\pi_1(S) \to \PSL_2\C$ is a discrete faithful representation such that the limit set of its image $\Im \rho$ is a Jordan curve in $\CP^1$. 
The set $\QF$ of quasi-Fuchsian representations is called the {\it quasi-Fuchsian space}, and its real slice is the Teichmüller space $\TT$. 
It is straightforward to similarly define the complexified bending map $B_L$  on the quasi-Fuchsian space $\QF$ in $\rchi$, 
 since, for every quasi-Fuchsian representation, every geodesic lamination is realized by the pleating locus of an equivariant pleated surface. 

\begin{thm}\Label{GeneralComplexBendingVarietiy}
For every measured lamination $L$ on $S$, let $\ell_i$ be a sequence of non-separating weighted loops converging to $L$ as $i \to \infty$. 
Then, up to a subsequence,  the closed $\C$-analytic set $\Im B_{\ell_i}$  converges to a closed $\C$-analytic set in $\rchi \times \rchi$ as $i \to \infty$ which is $\C$-symplectic on the smooth part.

Moreover,  the closed $\C$-analytic set $\lim_{i \to \infty} \Im B_{\ell_i}$ contains a unique irreducible connected component $\BB_L$ containing $B_L (\QF)$, such that  $B_L = \psi \circ b_L$ on $\TT$.
(\S \ref{sComplexBendingVariety}.)
\end{thm}

\begin{figure}

\begin{tikzpicture}[
  every node/.style={inner sep=5pt},
  arrow/.style={->},
    hookarrow/.style={{Hooks[right]}->},
]

\node (A) at (0, 0) {$\QF$};
\node (B) at (2, 0) {$\BB_L$};
\node (E) at (3.1, 0) {$\sub \rchi \times \rchi$};
\node (C) at (0, -1.5) {$\TT$};
\node (D) at (2, -1.5) {$\rchi$};

\draw[arrow] (A) -- (B) node[midway, above] {\small$B_L$};
\draw[hookarrow ] (C) -- (A) node[midway, left] {};
\draw[hookarrow] (D) -- (B) node[midway, right] {\small $\psi$};
\draw[arrow] (C) -- (D) node[midway, above] {\small $b_L$};

\end{tikzpicture}

 \caption{The commutative diagram describing the complexification of $\Im b_L$.}\label{fComplexificationDiagram}
\end{figure}

\subsection{Outline of the paper}
In \S \ref{sPreliminaries}, we explain some basic notions used in this paper.  
In particular, we recall that a measured lamination on a hyperbolic surface induces an equivariant locally convex pleated surface $\H^2 \to \H^3$, then we define the real bending map $b_L \col \TT \to \rchi$ for a measured lamination.
In \S \ref{sRealBendingMapIsInjective}, we show the injectivity of the real bending map. 
In \S \ref{sRealBendingMapIsProper}, we prove the properness of the real bending map for most of the measured laminations $L$.
On the other hand, in Theorem  \ref{NonpropernessOfBendingTeichmullerSpaces},  we characterize the non-properness of the bending map.
In \S \ref{GenericConvergenceToThurstonBoundary},  we prove \Cref{BoundaryMap}.

In \S \ref{sFramedCharacterVarieties}, we introduce the space of representations double-framed along a weighted multiloop $M$  on $S$ (the framed character variety $X_M$). 
Then, in \S \ref{sComplexifiedBending}, we define the complex bending map from the framed character variety $\rchi_M$ to the product character variety $\rchi \times \rchi$. 
For the definition, a more general type of bending deformation is introduced.
In fact, when a representation framed along $M$ is bent along $M$, accordingly, the hyperbolic space $\H^3$ is equivariantly ``bent'' inside the $\H^3 \times \H^3$ (\S\ref{sSupportSpaces}).  
In \S \ref{sInjectivity}, we show that the complex bending map is injective almost everywhere.
In \S \ref{sAlmostProperComplexBendingMap}, we show that the complex bending map is proper almost everywhere. 
 In \S \ref{sAnalyticity},  using the ``almost-everywhere" injectivity, we prove the analyticity of the complex bending map on the entire domain. 
In \S \ref{sCoplexification}, we show that the complex bending map is a complexification of the real bending map. 
In \S \ref{sPropernessForNonseparatingLoop}, we prove that the complex bending map is, indeed, genuinely a proper mapping when $M$ is a single non-separating loop of the weight  {\it not} equal to $\pi$. 
In \S \ref{sSymplectic}, we show that the real bending map is symplectic and the complex bending map is complex symplectic. 
In \S \ref{sComplexBendingVariety}, we prove \Cref{GeneralComplexBendingVarietiy}.

\section{Acknowledgements}
I first thank Misha Kapovich, Bill Goldman, Hisashi Kasuya, and James Farre for their helpful conversations. 
I deeply appreciate the anonymous referees for their valuable comments, which significantly improved this paper.  In particular, Theorem B and Theorem E are proven in response to the referees' inspiring comments.
I thank Kento Sakai for telling me about harmonic maps, and lastly thank Tengren Zhang for his comments, with which I could simplify the definition of complex bending maps. 
The author was supported by the Grant-in-Aid for Scientific Research 20K03610 and 24K06737.

\section{Preliminaries}\Label{sPreliminaries}
\subsection{Bending deformation}\Label{sBendingDeformation}  (\cite{Thurston-78}, \cite{Epstein-Marden-87}.)
Thurston discovered that the boundary of the convex core of a hyperbolic three-manifold is a hyperbolic surface bent along a measured lamination (\cite{Thurston-78}). 
More generally, one can bend a hyperbolic surface along an arbitrary measured lamination and obtain a holonomy representation from the surface fundamental group into $\PSL_2\C$ as follows. 
 
 We shall first describe basic bending maps when the bending locus is a single loop.
 Let $\tau$ be a hyperbolic structure on $S$, and let $\ell$ be a geodesic loop on $\tau$ with weight $w > 0$. 
 The union $\ti{\ell}$ of all lifts of $\ell$ to the universal cover $\H^2$ of $\tau$ is a set of disjoint geodesics, each with weight $w$, and it is invariant under the deck transformation. 
We call the union $\ti\ell$ the {\sf total lift} of $\ell$. 

Put the universal cover $\H^2$ in the three-dimensional hyperbolic space $\H^3$ as a totally geodesic hyperbolic plane. 
By this embedding, the isometric deck transformations of $\H^2$  extend to an isometric action on  $\H^3$, and we obtain a representation of $\rho_\tau\col \pi_1(S) \to \PSL_2 \C$.
Note that, as $S$ is oriented, the orientation of the universal cover $\H^2$ determines a normal direction of the plane. 
Thus we can bend $\H^2$ along every geodesic $\alpha$ of $\ti{\ell}$ by angle $w$ so that the normal direction is in the exterior. 
Thus we obtain a bending map $\beta\col \H^2 \to \H^3$, which is totally geodesic on every complement of $\H^2 \minus \ti{\ell}$. 
The map $\beta$ is unique up to an orientation-preserving isometry of $\H^3$. 
Moreover, $\beta$ is equivariant by its holonomy representation $\rho\col \pi_1(S) \to \PSL_2\C$. 
 This $\rho$ is called the bending deformation of $\rho_\tau$ of $L$. 
 
 If $C_1, C_2$ are components of $\H^2 \minus \ti\ell$ such that $C_1, C_2$ are adjacent along a geodesic $\alpha$ of $\ti\ell$.
 Let $G_1$ and $G_2$ be the subgroups of $\pi_1(S)$ which preserve $C_1$ and $C_2$, respectively. 
 If $\beta$ is normalized so that $\beta_\tau = \beta$ on $C_1$, then the restriction of $\beta$ to $G_2$ is the conjugation of the restriction of $\rho_\tau$ to $G_2$ by the elliptic isometry with the axis $\alpha$ by angle $w$. 

More generally, given an arbitrary measured lamination $L$ on $\tau$, we can take a sequence of weighted loops $\ell_i$ converging to $L$ as $i \to \infty$. 
For each $i$, let $\rho_i\col \pi_1(S) \to \PSL_2\C$ be the bending deformation of $\rho_\tau$ along $\ell_i$. 
Then $\rho_i$ converges to a representation $\pi_1(S) \to \PSL_2\C$ as $i \to \infty$ if $\rho_i$ are appropriately normalized by $\PSL_2\C$.
This limit is the bending deformation of $\rho_\tau$ along $L$, and it is unique up to conjugation by an element of $\PSL_2\C$.

   \subsubsection{Equivariant property of the real bending map} \Label{sRealBendingMap}

The equivariant property of $b_L\col \pi_1(S) \to \PSL_2\C$ in \Cref{BendingTeichmullerSpaces} can directly be proven from the definition of the bending map. 
Here, we show this property in a broader context. 

 A {\sf $\CP^1$-structure} on $S$ is a $(\CP^1, \PSL_2\C)$-structure.
That is, an atlas of charts mapping open subsets of $S$ into $\CP^1$ with translation maps in $Aut(\CP^1) = \PSL_2\C$.  (General references about $\CP^1$-structures are\cite{Dumas-08, Kapovich-01,  Goldman22GeometricStructuresOnManifolds}).
Recall that $\CP^1$ is the ideal boundary of the hyperbolic space $\H^3$, and $\PSL_2\C$ is the group of orientation-preserving isometries of $\H^3$. 
Using equivariant bending maps described above, 
Thurston gave a parametrization of the deformation space $\PP$ of $\CP^1$-structures by corresponding them with holonomy-equivariant pleated surfaces in $\H^3$. 
\begin{theorem}[Thurston, \cite{Kulkani-Pinkall-94, Kamishima-Tan-92}]\Label{ThurstonParametrization}
The following canonical identification holds by  a  (tangential) homeomorphism, 
$$\PP  = \TT \times \ML.$$
\end{theorem}
Then $b_L(\tau) = \Hol (\tau, L)$ where $(\tau, L) \in \TT \times  \ML$ denote the $\CP^1$-structure in Thurston coordinates. 

\begin{lemma}\Label{RealBendingEquivariant}
For $L \in \ML$, let $\GG_L$ be the subgroup of $\MCG$ which preserves $L$. 
Then, the real bending map
$b_L\col \TT \to \rchi$ is $\GG_L$-equivariant. 
\end{lemma}
\begin{remark}\Label{rStabilizer}
If $L$ is a multiloop, then $\GG_L$ contains the subgroup of $\MCG$ generated by Dehn twists along loops not intersecting $L$ (but including the loops of $L$).
On the other hand, for almost all $L$ in $\ML$, $\GG_L$ is the trivial group, since $\MCG$ is a countable group. 
\end{remark}
\begin{proof}
The  $\MCG$-action on $\PP$ is given by marking change and on $\rchi $ by precomposing induces isomorphisms $\pi_1(S) \to \pi_1(S)$. 
Then the holonomy map $\Hol \col \PP \to \rchi$ is ${\rm MCG}$-equivariant (see, for example, \cite{Goldman06}). 

By Thurston's parametrization, 
for all $\tau \in \TT$ and $h \in \MCG$,  $h(\tau, L) = (\tau, L)$.
$$h \cdot b_L (\tau) =  h \cdot \Hol (\tau, L) = \Hol (h, L) = b_L (h \tau).$$
Thus the desired equivariant property holds. 
\end{proof}

\subsection{Quasi-geodesics in the hyperbolic space} 
We first recall the definition of quasi-isometries.
Let $(X, d_X), (Y, d_Y)$ be metric spaces, where $d_X, d_Y$ are the distance functions. 
Then, for $P > 1, Q > 0$,  a mapping $f\col X \to Y$ is a $(P, Q)$-{\it quasi-isometry} if, 
for all $x_1, x_2 \in X$, $$P^{-1} d_X(x_1, x_2) - Q <  d_Y(f(x_1), f(x_2)) < P\, d_X(x_1, x_2) + Q.$$ 

In this section, we discuss some conditions for a piecewise geodesic curve in $\H^3$ to be a quasi-geodesic.

\subsubsection{Quasi-geodesics in $\H^3$}
Let $c$ be a bi-infinite piecewise geodesic curve in $\H^3$. 
Let $s_i \,(i \in \Z)$ be the geodesic segments of $c$ indexed along $c$, so that $s_i$ and $s_{i +1}$ are adjacent geodesic segments for every $i \in \Z$ and $c = \cup_{i \in \Z} s_i$.
\begin{lemma}\Label{AlmostGeodesicPieacewiseGeodesic}
For every $\ep > 0$, there are (large) $R > 0$ and (small) $\del > 0$, such that if $\length s_i > R$ for all $i \in \Z$ and the (interior) angle between arbitrary adjacent geodesic segment $s_i, s_{i+ 1}$  is at least $\pi - \del$, then $c$ is a $(1 + \ep)$-biLipschitz embedding. 
\end{lemma}
\begin{proof}
This lemma follows from 
\cite[I.4.2.10]{CanaryEpsteinGreen84} and \cite[Proof of Theorem 4.2]{Epstein-Mardern-Markovic}.
\end{proof}

\begin{proposition}\Label{PiecewiseGeodesicCurve}
For all $\ep > 0$ and $\ep' \in (0, \pi)$, there are $R > 0$ and $Q > 0$, such that if $\length s_i > R$ for all $i \in \Z$ and the angle between every pair of adjacent geodesic segments is at least $\pi - \ep'$, then $c$ is a $(1 + \ep, Q)$-quasi-isometric embedding. 
\end{proposition}

\proof
For each $i \in \Z$, let $x_i$ be the common endpoint of $s_{i -1}$ and $s_i$, so that $x_i$ is a non-smooth point of $c$.  
Pick $r > 0$ and we assume that $R > 2 r$. 
Let $x_i^-$ be the point on $s_{i - 1}$ such that $d(x_i^-, x_i) = r$. 
Let $x_i^+$ be the point on $s_i$ such that $d(x_i, x_i^+) = r$. 
Then, we replace  two geodesic segments $[x_i^-, x_i] \cup [x_i, x_i^+]$ of  $c$ with the single geodesic segment $[x_i^-, x_i^+]$; see \Cref{fShortCut}.
Let $c_r$ be the piecewise geodesic in $\H^3$ obtained from $c$ by applying this replacement for every $i \in \Z$. 

By basic hyperbolic geometry, the following holds. 
\begin{lemma}\Label{LargeTriangles}
For every $\del > 0$, there is $r_\del > 0$ satisfying the following: For every  $r > r_\del$ and every  $R > 3 r$, then the angle at every non-smooth point of $c_r$ is more than $\pi - \del$. 
\end{lemma}

\begin{figure}
\begin{overpic}[scale=.03
] {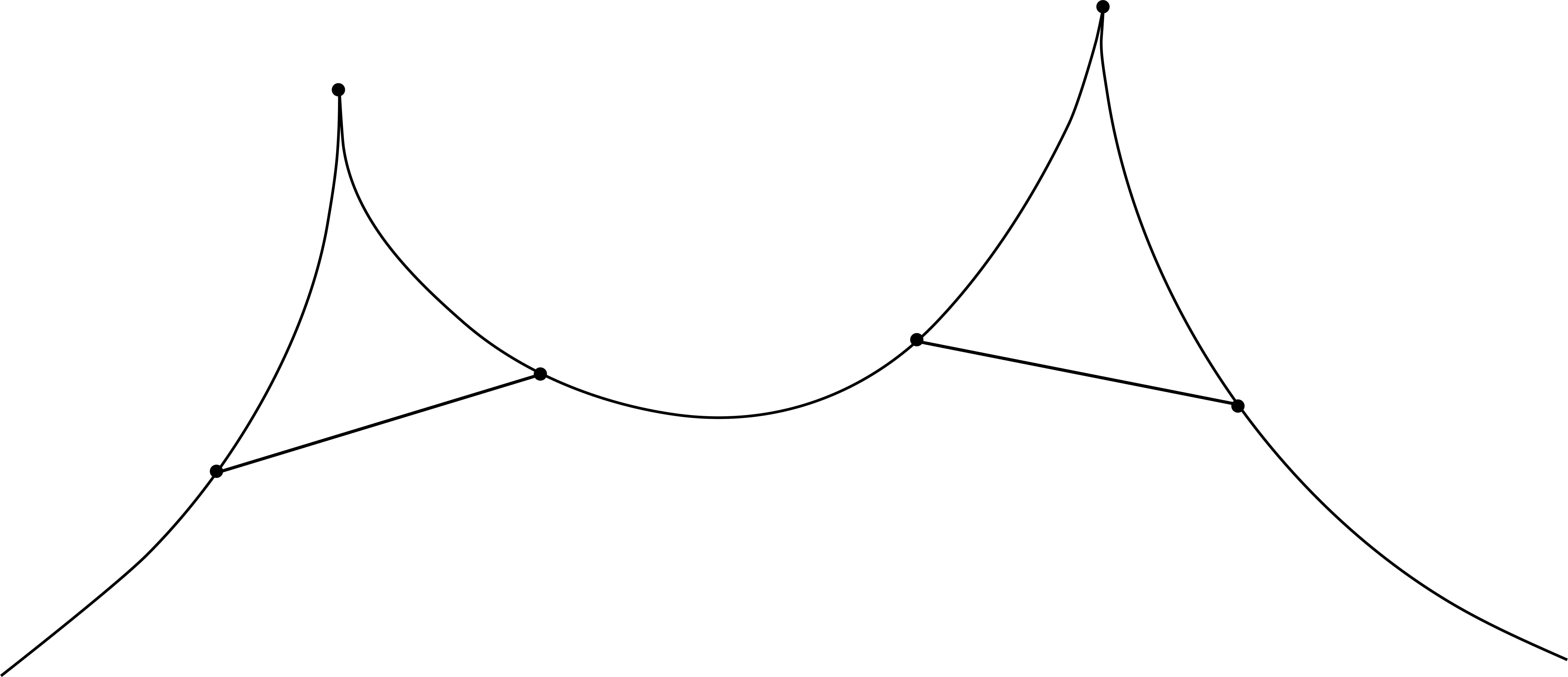} 
   \put(20 ,39 ){$x_i$}  
   \put(8 , 14 ){$x_i^-$}
     \put( 34 , 21 ){$x_i^+$}  
     \put(70 , 44 ){$x_{i + 1}$}  
      \put( 48 , 23 ){$x_{i+1}^-$} 
            \put( 80 , 19 ){$x_{i+1}^+$} 
   \put( 41 , 12 ){$s_i$} 
   \put(6 , 2 ){$s_{i -1}$} 
   \put(80 , 5 ){$s_{i +1}$}
      \end{overpic}
\caption{A modification of a piecewise geodesic}\Label{fShortCut}
\end{figure}

Then \Cref{AlmostGeodesicPieacewiseGeodesic} and \Cref{LargeTriangles} imply the proposition.
\Qed{PiecewiseGeodesicCurve}

 \subsection{Angles between geodesic laminations}\Label{sAngleBetweenLaminations}(See \cite{Baba-15gt}.)
 
 Let $\tau$ be a hyperbolic surface. 
 If  two geodesics $\ell_1, \ell_2$ on $\tau$ intersect in a point $p$, then let $\angle_p(\ell_1, \ell_2)$ denote the angle between them  which takes a value in $[0, \pi/2]$.
 More generally, let $\lam_1$ and $\lam_2$ be geodesic laminations on $\tau$.  
 Then 
 the angle $\angle_{\tau} (\lam_1, \lam_2) \in  [0, \pi/2]$ is the supremum of $\angle_p(\ell_1, \ell_2)$ over all leaves $\ell_1 \in \lam_1$ and $\ell_2 \in \lam_2$ intersecting a point $p$. 
 
     
\subsection{The Morgan-Shalen compactification} (See \cite{Dumas18HolonomyLimit} \cite{Kapovich-01}.) 
We describe a compactification of the $\PSL(2, \C)$-character variety of $S$. 
Let
$$\rchi \to \R_{ > 0}^{\pi_1(S)} / \R_{>0}$$ be the mapping defined by 
$$\rho \mapsto \log (\vert \tr \rho(\gam) \vert  + 2)_{\gam \in \pi_1(S)}.$$
Then the image is relatively compact in the infinite-dimensional projective space $ \R_{ > 0}^{\pi_1(S)} / \R_{>0}$, and the compactification in the projective space is called the Morgan-Shalen compactification of $\rchi$.

A boundary point $p$ of the Morgan-Shalen compactification corresponds to a  minimal small $\pi_1(S)$-action on a metric tree $T_p$.
Namely, if a sequence $\rho_i \col \pi_1(S) \to \PSL_2\C \in \rchi$ converges to $p$, then there is a sequence $r_i >0$ converging to $0$ such that 
\begin{itemize}
\item there is a sequence of points $x_i \in \H^3$ such that $\pi_1(S)$-action on $\H^3$ with the base point $x_i$ converges to the $\pi_1(S)$-action on $T_p$ when the metric of $\H^3$ is scaled by $r_i$, and
\item  the translation lengths of $\rho(\gam)$ for $\gam \in \pi_1(S)$ scaled by $\gam_i$ converge to the translation lengths on $T_p$ by the $\pi_1(S)$-action.
\end{itemize} 

A Teichm\"uller space $\TT$ can be regarded as the space of marked hyperbolic structures on $S$. 
Thus $\TT$ can be regarded as a component of the space of discrete faithful representations $\pi_1(S) \to \PSL_2 \R$ up to conjugation.  
The Thurston compactification of the Teichmüller space is the compactification given by the projective length spectra of the translation lengths (\cite{FLP}). 
Thus, via Bonahon's interpretation via geodesic currents, the Thurston boundary can be regarded as a part of the Morgan-Shalen boundary, by embedding $\TT$ into $\rchi$ as a component of a real slice. (See \cite[\S 11.16]{Kapovich-01}.)

\subsection{Complex analytic geometry}
(\cite{Goldman84SymplecticNatureOfFundamentalGroupsOfSurfaces}.)
We recall a standard theorem about a complex-analytic set.

\begin{theorem}[Removable Singularity Theorem; see for example \cite{Taylor02SeveralComplexVariables}, \S3.3.2]\Label{RemovableSingularity}
Let $Y$ be an analytic set. 
Let $A$ be a closed subset of $Y$ contained in a proper subvariety of $Y$. 
Suppose that  $f\col Y \minus A \to \C$ is an analytic function which is bounded in a small neighborhood of every point in $A$. 
Then $f$ continuously extends to an analytic function on  $Y$. 
\end{theorem}

\subsection{Goldman's symplectic form}
(\cite{Goldman84SymplecticNatureOfFundamentalGroupsOfSurfaces}.)
Let $\mathfrak{g}$ be the $\PSL_2\C$-Lie algebra.  
Then the adjoint representation $\Ad\col  \PSL_2\C \to Aut \mathfrak{g} \sub {\rm GL}_3\C$ is induced by the conjugation of $\PSL_2\C$ by $\PSL_2\C$.
By $\mathfrak{g}_{\Ad \rho}$, we regard $\mathfrak{g}$ as a $\pi_1(S)$-module via the composition of $\rho\col \pi_1(S) \to \PSL_2\C$.
Then the Zariski tangent space of the representation variety $\RR$ at $\rho \in \RR$ is the vector space of  1-cocycles
$$Z^1(\pi_1(S); \mathfrak{g}_{Ad \rho}) = \{ u \in \mathfrak{g}^{\pi_1(S)}  \mid u(xy) = u(x) + ({\rm Ad} \rho(x)) \, u(y)\}.$$
The subspace of 1-coboundaries
$B^1(\pi_1(S); \mathfrak{g}_{Ad \rho})$
consists of $u \in \mathfrak{g}^{\pi_1(S)}$, such that there is $u_0 \in \mathfrak{g}$  satisfying $u(x) =  u_0 - \Ad(\rho(x)) u_0$ for all $x \in \pi_1(S)\}$.
Then the Zariski tangent space of $\rchi$ at $\rho$ is the quotient vector space $$H^1(\pi_1(S); \mathfrak{g}_{Ad \rho}) = \frac{Z^1(\pi_1(S); \mathfrak{g}_{Ad \rho})}{B^1(\pi_1(S); \mathfrak{g}_{Ad \rho})}.$$   
  Let $w(\rho)$ denote the bilinear form on the Zariski tangent space obtained by the composition 
  \begin{eqnarray*}
  H^1(\pi_1(S); \mathfrak{g}_{Ad \rho}) \times H^1(\pi_1(S); \mathfrak{g}_{Ad \rho}) &\xrightarrow{\cup}& H^2(\pi_1(S); \mathfrak{g}_{Ad \rho} \otimes \mathfrak{g}_{Ad \rho}) \\
  &{\overset{ \cong} \rightarrow}& H^2(\pi_1(S); \C)  \cong \C.
\end{eqnarray*} 
Here the first mapping is the cup product, and the second mapping is an isomorphism given by the coefficients pairing by the bilinear form $\mathfrak{B}\col  \mathfrak{g}_{Ad \rho} \otimes \mathfrak{g}_{Ad \rho} \to \C$  given by $(A, B) \to \tr AB$.
 Goldman proved that $w$ is a complex symplectic form on $\rchi$, i.e. a non-degenerate closed holomorphic $(2, 0)$-form on the character variety $\rchi$; see \cite{Goldman84SymplecticNatureOfFundamentalGroupsOfSurfaces}.

 \subsection{Harmonic maps between hyperbolic surfaces}\Label{sHarmonicMaps}
(See (\cite{Wolf91, Minsky92}; see also  \cite{Sakai(25)}.)
 
 For (marked) Riemann surfaces $X, Y \in \TT$, there is a unique harmonic map $h\col X \to Y$ preserving the marking. 
Then the Hopf differential of the harmonic map $h$ is a holomorphic quadratic differential $q$ on $X$.
Away from the zeros of $q$, the differential $q$ gives natural coordinates $w = x + i y$ in $\C$ so that $q = dw$ (see, for example, \cite{FarbMarglit12}).
By these coordinates, the Euclidean structure on $\C$ induces a singular Euclidean metric on $X$ where the zeros of $q$ are the singular points, and the Euclidean structure realizes the conformal structure of $X$.

The {\it Beltrami differential} of $h$ is given by 
 $$\nu_h = \frac{f_{\bar{z}} d\bar{z}}{f_z dz}. $$ 
Then $\vert \nu_h(z) \vert < 1$.
Then as $Y$ leaves every compact set in $\TT$ while $X$ is fixed, $\vert \nu_h(z) \vert$  converges to 1.
Let 
$$G(h) = \log \bigg(\frac{1}{|\nu(h)|}\bigg).$$

Let $g_Y$ denote the hyperbolic metric on $Y$, and let $g = h^\ast(g_Y) $ be the pull-back metric on $X$ by $h$.
 Then $g$ is, in natural coordinates $x + i y$ given by $q$,
 \begin{eqnarray}
 d s^2 = \frac{\cosh G(t) + 1}{2} d x^2 +   \frac{\cosh G(t) - 1}{2} d y^2  \label{ePullBackMetric}
\end{eqnarray}

The $L^1$-norm
$\| q \|   = \int_X | \phi_i | dz d\bar{z}$ is the total area w.r.t. the flat metric. 
If  the $r$-ball centered at $p \in X$ contains no zeros in the flat metric, then 
 \begin{eqnarray}
 G(h)(p) \leq \sinh^{-1} \frac{\Area X}{ 2 \pi r^2},  \Label{iEstimate}
 \end{eqnarray}
 where $\Area X$ denotes the hyperbolic area $2\pi(2g-2)$ of $X$ (Minsky \cite[Lemma 3.2]{Minsky92}).

Therefore,  if $Y$ leaves every compact subset in $\TT$ while $X$ is fixed, the hyperbolic metric is stretched in the horizontal direction and shrinks in the vertical direction away from the zeros.

More specifically, we let $(Y_i)_{i = 1}^\infty$ be a sequence in $\TT$ converging to a point $[V] \in \PML = \bdr \TT$ in the Thurston boundary,  where $\PML$ denotes the space of projective measured foliations on $S$. 
Let $h_i\col X \to Y_i$ be the unique harmonic map, and let $q_i$ be the holomorphic quadratic differential on $X$ given by the Hopf differential of $h_i$.
Let $V_i$ be the vertical measured foliation of $q_i$, and let $H_i$ be the horizontal measured foliation of $q_i = \phi_i dz^2$.
Then its projective class $[V_i]$ converges to $[V]$ as $i \to \infty$ (\cite{Wolf91}).

The total Euclidean area $\| q_i\| = \int_X | \phi_i | dz d\bar{z} $ diverges to infinity as $i \to \infty$, and by (\ref{ePullBackMetric}) and (\ref{iEstimate}),
$h_i$ stretches $X$ in the horizontal direction $H_i$ so that the Euclidean length and the hyperbolic length are close away from the zeros, and shrinks in the vertical direction $V_i$ more and more.

\section{Injectivity of the real bending maps}\Label{sRealBendingMapIsInjective}

Let $\ML$ be the space of measured laminations on $S$. 
Each pair $(\tau, L) \in \TT \times \ML$ induces an equivariant pleated surface $\H^2 \to \H^3$, unique up to $\PSL_2\C$. 
Let $b \col \TT \times \ML \to \rchi$ be the holonomy map of the bending maps.
\begin{theorem}\label{BendingTeich}
Fix arbitrary $L \in \ML(S)$.
Then the restriction $b$ to $\TT \times \{L\}$ is a real-analytic embedding.  
Moreover, this embedding is proper if and only if $L$ contains no periodic leaf of weight $\pi$ modulo $2\pi$. 
\end{theorem}
\Cref{BendingTeich} is implied by \Cref{RealAnalyticEmbedding} and \Cref{RealPropermess}.

Let $b_L\col \TT \to \rchi$ denote the restriction of $b$ to $\TT \times \{L\}$.
 Given a representation $\rho\col \pi_1(S) \to \PSL_2\C$, geodesic lamination $\lambda$ on $S$ is {\it realizable} if there is a $\rho$-equivariant pleated surface $\H^2 \to \H^3$, such that its pleating loci contains $\lambda$.
Then, for $L \in \ML$,
let $N = N_L$ be an open neighborhood of the Fuchsian space $\TT$ in the smooth part of $\rchi$ such that the underlying geodesic lamination $|L|$ is {realizable} for all $\rho \in \rchi$.
Then, $b_L\col \TT \to \rchi$ extends to the bending map $\hat{b}_L\col N_L \to \rchi$ by bending cocycle (\cite{Bonahon96ShearingBendingSymplecticForm}). 
\begin{proposition}\Label{Injectivity}
For all $L \in \ML$, 
$\hat{b}_L\col N_L \to \rchi$ is injective. 
\end{proposition}
\begin{proof}
As $|L|$ is realizable on $\Im \hat{b}_L$, we have the unbending map $\hat{b}_{-L}\col \Im \hat{b}_L \to \rchi$ by $-L$ . 
Then, clearly, $\hat{b}_{-L} \circ\hat{b}_L$ is the identity map on $N_L$. 
Thus $\hat{b}_L$ is injective. 
\end{proof}

\begin{proposition}\Label{RealAnalyticEmbedding}
The injective map $b_L\col \TT \to \rchi$ is a real-analytic embedding. 
\end{proposition}
\begin{proof}(cf. \cite{Kerckhoff85Earthquakes}.)
We regard $\TT$ as the Fricke space, i.e. the space of discrete faithful representations into $\PSL(2, \R)$ up to conjugation by $\PSL_2\R$.
Then, take a small open neighborhood $N$ of $\TT$ whose closure is contained in $N_L$. 

If $L$ is a weighted multiloop,  the bending map $b_L$ is holomorphic on $N$ as bending transforms the holonomy along a loop by some elliptic elements in a holomorphic manner. 
In general, pick a sequence of weighted multiloops $M_i$ converging to $L$ as $i \to \infty$. 
By the injectivity of \Cref{Injectivity}, $\hat{b}_{M_i}\col N_{M_i} \to \chi$ is a holomorphic embedding. 
Then, the holomorphic embedding $\hat{b}_{M_i} \vert N$ converges uniformly to $b_L | N$ uniformly on compacts as $i \to \infty$. 
Therefore $\hat{b}_L \vert N$ is a holomorphic embedding. 

Since $\TT$ is a real-analytic submanifold of $N$ in $\rchi$, thus $b_L \vert \TT$ is a real-analytic embedding. 
\end{proof}

\section{Properness of the bending maps from the Teichmüller spaces}\Label{sRealBendingMapIsProper}
In this section, we prove the properness stated in \Cref{BendingTeich}.
\begin{theorem}\Label{RealPropermess}
Let $L \in \ML$. 
Then, the bending map $b_L\col \TT \to \rchi$ is proper if and only if  $L$ contains no leaf of weight $\pi$ modulo $2\pi$. 
\end{theorem}

First, we prove the sufficiency of the condition in Theorem \ref{RealPropermess}. 

\begin{lemma}\Label{LengthCoordinatesForL}
Fix $L \in \ML$ such that every closed leaf of $L$ contains no leaf of weight $\pi$ modulo $2\pi$. 
Let $M$ be the (possibly empty) sublamination of $L$ consisting of the periodic leaves of $L$. 
Then, for all $\upsilon > 0$
there are finitely many loops $\ell_1, \dots, \ell_n$ on $S$ such that
\begin{itemize}
\item the lengths of $\ell_1, \dots, \ell_n$  form  length parameters of $\TT$, and 
\item for each $i = 1, \dots, n$, 
\begin{itemize}
\item the transversal measure $(L \minus M) (\ell_i) < \upsilon$,  and 
\item $\ell_i$ intersects at most one leaf of  $M$, and their intersection number is at most two. 
\end{itemize}
\end{itemize}
\end{lemma}
\begin{proof}

For every $\del > 0$, 
there is a {\sf pants  decomposition} $P = P_\del$ (i.e. a maximal multiloop) on $S$ consisting of 
\begin{itemize}
\item the loops of $M$,  
\item loops which are disjoint from $L$, 
\item loops $\ell$ with  $L(\ell) < \del$ (so that $\ell$ is a good approximation of  a minimal irrational sublamination of $L$). 
\end{itemize} 
By the third condition, if $R$ is a component of $S \minus P$, and $\alpha$ is an arc on $R$ with endpoints on $\partial R$, then there is an isotopy of $\alpha$ keeping its endpoints on $\partial R$ such that $L(\alpha) < 3 \del$. 
Therefore, if $\del > 0$ is small enough, for each loop $m$ of $P$, we can take two loops $m_1, m_2$ such that 
\begin{itemize}
\item $m_i$ intersects $m$ at a point or two, and it does not intersect any other loop of $P$, and
\item $(L \minus M) (m_i) < \upsilon$. 
\end{itemize}
Then we obtain a desired set of loops by adding two such loops for all loops of $M$.
(For length coordinates of $\TT$, see \cite[Theorem 10.7]{FarbMarglit12}  for example.)\end{proof}
 
\proof[Proof of the sufficiency of  \Cref{RealPropermess}]
For $\ep > 0$, let $\ell_1, \dots, \ell_n$ be the set of loops given by \Cref{LengthCoordinatesForL}.
Let $\tau_i$ be a sequence in $\TT$ which leaves every compact subset. 
Then,  for some $1 \leq k \leq n$, $\length_{\tau_i} \ell_k \to \infty$ as $i \to \infty$ up to a subsequence. 

\begin{claim}
For every $\ep > 0$, if $\del > 0$ is sufficiently small, then  
\begin{enumerate}
\item if $L(\ell_k) < \del$, then $\beta_i | \ti{\ell}_k$ is a $( 1 + \ep)$-biLipschitz embedding for sufficiently large $i$, and  \Label{iLoopsWithSmallTransversalMeasure}
\item if $\ell_k$ intersects a loop $m$ of $M$, then $\beta_i | \ti\ell_k$ is a $(1 + \ep, Q)$-quasi-isometric embedding for all sufficiently large $i$, where $Q$ only depends on the weight of $m$.   \Label{iLoopsInversctingPeriodicLeaves}
\end{enumerate}
\end{claim}
\begin{proof}

(\ref{iLoopsWithSmallTransversalMeasure})
See
\cite[Lemma 5.3]{Baba-10}, which was proved based on \cite[I.4.2.10]{CanaryEpsteinGreen84}.

(\ref{iLoopsInversctingPeriodicLeaves})
   We straighten $\ell_k$ and $M$ on $\tau_i \in \TT$. 
From \Cref{LengthCoordinatesForL}, $\ell$ intersects only one loop $m$ of $M$, and their (geometric) intersection number is one or two. 
  We thus assume that $\ell_k \cap m$ consists of two points $x_1, x_2$ --- the proof when the intersection number is one is similar. 
Then $x_1$ and $x_2$ decompose $\ell_k$ into 2 geodesic segments $a_1$ and $a_2$. 
Since $\length_{\tau_i} \ell_k \to \infty$,  the lengths of $a_1$ and $a_2$ both go to $\infty$ as well. 
Let $\ti\ell_k$ be the geodesic in $\H^2$ obtained by lifting $\ell_k$ to the universal cover. 
Let $\ti{a}_j$ be a lift of $a_j$ to $\ti\ell_k$, and let $\ti{x}_j$ and $\ti{x}_{j + 1}$ be its endpoints. 
For every $\ep' > 0$, if $\upsilon > 0$, is sufficiently small, then $\beta_i(\ti{a}_j)$ is $\ep'$-close to the geodesic segment $[\beta_i x_j, \beta_i x_{j + 1}]$ connecting its endpoints $\beta_i x_j$ and $\beta_i x_{j + 1}$ in the Hausdorff metric. 
Since every periodic leaf of $L$ has weight not equal to $\pi$ modulo $2\pi$, there is $\omega > 0$ such that, for every periodic leaf $\ell$ of $L$, the distance from the weight of $\ell$ to the nearest odd multiple of $\pi$ is at least $\omega$. 
Therefore, if $\del  > 0$ is sufficiently small, then 
the angle between $[\beta_i x_j, \beta_i x_{j + 1}]$ and $[\beta_i x_{j -1}, \beta_i x_j]$ at $x_j$ is  at least $\omega/2$.
Let $c_i$ be the piecewise geodesic in $\H^3$ which is a union of the geodesic segments  $[\beta_i x_j, \beta_i x_{j + 1}]$ over all lifts $\ti{a}_1, \ti{a}_2$ of $a_1, a_2$ to  $\ti\ell_k$.
Then $c_i$ is $\ep'$-Hausdorff close to $\beta_i \ti\ell_k$. 
Therefore, by \Cref{PiecewiseGeodesicCurve}, we see that $c_i$ is a $(1+ \ep, Q)$-quasi-geodesic. 
\end{proof}
By this claim, for large $i$, the holonomy of $b_L (\tau_i)$ along $\ell_k$ is hyperbolic, and its translation length diverges to $\infty$ as $i \to \infty$. 
Thus $b_L(\tau_i)$ leaves every compact set in $\rchi$. 
Thus we have proven the properness. 
\Qed{RealPropermess}


\section{Characterization of non-properness}
In this section, we explicitly describe how $b_L\col \TT \to \rchi$ is non-proper when the condition in \Cref{RealPropermess} fails.
Let $L$ be a measured lamination on $S$. 
Let $m_1, \dots, m_p$ be the periodic leaves of $L$ which have weight $\pi$ modulo $2\pi$.
Then, set $M = m_1 \sqcup \dots \sqcup m_p$.
Pick any  pants decomposition $P$ of $S$ which contains  $m_1, \dots, m_p$. 
Consider the Fenchel-Nielsen coordinates of $\TT$ associated with $P$. 
Recall that its length coordinates take values in $\R_{>0}$ and its twist coordinates in $\R$.

\begin{theorem}\Label{NonpropernessOfBendingTeichmullerSpaces}
Let $\tau_i$ be a sequence in  $\TT$ which leaves every compact subset. 
Then
$b_L(\tau_i)$  converges in $\rchi$ if and only if
\begin{itemize}
\item  $\length_{\tau_i} m_j \to 0$ for some $j \in \{1, \dots, p\}$ as $i \to \infty$ ({\it pinched}), and
\item the Fenchel-Nielsen coordinates of $\tau_i$  w.r.t. $P$ converge in their parameter spaces as $i \to \infty$, except that the length parameters of the pinched loops go to zero. 
\end{itemize}
\end{theorem}

\proof[Proof of Theorem \ref{NonpropernessOfBendingTeichmullerSpaces}]

Let $F$ be a component of $S \minus M$.
Then $b_L (\tau_i) | F$ converges in $\rchi(F)$ if and only if $\tau_i | F \coloneqq \tau_i \vert \pi_1(F)$ converges. 

Let $E$ and $F$ be adjacent components of $\ti{S} \minus \ti{M}$. 
Let $\ti{m}$ be the component of $\ti{M}$ separating $E$ and $F$, and let $m$ be the loop of $M$ which lifts to $\ti{m}$. 
Let $\Gamma_E$ and $\Gamma_F$ be the subgroups of $\pi_1(S)$ preserving $E$ and $F$, respectively.
Then $E/\Gam_E$ and $F / \Gam_F$ are the components of $S\minus M$; let  $S_E = E/\Gam_E$ and $S_F = F / \Gam_F$.

\begin{proposition}\Label{FenchelNielsenForPinchedLoop}
Let $\tau_i$ be a sequence in $\TT$, such that the restrictions of $\tau_i$ to $S_E$ and to $S_F$ converge in their respective Teichmüller spaces as $i \to \infty$. 
Pick, for each $i$,  a representative $\xi_i\col \pi_1(S) \to \PSL_2\C$ of $b_L(\tau_i) \in \rchi$ so that $\xi_i | \Gam_E$ converges. 
Then, 
the restriction $\xi_i | \Gam_F$ converges if and only if the Fenchel-Nielsen twist parameter along $m$ converges as $i \to \infty$. 
\end{proposition}
\begin{proof}
In order to measure a twist parameter along $m$, pick pairs of pants $P_E, P_F$ of a pants decomposition of $S$ containing $m$, so that $P_E, P_F$ lift to regions in $\ti{S}$ that are adjacent along $\ti{m}$. 
Then let $a$ be a simple arc on $P_E$  connecting $m$ to another boundary component of $P_E$. 
Similarly, let $b$ be an arc on $F$ connecting $m$ to another boundary component of $P_F$.

For each $i$, let $\beta_i \col \H^2 \to \H^3$ be the bending map for $(\tau_i, L)$ equivariant via $\xi_i$, so that $\beta_i$ converges on $E$ as $i \to \infty$.
Let $M_{\tau_i}$ be the geodesic representative of $M$ on $\tau_i$, and let $\ti{M}_{\tau_i}$ be the $\pi_1(S)$-invariant total lift of $M_{\tau_i}$ to the universal cover $\H^2$ of $\tau_i$.
Let $\ti{m}_i$ be the component of $\ti{M}_{\tau_i}$ corresponding to $\ti{m}$.
Let $F_i, E_i$ be the region on $\ti{\tau}_i \minus \ti{M}_{\tau_i}$ corresponding to $F$ and $E$, respectively.

    Pick a (topological) lift $\ti{a}$ of $a$ to the universal cover $\ti{E}$ whose endpoint is on $\ti{m}$.
     For each $i$, let $a_i$ be the corresponding geodesic representative of $\ti{a}$ on the $E_i$ with an endpoint on $\ti{m}_i$.
 Similarly, pick a lift $\ti{b}$ of $b$ to $\ti{F}$ whose endpoint is on $\ti{m}$.
Let $b_i$ be the geodesic representative of $\ti{b}$ on $F_i$ whose endpoint is on $\ti{m}_i$.
Then the twist coordinate of $\tau_i$ along $m$ of the pants decomposition is the signed distance between the base points of the geodesic arcs $a_i$ and $b_i$ on $\ti{m}_i$ divided by the length of $m_i$.  

Let $v_i$ be the unit tangent vector of $b_i$ at the base point on $\ti{m}_i$. 
Since the weight of $\ti{m}_i$ is $\pi$ modulo $2\pi$,  $d \beta_i (v_i)$ is tangent to $\beta_i (E_i)$ at a point of $\ti{m}_i$; see \Cref{fTangentVectorAtPiBending}, Left. 
(Suppose,  against the hypothesis, that the weight of $m_i$ is not $\pi$ modulo $2\pi$ and $\length_{\tau_i} m_i \to 0$.
Let $\alpha_i \in \pi_1(S)$ represent $m_i$ fixing $\ti{m}_i$.
Then, by the hypothesis of the weight,  the totally geodesic hyperbolic plane in $\H^3$ containing $\beta_i(F_i)$ transversally intersects the totally geodesic hyperbolic plane containing $\beta_i(E_i)$ along the geodesic $\beta_i(\ti{m}_i)$ at the constant angle, the weight of $m$,  independent of $i$.
 Therefore $\beta_i(F_i)$ must diverge to the parabolic fixed point of the limit of  the hyperbolic element $b_L (\tau_i) \alpha_i$ as $i \to \infty$; therefore $b_L(\tau_i)$ diverges to infinity, which contradicts the other hypothesis.)

First suppose that $\lim_{i \to \infty} \length_{\tau_i} m$ is positive.  
The twist parameter is given by the distance between $\beta_i(a_i)$ and $\beta_i(b_i)$ along the geodesic $\beta_i(\ti{m}_i)$ divided by the length of $m_i$.
Therefore $\xi_i \vert \Gamma_F$ converges if and only if $\beta_i(b_i)$ converges, which is equivalent to saying the twisting parameter of $m$ converges in $\R$ as $i \to \infty$. 

Next suppose that $\lim_{i \to \infty} \length_{\tau_i} m$ is zero.   
Since $\tau_i$ converges to a hyperbolic surface with cusps and the discrete representation $\xi_i \vert \Gamma_E$ remains discrete in the limit as $i \to \infty$,  the holonomy along $m$ must converge to a parabolic element. 
Then the holonomy of $m$ converges to a parabolic element different from the identity.  

Then, since $\xi_i | \Gam_E$ converges as $i \to \infty$, thus $\xi_i \vert \Gamma_F$ converges, if and only if $\beta_i(v_i)$ converges to a tangent vector starting from the parabolic fixed point. 
The hyperbolic element $b_L(\tau_i) m$ converges to the parabolic element.
Therefore we can take a sequence of the hyperbolic one-parameter subgroups in $\PSL_2 \C$ containing $b_L(\tau_i) m$ parameterized by the twist coordinate so that the subgroup converges to the one-parameter parabolic subgroup of $\PSL_2\C$ containing the parabolic element $ \lim_{i \to \infty} \beta_L(\tau_i) m$, so that this parabolic element acts as the unit-length translation on the limit twisting parameter $\R$ along $m$. 
0
Therefore, the above convergence is equivalent to saying the twisting parameter of $m$ converges as $i \to \infty$ (\Cref{fTangentVectorAtPiBending}, Right). 
\begin{figure}
\begin{overpic}[scale=.03
]{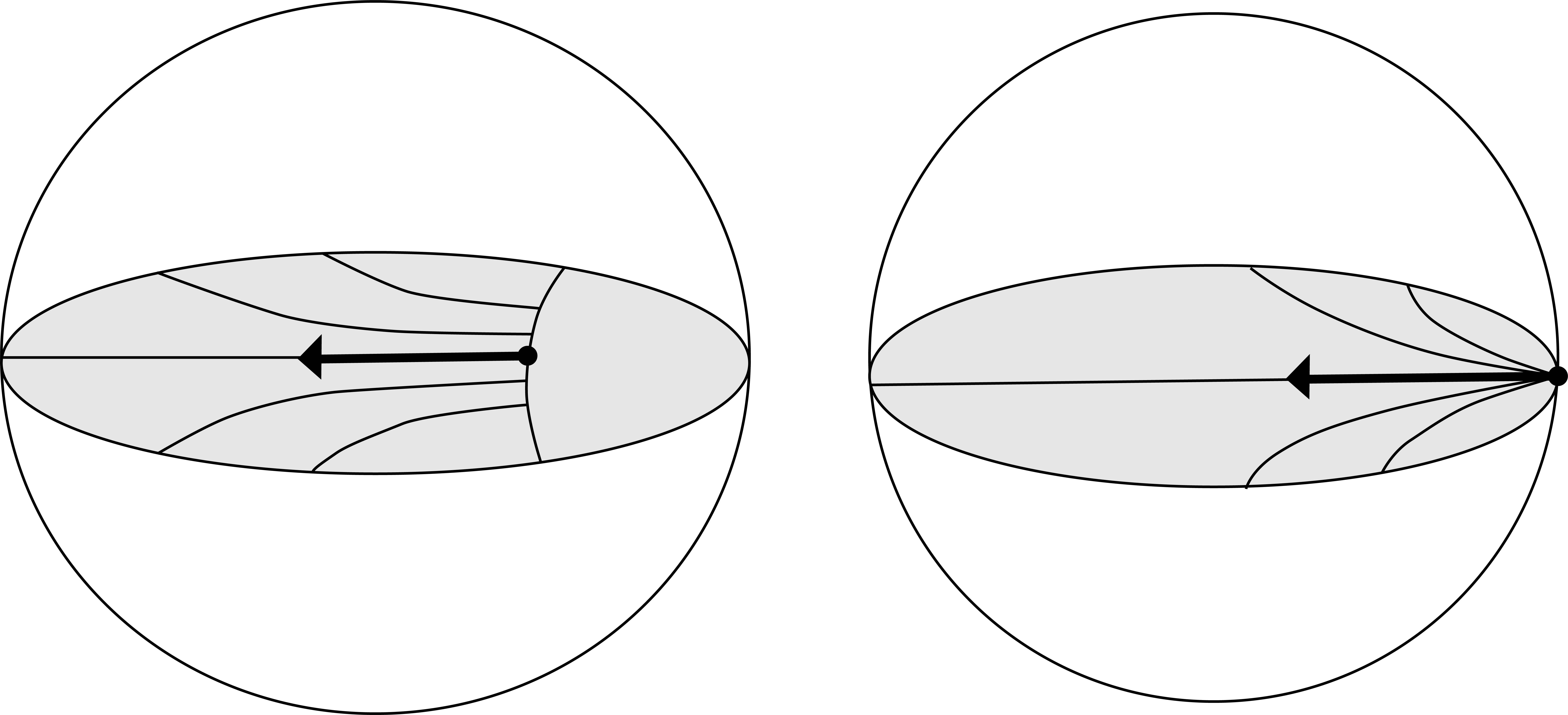} 
 \put(20 , 37 ){\textcolor{darkgray}{$\H^3$}}  
  \put(70 , 37 ){\textcolor{darkgray}{$\H^3$}}  
    \put(35 ,21 ){\small \textcolor{Black}{$d \beta_i v_i$}} 
     \put(65 ,23.5 ){\small \textcolor{Black}{$\lim d \beta_i v_i$}} 
      \end{overpic}
\caption{The convergence of the twist coordinate under neck-pinching.}\label{fTangentVectorAtPiBending}
\end{figure}
\end{proof}
The theorem follows from \Cref{FenchelNielsenForPinchedLoop} as follows: 
Suppose that $b_L(\tau_i)$ converges as $i \to \infty$. 
Then, the hyperbolic structure on every component of $S \minus M$ must converge.
Thus, for each loop $m$ of $M$,  $\length_{\tau_i} m$ limits to a non-negative number. 
By \Cref{FenchelNielsenForPinchedLoop}, as  $b_L(\tau_i)$ converges,  the twist parameters along each loop of $M$ converge. 
Since $\tau_i$ leaves every compact subset, at least one loop of $M$ must be pinched as $i \to \infty$. 
Hence the two conditions hold.

To prove the other direction, suppose that the lengths of some loops of $M$ limit to zero and all the other Fenchel-Nielsen coordinates with respect to $P$ converge in the parameter space as $i \to \infty$. 
Let $M'$ be the sub-multiloop of $M$ consisting of the loops whose lengths go to zero. 
Then, for each component $F$ of $S \minus M'$,  $b_L (\tau_i) \vert \pi_1(F)$ converges as $i \to \infty$. 
Therefore, by \Cref{FenchelNielsenForPinchedLoop}, $b_L(\tau_i)$ converges. 
This completes the proof.
   \Qed{NonpropernessOfBendingTeichmullerSpaces}


\section{The boundary map of the real bending map} 

\begin{theorem}\Label{GenericConvergenceToThurstonBoundary}

Let $[V] \in \PML \cong \bdr\TT$ be a Thurston boundary point. 
Suppose that every singular leaf of  $V$ is a tripod, i.e. three rays sharing a common endpoint. 
Let $\tau_i \in \TT$ be a sequence of hyperbolic surfaces converging to $[V]$. 

Then, for every measured lamination $L \in \ML$,  the representation $b_L (\tau_i) \in \rchi$ converges to the Morgan-Shalen boundary point corresponding to $[V]$ as $i \to \infty$. 
\end{theorem}
\begin{remark}
Suppose, in contrast,  that the projective measured foliation $[V]$ contains a singular leaf that is not a tripod. 
Then, the limit tree of $\tau_i$ may possibly be ``folded" into a smaller tree by a ``straight map"; thus a limit of $b_L (\tau_i)$ might not coincide with the Morgan-Shalen boundary point of $[V]$, similar to the folding phenomenon in \cite{Dumas18HolonomyLimit}). 
\end{remark}
   
\proof
Let $\ell$ be an essential simple closed curve on $S$. 
As every singular leaf of $V$ is a tripod, $\ell$ is not contained in a leaf of $V$. 

Pick a marked Riemann surface $X \in \TT$ as a base point of the harmonic parametrization of $\TT$ (\cite{Wolf91}  \cite{Hitchin87}).  
Then, for each $i = 1, 2, \dots$,  there is a unique harmonic map $h_i\col  X \to \tau_i$, preserving the marking. 
Let $q_i$ be its Hopf differential on $X$, which is a holomorphic quadratic differential on $X$.
In this manner, $\TT$ is parametrized by the complex vector space of dimension $3g -3$ consisting of holomorphic quadratic differentials on $X$.  
 
This harmonic map parametrization is compatible with Thurston's boundary of $\TT$. 
For each $i = 1,2, \dots$, let $E_i$ be the flat surface corresponding to $(X, q_i)$.
As $\tau_i$ converges to the boundary point $[V]$ as $i \to \infty$ (\cite{Wolf91}),  the unit-area surface $E_i/ \sqrt{{\rm Area}\, E_i}$ converges to the flat surface $E_\infty$ as $i \to \infty$ so that $E_\infty$ realizes the conformal structure of  $X$ and the projective measured foliation $[V]$ as its vertical foliation.  
 
For each $i = 1, 2, \dots$, let $\ell_i$ be the geodesic loop on the hyperbolic surface $\tau_i$ representing $\ell$.  
Similarly, let $m_i$ be a geodesic representative of $\ell$ on $E_i$.
We also let $m_\infty$ be a geodesic loop on $E_\infty$ realizing $\ell$. 
 By taking appropriate representatives $m_i$, we may assume that $m_i$ on the normalization $E_i /\sqrt{ \Area E_i}$ converges to $m_\infty$ on $E_\infty$.
 Note that as $V$ has no saddle connections, if $i$ is sufficiently large, $m_i$ is transverse to the vertical foliation of $E_i$ except at the singular points. 

We divide the proof into the following two cases.
\begin{enumerate}
\item $| L| = |V|$. \Label{iSupportCoincide}
\item  $|L| \neq |V|$.  \Label{iSupportDoNotCoincide}
\end{enumerate}
In Case (\ref{iSupportCoincide}), we show that the transversal measure of every $\ell_i$ the geodesic representative of $L$ on $\tau$ is uniformly small (\Cref{fSmallIntersection}). 
In Case (\ref{iSupportDoNotCoincide}),  we show that $\ell_i$ intersects the geodesic representative of $L$ on $\tau$ at a small angle (\Cref{fNonVerticalBendingLamination}). 

(Case \ref{iSupportCoincide})
Suppose that  $| L| = |V|$.
We show that,  if $i$ is sufficiently large, every unit-length segment of $\ell_i$ on $\tau_i$ intersects the geodesic representative of $L$ in a uniformly small amount w.r.t. the transversal measure of $L$--- this implies that the translation length of $\ell$ does not change much by bending along $L$. 

Since $| L |  = |V|$, let $V_L$ denote the measured foliation on $S$ corresponding to $L$.
For $\ep > 0$, pick  segments $m_{\infty, 1} \dots m_{\infty, n}$ of $m_\infty$, such that 
$V_L(m_{\infty, j})  < \ep/3$ for all $j = 1, \dots, n$ and the endpoints of $m_{\infty, j}$ are not at the singular points of $E_\infty$, where  $V_L(m_{\infty, j})$ denotes the measure of $m_{\infty, j}$ given by the transversal measure of $V_L$.

As $| L | = | V |$, 
we can pick a scaling $V_i$ of  the vertical measured foliation of $E_i$, such that $V_i$ converges to $V$ as $i \to \infty$.  
We may assume $V_i$ is the vertical foliation of the scaled flat surface $E_i /\sqrt{\Area E_i}$, if necessarily, by rescaling $V$. 
By the convergence of $m_i$ to $m_\infty$, we cover $m_i$ by open geodesic segments  $m_{i, 1} \dots m_{i, n}$  which respectively converge to the cover $m_{\infty, 1} \dots m_{\infty, n}$ of $m_\infty$ as $i \to \infty$. 
Moreover, if $i > 0$ is sufficiently large, then $V_L(m_{i, j})  < \frac{\ep}{3}$ for all $j = 1, \dots, n$.

Using the asymptotic behavior of harmonic maps away from singularities described in the preliminaries (\S \ref{sHarmonicMaps}),  
the $h_i$-image of $m_i$ is a quasi-geodesic loop on $\tau_i$ homotopic to $\ell_i$ such that the quasi-isometric distortion is uniformly bounded in $i = 1,2, \dots$. 
  Pick $\alpha \in \pi_1(S)$ representing $\ell$. 
Then, let $\ti{m}_i$ be the lift of $m_i$ to the universal over $\ti{E}_i$ invariant by $\alpha$, and similarly $\ti{\ell}_i$ be the lift of $\ell_i$ to the universal cover of $\tau_i$ invariant by $\alpha$.    
  Let $\ti{h}_i$ be the lift of the harmonic map $h_i\col X \to \tau_i$ to a $\pi_1(S)$-equivariant mapping between their universal covers.
The distance between $\ti{h}_i(\ti{m}_i)$ and $\ti\ell_i$ is bounded from above uniformly in $i$ by the convergence of $\tau_i \to [V] \in \bdr \TT$ and the asymptotic behavior of harmonic maps $h_i$ as $i \to \infty$.

Let $Z_i$ be the set of the singular points of the flat surface $E_i$. 
For $r> 0$, let $N_i^r$ be the neighborhood of $Z_i$ on $E_i$ corresponding to the $r$-neighborhood of $Z_i$ on $E_i / \sqrt{\Area E_i}$.
Fix a small $r > 0$. 
Thus,  for every $\ep > 0$, if $i > 0$ is sufficiently large, then the $h_i$-image of $m_i \minus N_i^r$ are $(1+\ep)$-biLipschitz embeddings w.r.t. the horizontal Euclidean length of $E_i$.
Therefore, this $h_i$-image is  contained in an $\ep$-neighborhood of $\ell_i$ (\S 3.5). 
Then, by the uniform quasi-isometric property of $\ti{h}_i \vert \ti{m}_i$, we can cover the geodesic loop $\ell_i$ by geodesic segments $\ell_{i, 1} \dots \ell_{i, n}$ corresponding to  $m_{\infty, 1} \dots m_{\infty, n}$ such that $\ell_{i, j}$ is a geodesic segment on $\tau_i$  whose endpoints are $\ep$-close to the $h_i$-image of the endpoints of $m_{i, j}$.
\begin{lemma}\Label{SmallTransversalMeasureRelativeToTheLegth}
For every $r > 0$, if $i > 0$ is sufficiently large, then  $L_i (\ell_{i, j}) < \ep$  and $\length(\ell_{i, j}) > r$ for all $j = 1, \dots, n$, where $L_i (\ell_{i, j}')$ denote the measure of $\ell_{i j}'$ given by the transversal measure of $L_i$. 
\end{lemma}
\begin{proof}
We first prove $L_i (\ell_{i, j}) < \ep$, the main part of the claim.  
Pick $\alpha \in \pi_1(S)$ representing the loop $\ell$. 
Then, let $\ti{m}_\infty$ be  a (bi-infinite) lift of $m_\infty$ to the universal cover $\ti{E}_\infty$ invariant by $\alpha$. 
For each $j = 1, \dots, n$, pick a lift $\ti{m}_{\infty, j}$ of the segment $m_{\infty, j}$ to the universal cover $\ti{E}_\infty$, such that  $\ti{m}_{\infty, j}$ is contained in $\ti{m}_\infty$. 
Let $T_\infty$ be the $\R$-tree obtained by collapsing the vertical leaves of $\ti{E}_\infty$, 
and let $\Psi_\infty\col \ti{E}_\infty \to T_\infty$ denote the quotient map.
Then  $\ti{m}_\infty$ is embedded in $T_\infty$ by $\Psi$ since $m_\infty$ is transverse to $V_\infty$. 

For each endpoint $p$ of $\ti{m}_{\infty, j}$ and each component $C$ of $\ti{E}_\infty \minus \ti{m}_\infty$, pick a rectangle $R$ with horizontal and vertical edges such that: 
 \begin{itemize}
 \item $p$ is a vertex of $R$;
 \item the interior of $R$ contains  no singular point  of $\ti{E}_\infty$;
 \item the opposite vertex of $p$ on the rectangle $R$ is a singular point $z$ of $\ti{E}$;
 \item the vertical edge of $R$ starting from $p$ is contained in $C$; 
 \item leaves of $V_i$ intersecting the interior of $R$ are disjoint from $\ti{m}_{\infty, j}$;
 \item the horizontal length of $R$ is less than $\ep/3$
  (Figure \ref{fSmallIntersection}, Left). 
\end{itemize}
Let $R_{j, 1}, R_{j, 2}$ denote the rectangles for one endpoint of $\ti{m}_{\infty, j}$ contained in different components of $\ti{E}_\infty \minus \ti{m}_\infty$, and let $R_{j, 3}, R_{j, 4}$ denote the rectangles for the other endpoints of $\ti{m}_{\infty, j}$.  
Similarly, let $z_1, z_2$ denote the singular points of $\ti{E}_\infty$ which are vertices of $R_{j, 1}, R_{j, 2}$ opposite from the endpoint vertex, and let $z_3, z_4$ denote the singular points of $\ti{E}_\infty$ which are vertices of $R_{j, 3}, R_{j, 4}$ opposite from the other endpoint vertex. 
We may assume that the projections $\Psi_\infty (z_1), \Psi_\infty (z_2),\linebreak[2]\Psi_\infty (z_3),\Psi_\infty (z_4)$ lie on $\ti{m}_{\infty}$ in this order, if necessary, by exchanging $z_1$ and $z_2$, and $z_3$ and $z_4$.

For each $k = 1,2, 3, 4$,  pick a small tripod neighborhood $\gam_k$ of the singular point $z_i$ in the horizontal leaf containing $z_i$ (Figure \ref{fSmallIntersection}). 
As $E_i/ \sqrt{\Area E_i}$ converges to $E_\infty$, for sufficiently large $i$, we pick tripod neighborhoods $\gam_{i, j, 1}, \gam_{i, j, 2}, \gam_{i, j, 3}, \gam_{i, j, 4}$ of the singular points of $\ti{E}_i$ such that 
\begin{itemize}
\item  $\gam_{i, j, 1}, \gam_{i, j, 2}, \gam_{i, j, 3}, \gam_{i, j, 4}$ converge to $\gam_{j, 1}, \gam_{j, 2}, \gam_{j, 3}, \gam_{j, 4}$ as $i \to \infty$, respectively. 
\end{itemize}
 If $i$ is sufficiently large,  by the harmonic map $h_i$, a small neighborhood of $\gam_{i, j, k}$ maps to a region close to an ideal triangle $\Delta_{i, j, k}$ in a large compact subset in $\ti\tau_i$  \cite{Minsky92}.
 Since the interior of $R_{i, j, k}$ contains no singular point, we may assume that $\ti\ell_i$ is a common edge of $\Delta_{i, j, 1}, \Delta_{i, j, 2}, \Delta_{i, j, 3}, \Delta_{i, j, 4}$by small perturbation. 
Then the endpoints of $\ti\ell_i$ are two of the ideal vertices of $\Delta_{i, j, k}$.
 Let $v_k$ be the (other) ideal vertex of $\Delta_{i, j, k}$ which is not an endpoint of $\ti{\ell}_i$.
By reordering, we may additionally assume that $\Delta_{i, j,1 }$ and $\Delta_{i, j, 4}$ are contained in the same component of $\H^2 \minus \ti\ell_i$ and  $\Delta_{i, j, 2}$ and $\Delta_{i, j, 3}$ are contained in the other component of $\H^2 \minus \ti\ell_i$. 

Let $\ti\ell_{i, j}$ be the $\alpha$-invariant lift of $\ell_{i, j}$ in $\ti\tau_i \cong \H^2$. 
Let $L_i$ denote the geodesic measured lamination on the hyperbolic surface $\tau_i$ representing $L$.
For sufficiently large $i$,  if a leaf $\ell$ of $\ti{L}$ intersects $\ti{\ell}_{i, j}$, then an endpoint of $\ell$ is between $v_1$ and $v_4$ and the other endpoint between $v_2$ and $v_3$.
 Since $V_L (m_{i, j})  < \frac{\ep}{3}$ and the horizontal widths of $R_{j, 1},R_{j, 2},R_{j, 3},R_{j, 4}$ are less than $\frac{\ep}{3}$, therefore, $L (\ell_{i, j}) < \ep$ for sufficiently large $i$. 
 
 Since $m_{\infty, j}$ is transverse to the vertical foliation and the harmonic map $h_i$ stretches in the horizontal direction more and more, the length of $\ell_{i, j}$ diverges to $\infty$ as $i \to \infty$.
\end{proof}

\begin{figure}\Label{fRectanglesForSegments}
\begin{overpic}[scale=.2
] {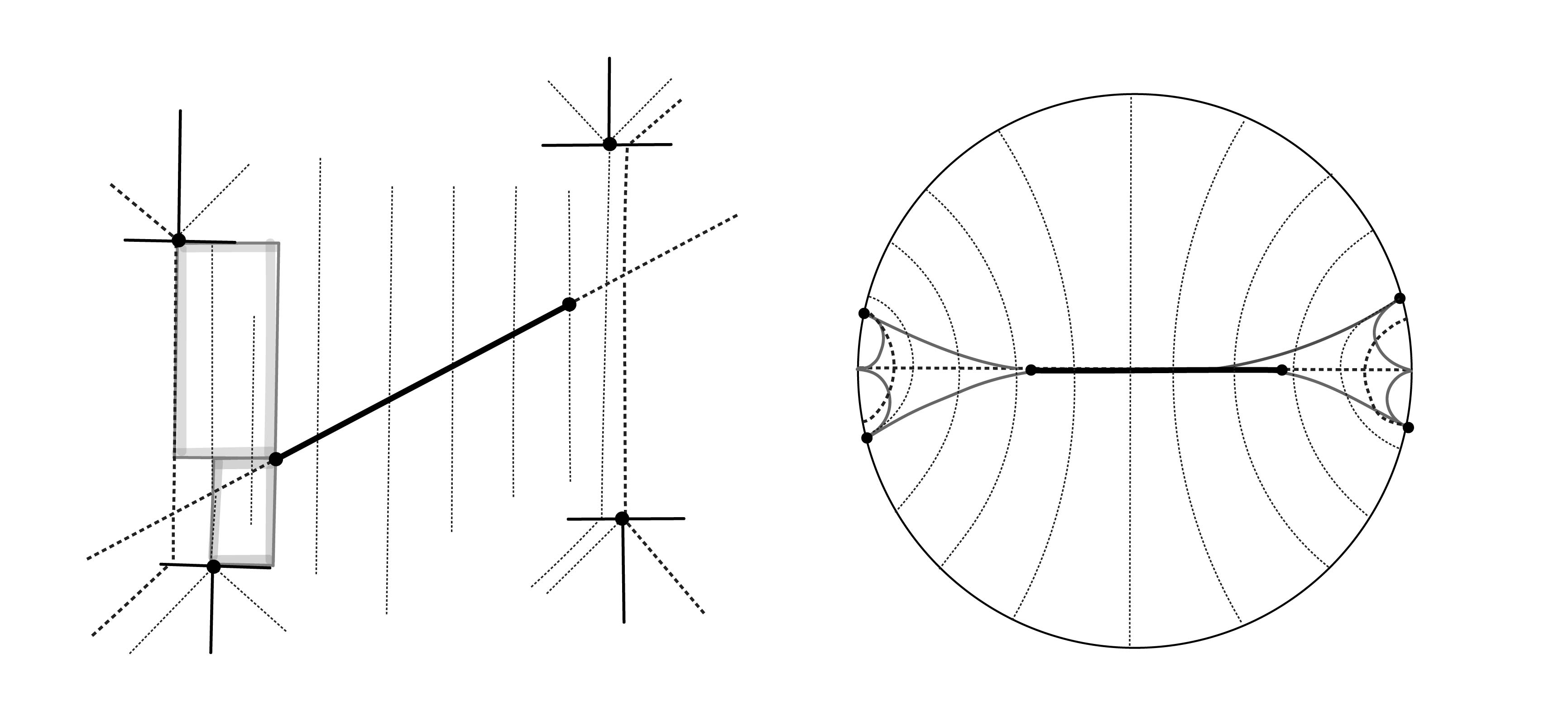} 
 \put(42,27 ){\textcolor{darkgray}{\small \contour{white}{$\ti{m}_\infty$}}}   
  \put(12, 34 ){\textcolor{black}{\small \contour{white}{$\gam_1$}}}   
  \put(14, 6 ){\textcolor{black}{\small \contour{white}{$\gam_2$}}}   
  \put(36, 7 ){\textcolor{black}{\small \contour{white}{$\gam_3$}}}   
  \put(34, 35 ){\textcolor{black}{\small \contour{white}{$\gam_4$}}} 
  \put(24, 18 ){\textcolor{black}{\small \contour{white}{$\ti{m}_{\infty, j}$}}} 
  \put(20, 38 ){\textcolor{black}{\small \contour{white}{$\ti{E}_\infty$}}} 
  \put(51.5, 27 ){\textcolor{black}{\small \contour{white}{$v_1$}}}   
  \put(51.5, 16 ){\textcolor{black}{\small \contour{white}{$v_2$}}}   
  \put(90, 17 ){\textcolor{black}{\small \contour{white}{$v_3$}}}  
    \put(89, 31 ){\textcolor{black}{\small \contour{white}{$v_4$}}}   %
    \put(71, 19 ){\textcolor{black}{\small \contour{white}{$\ti\ell_{i, j}$}}} 
      \put(11.5, 23 ){\textcolor{black}{\small \contour{white}{$R_{j, 1}$}}} 
      \put(70, 35 ){\textcolor{black}{\small \contour{white}{$\ti{L}$}}} 
      \end{overpic}
\caption{The geodesic segment $\ti{\ell}_{i,  j}$ has a small transversal measure. The ideal triangle with the vertex $v_k$ is $\Delta_{i, j, k}$.}\Label{fSmallIntersection}
\end{figure}

Let $\beta_i$ be the $b_L(\tau_i)$-equivariant pleated surface $\H^2 \to \H^3$ obtained by bending, in $\H^3$,  the universal cover of $\tau_i$ along the inverse-image of $L$.
By \Cref{SmallTransversalMeasureRelativeToTheLegth},
for every $\ep > 0$, if $i$ is sufficiently large, then  the restriction of $\beta_i\col \ti{\tau}_i \to \H^3$ to  $\ti\ell$ into $\H^3$ is a $(1 - \ep, 1 + \ep)$-biLipschitz embedding (\cite[I.4.2.10]{CanaryEpsteinGreen84}). 
Hence the ratio of the translation length of $b_L(\tau_i) \alpha$ and the $\tau_i$-length of $\ell_i$ converges to one as $i \to \infty$. 
Therefore $b_L(\tau_i)$ converges to $[V]$ in the Morgan-Shalen boundary.

(Case \ref{iSupportDoNotCoincide})
Suppose that  $| L| \neq |V|$.
In this case, we show that every unit segment of the geodesic representative of $\ell$ on $\tau_i$ intersects $L$ at uniformly small angles if $i$ is sufficiently large.

Let $\LL_\infty$ be the geodesic representative of $L$ on $E_\infty$.
Recall that $m_\infty$ is the geodesic representative of  $\ell$ on the limit flat surface $E_\infty$ of unit area. 
Consider the  $\pi_1(S)$-invariant measured lamination of $\ti\tau_i$  obtained by pulling back $L_i$ on $\tau_i$ by the universal covering map. 
 Let $\ti{L}_i$ be its $\alpha$-invariant measured lamination on $\ti\tau_i$ consisting of leaves intersecting $\ti\ell_i$.

\begin{proposition}\Label{AsymptoticallySmallAngle}
For every $\ep > 0$, if $i > 0$ is sufficiently large, then $\angle_{\ti\tau_i} (\ti\ell_i, \ti{L}_i) < \ep$.
\end{proposition}
 (See \S \ref{sAngleBetweenLaminations} for the definition of the angle $\angle_{\ti\tau_i} (\ti\ell_i, \ti{L}_i)$.)
\begin{proof}
Let $\mu$ be a leaf of $\LL_\infty$ intersecting $\ti{m}_\infty$ at a point $p_\infty$.

Pick Euclidean rectangles $R_{\infty, 1}, R_{\infty, 2}$ in $\ti{E}_\infty$ such that 
\begin{itemize}
\item   $R_{\infty, 1}, R_{\infty, 2}$ have horizontal and vertical edges and no singular points in their interiors; 
\item the interiors of $R_{\infty, 1}$ and $R_{\infty, 2}$ are contained in different components of $\ti{E}_\infty \minus \ti{m}_\infty;$
\item one horizontal edge of $R_{\infty, k} (k = 1, 2)$ is contained in $\ti{m}_\infty$, and each vertical edge of $R_{\infty, k}$ contains exactly one singular point of $\ti{E}_\infty$;
\item  the singular points on the vertical edges of $R_{\infty, k}$ divide the boundary $\bdr R_{\infty, k}$ into two piecewise linear curves, and  $\mu$ passes through $R_{\infty, k}$ and $\mu$ intersects each  piecewise-linear segment of $\bdr R_{\infty, k}$ in a single point
(\Cref{fNonVerticalBendingLamination}).
\end{itemize}
Since $E_i/ \sqrt{\Area E_i}$ converges to $E_\infty$ as $i \to \infty$, 
for sufficiently large $i$, 
we pick Euclidean rectangles $R_{i, 1}, R_{i, 2}$ in $\ti{E}_i$ such that $R_{i, j} \to R_i$ as $i \to \infty$. 
Let $\LL_i$ be the geodesic representative of the measured lamination $L$ on the flat surface $E_i$. 
By the convergence,  the properties of  $R_{\infty, 1}$ and $R_{\infty, 2}$ carry over to  $R_{i, 1}$ and $R_{i, 2}$ for sufficiently large $i$.   
Namely, letting $\mu_i$ be the leaf of $\ti\LL_i$ on $\ti{E}_i$ corresponding to $\mu$, 
\begin{itemize}
\item   $R_{i, 1}, R_{i, 2}$ have horizontal and vertical edges and no singular points in their interiors; 
\item the interiors of $R_{i, 1}$ and $R_{i, 2}$ are contained in different components of $\ti{E}_i \minus \ti{m}_i;$
\item one horizontal edge of $R_{i, k} (k = 1, 2)$ is contained in $\ti{m}_i$, and each vertical edge of $R_{i, k}$ contains a unique singular point of $\ti{E}_i$; 
\item the singular points on the vertical edges of $R_{i, k}$ divide the boundary $\bdr R_{i, k}$ into two piecewise geodesic curves, and  $\mu_i$ passes through $R_{i, k}$ and $\mu_i$ intersects each component of $\bdr R_{\infty, k}$ minus the singular point in a single point. 
\end{itemize}
Let $z_1, z_2$ be the singular points of $\bdr R_{\infty, 1}$ and $z_3, z_4$ be the singular points of $\bdr R_{\infty, 2}$.
For $k = 1,2, 3, 4,$  let $\gam_k = \gam_{\infty, k}$ be a small horizontal tripod neighborhood of $z_k$ (\Cref{fNonVerticalBendingLamination}).
We may assume that the projections of $z_1, z_2, z_3, z_4$ to $\ti{m}_i$ along vertical leaves lie on $\ti{m}_i$ in this order (of indices).

For $i$ large enough, 
 let $z_{i, k}$ be a singular point of the vertical edge of $R_{i, k}$ such that $z_{i, k} \to z_k$ as $i \to \infty$.
 
Let $\gam_{i, k}$ be a horizontal tripod neighborhood of $z_{i, k}$ such that $\gam_{i, k}$ converges to $\gam_{\infty, k}$ as $i \to \infty$. 
As in Case One, by the work of Wolf and Minsky (\cite{Wolf91, Minsky92}), if $i$ is sufficiently large,  a small neighborhood of $\gam_{i, k}$ in $\ti{E}_i/\sqrt{\Area E_i}$  maps to a region in $\ti\tau_i \cong \H^2$ close to an ideal triangle $\Delta_{i, k}$ in a large compact subset.
Since $R_{i, k}$ and $R_{\infty, k}$ contain no singular points in their interiors, we may assume that the geodesic $\ti\ell_i$ is a unique common boundary edge of $\Delta_{i, 1}, \Delta_{i, 2}, \Delta_{i, 3}, \Delta_{i, 4}$.

Let $v_{i, k}$ be the ideal vertex of $\Delta_{i, k}$ which is not an endpoint of $\ti\ell_i$ for $k = 1, 2, 3, 4$.
Then, since the hyperbolic metric stretches in the horizontal direction and shrinks in the vertical direction of $E_i$ (\S \ref{sHarmonicMaps}), 
the distance between the projections of $v_{i, 2}$ and $v_{i, 3}$ to $\ti\ell_i$ diverges to infinity.    
 
Let $\lambda_i$ be the geodesic in $\ti\tau_i$ fellow travels with $h_i (\mu_i)$.
The boundary circle $\bdr \ti\tau_i \cong \mathbb{S}^1$ with the four points $v_{i, 1}, v_{1, 2}, v_{i, 3}, v_{i, 4}$ removed consists of four circular arcs. 
Then, one endpoint of $\lam_i$ is in the circular arc between $v_{i, 1}$ and $v_{i, 2}$, and the other endpoint is in the circular arc between $v_{i, 3}$ and $v_{i, 4}$. 
Since those circular arcs contain the endpoints of $\ti\ell_i$, the divergence of distance above implies  $\angle_{\ti\tau_i} (\ti\ell_i, \lam_i) \to 0$ as $i \to \infty$ (\Cref{fNonVerticalBendingLamination}). 

Suppose that another leaf $\mu'$ of $\ti\LL_\infty$ is sufficiently close to $\mu$ in a large compact subset containing the intersection point $p_\infty$ and the rectangles $R_{\infty, 1}, $ and $R_{\infty, 2}$ in $\ti{E}_\infty$.
 Let $\lam'_i$ be the leaf of $\ti{L}_i$ corresponding to $\mu'$.
Then, similarly,  an endpoint of $\lam'_i$ is in the circular arc between $v_{i, 1}$ and $v_{i, 2}$  and the other endpoint is in the circular arc between $v_{i, 3}$ and $v_{i, 4}$ for sufficiently large $i$. 
Therefore, by the divergence of the distance between the projections,  $\angle_{\ti\tau_i} (\ti\ell_i, \lam'_i) \to 0$ as $i \to \infty$. 

Since $m_\infty$ is a closed curve on $E_\infty$, by compactness, we see that  $\angle_{\ti\tau_i} (\ti\ell_i, \ti{L}_i) \to 0$ as $i \to \infty$.  
\end{proof}

By \Cref{AsymptoticallySmallAngle}, 
$\angle_{\tau_i} (L_i, m_i) \to 0$ as $i \to \infty$. 
Let $\rho_i = b_L (\tau_i)\col \pi_1(S) \to \PSL_2\C.$
Then, since the angle between $L_i$ and $m_i$ goes to zero,  the ratio of $\length_{\tau_i} m_i$ and the translation length of $\rho_i (\alpha)$ converges to one as $i \to \infty$ (\cite[Proposition 4.1]{Baba-15gt}). 
Thus $b_L (\tau_i)$ converge to $[V]$ in the Morgan-Shalen compactification. 
\begin{figure}
\begin{overpic}[scale=.24
] {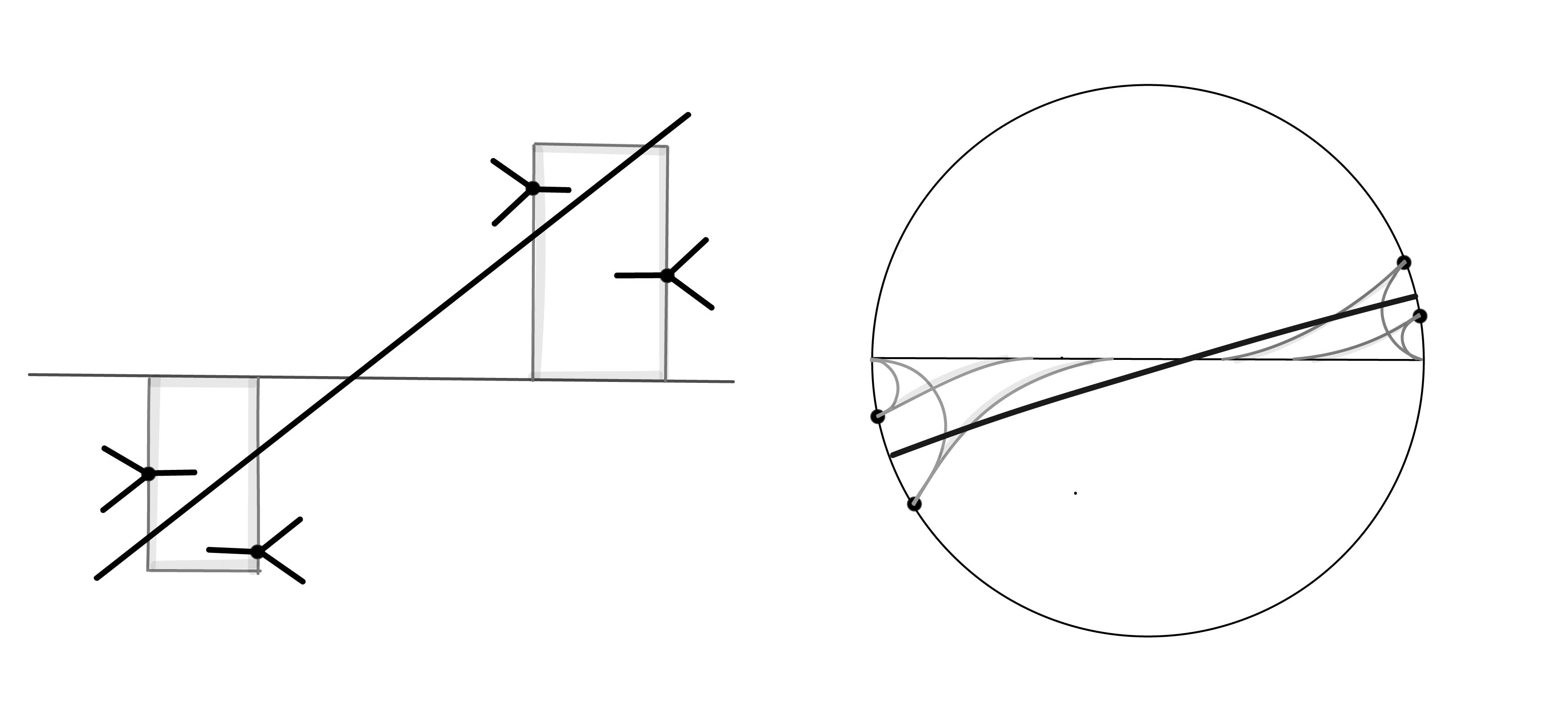} 
 \put(10, 38){\textcolor{darkgray}{\small \contour{white}{$\ti{E}_\infty$} } }   
 \put(10.1, 19){\textcolor{darkgray}{\small \contour{white}{$R_{\infty, 1}$} } }   
 \put(4, 15){\textcolor{black}{\small \contour{white}{$\gam_1$} } }   
 \put(20, 10){\textcolor{black}{\small \contour{white}{$\gam_2$} } }   
 \put(35.5, 23.5){\textcolor{darkgray}{\small \contour{white}{$R_{\infty, 2}$} } }   
  \put(22, 25){\textcolor{black}{\small \contour{white}{$\mu$} } }   
 \put(4, 23.5){\textcolor{darkgray}{\small \contour{white}{$\ti{m}_\infty$}}}   
 \put(70, 33){\textcolor{darkgray}{\small \contour{white}{$\ti\tau_i$} } }   
 \put(58, 24){\textcolor{darkgray}{\small \contour{white}{$\Delta_1$} } }   
 \put(61, 16){\textcolor{darkgray}{\small \contour{white}{$\Delta_2$} } }   
 \put(67, 24){\textcolor{darkgray}{\small \contour{white}{$\ti\ell_i$} } }   
 \put(70, 18){\textcolor{black}{\small \contour{white}{$\lam_i$} } }   
 \put(55, 11){\textcolor{black}{\small \contour{white}{$v_{i, 2}$} } }   
 \put(90, 31){\textcolor{black}{\small \contour{white}{$v_{i, 3}$} } }   
      \end{overpic}
\caption{$\angle(\mu_i, \ti\ell_i)$ is small for large $i$.  }\Label{fNonVerticalBendingLamination}
\end{figure}
\Qed{GenericConvergenceToThurstonBoundary}

\section{Framed character varieties along loops}\Label{sFramedCharacterVarieties}

    We have analyzed the real-analytic embedding $b_L\col \TT \to \rchi$ defined for an arbitrary measured lamination $L \in \ML$. 
As $\TT$ is regarded as the Fricke space, a component of the real slice of the character variety $\rchi$, one can certainly extend $b_L$ to a holomorphic mapping from a neighborhood of $\TT$ in $\rchi$ into $\rchi$. 
However, it does {\it not} extend to the entire character variety $\rchi$ for multiple reasons. 
In particular, for a representation $\rho\col \pi_1(S) \to \PSL_2\C$, if there is no $\rho$-equivariant pleated surface in $\H^3$ {\it realizing} the measured lamination $L$ (i.e. its pleating locus contains $L$), then the bending the representation along $L$ does not make sense.  

For instance, if a representation $\rho\col \pi_1(S) \to \PSL_2\C$ takes a loop $\ell$ with weight $w$ on $S$ to a parabolic element,  then there is no $\rho$-equivariant pleated surface realizing $\ell$.  
Suppose in addition that the restriction of $\rho$ to each component of $S \minus \ell$ is non-elementary (which generically holds true). Then, if  a sequence of representations $\rho_i$ realizing $\ell$ converges to $\rho$, then the bending of $\rho_i$ along $\ell$ by angle $w$ must diverge in $\rchi$ as $i  \to \infty$. 

Therefore, in this section, we modify the character variety $\rchi$ and obtain a closed complex-analytic set, which will be a domain of the complexified bending map.
   
  For a surface with punctures,   Fock and Goncharov introduced a framing of a surface group representation (\cite{FockGoncharov06}). 
Their framing assigns a fixed point of peripheral holonomy around each puncture. 
In fact,  this framing was useful for describing the deformation space of $\CP^1$-structures on a surface with punctures via their framed holonomy representations (\cite{Allegretti_Bridgeland_20, GuptaMj21, Baba-25}). 

In this paper, we introduce a certain framing along loops which assigns a pair of distinct fixed points of their holonomy. 
Such framings will be used to determine the axes for bending deformation even when the holonomy along loops is trivial. 
   
\subsection{Framing of  Representations  along a loop}\Label{sRepresentataionsFramedAlongALoop} 

For simplicity, we first discuss the modification in the case that the bending lamination is a single loop.
Let $\RR$ be the space of representations $\pi_1(S) \to \PSL_2\C$ (without any equivalence relation).
Then $\RR$ is an affine algebraic variety: 
    Namely, pick a presentation of the fundamental group $\pi_1(S)$, for instance $$\pi_1(S) = \langle\, a_1, b_1, \dots, a_g, b_g \mid  \Pi_{i = 1}^n [a_i, b_i]\, \rangle.$$
    Since $\PSL_2(\C)$ embeds into ${\rm GL}_3(\C)$ by the adjoint representation, $\PSL_2(\C)$ is a complex affine Lie group sitting in $\C^9$.
    Then, by the embedding $\RR \to (\C^9)^{2 g}$ defined by $$\rho \mapsto (\rho(a_1), \rho(b_1), \dots, \rho(a_g), \rho(b_g)) \in (\C^9)^{2g},$$  $\RR$ has an affine algebraic structure on cut by the equation corresponding to the relator $\Pi_{i = 1}^n [a_i, b_i]$. 
 
Let $\ell$ be a simple closed curve on $S$. 
Let $\Gam_\ell$ be the set of elements in $\pi_1(S)$ whose free homotopy classes are the homotopy class of $\ell$ on $S$; clearly, elements in $\Lam$ are conjugate to each other by elements in $\pi_1(S)$.

Pick an element  $\alpha_\ell \in \Gam_\ell$. 
Let $\rho\col \pi_1(S) \to \PSL_2\C$ be a homomorphism.
Suppose that $\rho(\alpha_\ell)$ is {\it not} a parabolic (but it can be the identity). 
Then, there is an ordered pair $(u, v)$ of distinct points $u, v$ on $\CP^1$ fixed by $\alpha_\ell$ pointwise.
We can equivariantly extend a pair $(u, v)$ to pairs $(u_\gam, v_\gam)$ for all representatives $\gam \in \Gam_\ell$ so that $\gam$ fixes $u_\gam$ and $v_\gam$ in $\CP^1$. 
Such an equivariant assignment $(u_\gam, v_\gam)_{\gam \in \Gam_\ell}$ of ordered fixed points of  $\gam$ is called a {\sf framing} of $\rho$ along $\ell$. 
By abuse of notation, we denote this equivariant framing $\{ (u_\gam, v_\gam) \}_{\gam \in \Gam_\ell}$,  by $(u, v)$, since it is determined by the initial choice $(u, v)$ for $\alpha_\ell$. 
We call the triple $(\rho, u, v)$ {\sf a framed representation}. 
We, later, utilize the equivariant framing to produce the equivariant bending axes (\S \ref{sBendingFramedRepresentataions}).
 Let 
$$R_\ell = \bigg\{   (\rho, u, v)  \in \RR \times (\CP^1)^2   \,\bigg\vert\,
 \rho(\alpha_\ell) u = u, \rho(\alpha_\ell) v = v, u \neq  v
  \bigg\}.$$
Then $R_\ell$ is a closed analytic subset of $\RR \times (\CP^1 \times \CP^1 \minus D)$, where $D$ is the diagonal $\{(z, z)  \mid z \in \CP^1\}$.
Note that if $(\rho, u, v) \in R_\ell$, then the $\rho(\alpha_\ell)$ can {\it not} be a parabolic element, since $u, v$ are distinct fixed points of $\rho(\alpha_\ell)$.
On the other hand, $\rho(\alpha_\ell)$ can be either hyperbolic, elliptic, or the identity.

Let $\GG_\ell$ be the subgroup of the mapping class group of $S$ which preserves the loop $\ell$.
Clearly, $\GG_\ell$ acts on $R_\ell$ by marking change.

We now assume that the loop $\ell$ has a weight in $\R_{> 0}$.
Suppose, first, that the weight $\omega$ of the loop $\ell$ is not equal to $\pi$ modulo $2\pi$. 
Fix any complex number  $w \in \C$ with $|w| > 1$.
Then, given $(u, v) \in \CP^1 \times \CP^1$ with $u \neq v$, there is a unique hyperbolic element $\gam_{u, v, w} \in \PSL_2\C$, such that $u$ is the repelling fixed point, $v$ is the attracting fixed point of $\gam_{u, v, w}$ and that  $\gam_{u, v, w}$ can be conjugated to the hyperbolic element $z \mapsto w z$  by an element of $\PSL_2\C$.
Clearly, this mapping $(u, v) \mapsto \gam_{u, v, w}$ is a biholomorphic mapping onto its image.
Then, $(\rho, u, v) \in R_\ell$ biholomorphically corresponds to a unique element $(\rho, \gam_{u, v, w})$ of $\RR \times \PSL_2\C$. 
Thus $R_\ell \to \RR \times \PSL_2\C$  is a biholomorphic map onto its image.
Since $\PSL_2\C \cong \operatorname{SO}_3(\C) \sub \C^9$, we see that $R_\ell$ is biholomorphic to a closed analytic set in a complex vector space of finite dimension. 
(It is closed, since if $(u, v) \in (\CP^1)^2 \minus \Delta$ converges to a point in the diagonal $\Delta$, then $\gam_{u, v, w}$ must leaves every compact subset of $\PSL_2\C$.) 
Therefore $R_\ell$ is also a Stein space, as it is a closed analytic subset of a Stein space. 

The theory of categorical quotients of Stein manifolds has been developed analogously to GIT-quotients of affine algebraic varieties (see \cite{Snow82}). 
Let $X_\ell$ be the categorical quotient ({\it Stein quotient}) $R_\ell\sslash\PSL_2\C$,  which is again Stein.
In this quotient,  two framed representations $(\rho_1, u_1, v_1)$ and $(\rho_2, u_2, v_2)$ in $R_\ell$ are identified if and only if every $\PSL_2\C$-invariant analytic function $f$ on  $R_\ell$ takes the same value at $(\rho_1, u_1, v_1)$ and $(\rho_2, u_2, v_2)$; see \cite[\S 3]{Snow82}. 
We denote, by $[\rho, u, v] \in X_\ell$, the equivalence class of $(\rho, u, v)$.

Next suppose that $\ell$ has weight $\pi$ modulo $2\pi$. 
In this case, the ordering of the framing $(u, v)$ will {\it not} affect the complexified bending map, and thus we take a slightly stronger quotient.
Then, let $\gam_{u, v}$ be the elliptic element of angle $\pi$ with the axes connecting $u$ and $v$. 
Let $R_\ell/ \Z_2$ be the quotient of $R_\ell$ by the $\Z_2$-action which switches the ordering of the framing, namely, given by $(\rho, u, v) \mapsto (\rho, v, u)$.
Consider the map $R_\ell/\Z_2 \to \RR \times \PSL_2\C$ defined by $(\rho, u, v) \mapsto (\rho, \gam_{u, v})$.
Thus $R_\ell/ \Z_2$ is biholomorphic to a closed analytic set in $\RR \times \PSL_2\C$.
Similarly, we let $X_\ell$ be the Stein quotient $(R_\ell / \Z_2) \sslash \PSL_2\C$.
The action of $\GG_\ell$ on $R_\ell$ descends to an action on $X_\ell$.
   
\subsubsection{Coordinates for the quotient space of representations framed along a single loop}\Label{sCoordinateRing}
We defined the Stein space $X_\ell$ as a Stein quotient.  In this section, we indeed realize $X_\ell$ as an analytic set in an affine space by identifying it with a subset of a $\PSL_2\C$-character variety $\rchi(\pi_1(S) \ast\Z)$ of $\pi_1(S) \ast \Z$.
Recall that, for $(\rho, u, v) \in R_\ell$, the element $\gam_{u, v, w} \in \PSL_2\C$ is a certain hyperbolic element if the weight of the loop $\ell$ is {\it not} equal to $\pi$ modulo $2\pi$ and a certain elliptic element of angle $\pi$ otherwise.

Given $(\rho, u, v) \in R_\ell$, let $\hat\rho = \hat\rho_{u,v,w}$ be the homomorphism from the free product $\pi_1(S) \ast \Z$ to $\PSL_2\C$, such that 
 every $\gam \in \pi_1(S)$ maps to $\rho(\gam)$ and $1 \in \Z$ maps to $\gam_{u, v, w}$.  
Then, with respect to the $\PSL_2\C$-action on $R_\ell$, we clearly have the following. 
\begin{lemma}
\begin{enumerate}
\item Suppose that the weight of $\ell$ is not equal to $\pi$ modulo $2\pi$.
Then $(\rho_1, u_1, v_1)$ and $(\rho_2, u_2, v_2)$ are identified by an element of $\PSL_2\C$ if and only if $\hat{\rho}_1$ and $\hat\rho_2$ are conjugate by $\PSL_2\C$. 
\item Suppose that the weight of $\ell$ is equal to $\pi$ modulo $2\pi$.
Then $(\rho_1, u_1, v_1)$ and $(\rho_2, u_2, v_2)$ are identified by an element of $\PSL_2\C \times \Z_2$ if and only if $\hat{\rho}_1$ and $\hat\rho_2$ are conjugate by $\PSL_2\C$, where the $\Z_2$-action exchanges the ordering of the framing. 

\end{enumerate}\end{lemma}
 
Let $\hat\RR$ be the space of representations $\pi_1(S) \ast \Z \to \PSL_2\C$.
 Then $\hat\RR$ is an affine algebraic variety. 
Suppose that the weight of $\ell$ is {\it not} equal to $\pi$ modulo $2\pi$. 
 We have seen that the mapping $R_\ell \to \RR \times \PSL_2\C$ is a biholomorphic map onto its image by the mapping $(\rho, u, v) \mapsto \hat\rho$.
Let  $\hat\RR_\ell$ be this image. 
Then $\hat\RR_\ell$ is the closed analytic subset in $\hat\RR$ biholomorphic to $R_\ell$, and thus in particular it is Stein. 
Moreover, this biholomorphism $R_\ell \to \hat\RR_\ell$ is equivariant with respect to the $\PSL_2\C$-action. 
Thus the Stein space $X_\ell = R_\ell \sslash \PSL_2\C$ is biholomorphic to the subvariety $ \hat\RR_\ell \sslash \PSL_2\C$ of $\rchi(\pi_1(S) \ast \Z)$. 

A similar identification holds in the case when $\ell$ has weight $\pi$ modulo $2\pi$.  
The Stein space $R_\ell / \Z_2$ biholomorphically maps to its image, denoted by $\hat\RR_\ell$, in $\hat\RR$ by  the mapping $(\rho, u, v) \mapsto \hat\rho$.
Then $X_\ell  = (R_\ell / \Z_2) \sslash \PSL_2\C$ is biholomorphic to the Stein space $\hat\RR_\ell\sslash \PSL_2\C$.

Let $\gam \in \pi_1(S) \ast \Z$. 
Let $\tr^2(\gam)$ be the (polynomial) function on $\hat{\RR}_\ell$ defined by $(\rho, u, v ) \mapsto \tr^2 \rho(\gam).$
Then $\tr^2(\gam)$ is a $\PSL_2\C$-equivariant analytic function on $\hat{\RR}_\ell$.  
Then,  by \cite[Corollary 2.3]{HeusenerPorti04}, such trace square functions form coordinates of the Stein quotient $\hat\RR_\ell\sslash \PSL_2\C$, and they also form coordinates for $X_\ell \,(\cong \hat\RR_\ell\sslash \PSL_2\C)$.
Therefore we have the following.
\begin{proposition}\Label{CoordinateRing}
There are finitely many elements $\gam_1, \gam_2, \dots \gam_N$ in $\pi_1 (S) \ast \Z$, such that  the analytic mapping $\hat{\RR}_\ell \to \C^N$ given by $\tr^2(\gam_1), \tr^2(\gam_2) \linebreak[2] \dots, \tr^2(\gam_N)$ induces an analytic embedding of $X_\ell$ into $\C^N$. 
Thus $\tr^2(\gam_1), \linebreak[2] \tr^2(\gam_2) \dots, \tr^2(\gam_N)$ form a coordinate ring.
\end{proposition}

\subsection{Representations framed along a multi-loop} \Label{sRepresentationsFramedAlongMultiloop}
In \S \ref{sRepresentataionsFramedAlongALoop}, we introduced the space of representations $\pi_1(S) \to \PSL_2\C$ framed along a single (oriented) loop,  constructed a quotient space by the $\PSL_2\C$ action, and realized 
 as an analytic subset of a complex affine space. 
In this section, we similarly consider the space of representations framed along a weighted multiloop, and then construct its Stein quotient by the action of $\PSL_2\C$. 

Let $m_1, \dots m_n$ be non-isotopic essential simple closed curves on $S$, and let $M$ be their union $m_1 \sqcup m_2 \sqcup \dots \sqcup m_n.$
Recall that  $\RR$ denotes the representation variety $\{ \pi_1(S) \to \PSL_2\C\}$.  
For each $i = 1, \dots, n$, pick a representative $\alpha_i \in \pi_1(S)$ representing $m_i$.   
Then, consider the  space $R_M$ of tuples $(\rho, (u_i, v_i)_{i = 1}^n)  \in R  \times (\CP^1)^{2n}$ where
\begin{itemize}
\item $\rho \in R$ is a homomorphism $\pi_1(S) \to \PSL_2\C$, and
\item $u_i, v_i \in \CP^1$ are different fixed points of   $\rho(\alpha_i)$ for $i = 1, \dots, n$.
\end{itemize}
As in the case of a single loop,  $\rho(\alpha_i)$ are {\it not} parabolic elements (but can be the identity).
Then $R_M$ is a closed analytic subvariety of $R  \times (\CP^1 \times \CP^1 \minus \Delta)^n$, where $\Delta$ denotes the diagonal as before. 
Given $(\rho, (u_i, v_i)_{i = 1}^n) \in R_M$, we can equivariantly extend $(u_i, v_i)$ to  the pairs of fixed points for all representatives of $m_1, \dots, m_n$ in $\pi_1(S)$. 
We call this extension a {\sf framing} of $\rho$ along the multiloop $M$. 

\subsubsection{Framed character varieties}
Now we assign a positive number ({\it weight}) to each loop of $M$.
Let $p$ be the number of components $m_i$ of $M$, such that the weight of $m_i$ is $\pi$ modulo $2\pi$.  
Without loss of generality, we can assume $m_1, \dots, m_n$ are the loops of $M$ with weight $\pi$ modulo $2\pi$. 
Then, $\Z_2^p$ acts biholomorphically on $R_M$ by switching the ordering of the fixed points of the framing along  $m_1, \dots, m_p$.
Note that this $\Z_2^p$-action has no fixed points in $R_M$.

Fix a complex number $w \in \C$ with $\lvert w \rvert > 1$.
As in \S \ref{sCoordinateRing}, let $\gam_{u_i, v_i, w} \in \PSL_2\C$ be, if the weight of $m_i$ is $\pi$ modulo $2\pi$, then the elliptic element of angle $\pi$ whose axis is the geodesic connecting $u_i$ to $v_i$, and otherwise, the hyperbolic element with the repelling fixed point $u_i$ and the attracting fixed point $v_i$ such that  $\gam_{u_i, v_i, w}$ is conjugate to the dilation $z \mapsto w z$. 
Then, define the mapping $R_M \to \RR \times (\PSL_2\C)^m$ by $(\rho, (u_i, v_i)_{i = 1}^n) \mapsto (\rho, (\gam_{u_i, v_i, w})_{i = 1}^n)$.
This mapping takes $R_M/ \Z_2^p$ onto its image $\hat{R}_M$ biholomorphically.   
Thus $R_M/ \Z_2^p$ is a closed analytic set in a finite-dimensional complex vector space. 
Therefore $R_M/ \Z_2^p$ is Stein.
The Lie group $\PSL_2\C$ acts analytically on $R_M/ \Z_2^p$, by conjugation on $\rho$. 
By this action, we obtain its Stein quotient $(R_M/ \Z_2^p)\sslash \PSL_2\C \eqqcolon X_M$. 
Thus $X_M$ is a Stein space.

The biholomorphic map $R_M/ \Z_2^p \to \hat{R}_M$ is equivariant w.r.t. the $\PSL_2\C$-action, $X_M$ is biholomorphic to the corresponding Stein quotient $\hat{R}_M\sslash \PSL_2\C$.

We denote, by $[\rho, (u_i, v_i)]$, the equivalence class of $(\rho, (u_i, v_i)) \in R_M$ in $X_M$.
 The subgroup $\mathcal{G}_M$ of $\MCG$ acts on $R_M$, and descends to an action on $X_M$.
 \subsubsection{Coordinates of the quotient space of representations framed along a multiloop}
 
Let $g_1, g_2, \dots, g_n$ be a standard generating set of the free group $\mathbb{F}^n$ of rank $n$, so that there are no relators. 
Every $(\rho, (u_i, v_i)_{i = 1}^n) \in R_M$ corresponds to a unique representation  $\pi_1(S) \ast \mathbb{F}^n \to \PSL_2\C$ such that
  \begin{itemize}
  \item $\gam \in \pi_1(S)$ maps to $\rho(\gam)$, and
  \item $g_i$ maps to $\gam_{u_i, v_i, w}$ for every $i = 1, \dots, n$.
\end{itemize}
By this correspondence, $R_M$ analytically embed into the space of representations  $\pi_1(S) \ast \F^n \to \PSL_2\C$.
As in \S \ref{sCoordinateRing}, by the quotient of the image $\RR_M$  by $\PSL_2\C$,  \cite[Corollary 2.3]{HeusenerPorti04} yields the coordinate ring of $X_M \cong \RR_M\sslash \PSL_2\C$. 
\begin{proposition}
There are finitely many elements $\gam_1, \gam_2, \dots, \gam_N$ of $\pi_1(S) \ast \mathbb{F}^n$, such that $\tr^2(\gam_1), \dots \tr^2 (\gam_N)$ form a coordinate ring of $X_M$.
\end{proposition}

  \subsection{The projection from the framed character variety to the character variety}\Label{sForgetfulMap}
  
In this subsection, we explain the relation between the famed character variety $X_M$ and the original character variety $\rchi$. 
Let  $\rchi_M^p$ be the subvariety of $\rchi$ consisting of representations $\rho\col \pi_1(S) \to \PSL_2\C$, such that at least one loop of $M$ corresponds to a parabolic or the identity element of $\PSL_2\C$. 
Let $X_M^p$ be the subvariety of $X_M$ whose representations are in $\rchi_M^p$. 
Every representation $\rho\col \pi_1(S) \to \PSL_2\C$ in $\rchi \minus \rchi_M$ has $2^N$ choices for a framing along $M$,  where are exactly $N$ is the number of components of $M$.
Then the projection map from $  \minus X_M^p$ to $\rchi \minus \rchi_M^p$ is a finite holomorphic covering map, and the covering degree is $2^N$.
Therefore, by the Removable Singularity Theorem (\Cref{RemovableSingularity}), $X_M \to \rchi$  is a $\C$-analytic branched covering map.

  \section{Bending a surface group representation into $\PSL_2\C$ inside the representation space into $\PSL_2\C \times \PSL_2\C$}\Label{sComplexifiedBending}
 Originally, bending deformation equivariantly bends a totally geodesic $\H^2$ along a measured lamination (\cite{Thurston-78, Epstein-Marden-87}), so that bending is in one direction and the bent $\H^2$ is locally convex.
 Moreover, one can extend it to an equivariant bending pleated surface along the pleated locus using bending cocycles (\cite{Bonahon96ShearingBendingSymplecticForm}). 
In both cases, bending is done along (bi-infinite) geodesics in $\H^3$ which are embedded in the pleated surfaces.  

In this section, we introduce a certain bending deformation of more general equivariant surfaces in $\H^3$. 
Using such a more general bending, define a complex-analytic bending map $X_M \to \rchi \times \rchi$ which complexifies the real-analytic bending map $\TT \to \rchi$.

\subsection{A complexification of the Lie group $\PSL_2\C$ regarded as a real Lie group.}
We first recall a complexification of $\PSL_2\C$ when regarded as a real Lie group. 
\begin{proposition}[See Proposition 1.39 in  \cite{Ziller10LieGroupsRepresentationTheory}  for example]\Label{ComplexifiedLieAlgebra} 
Regard $\mathfrak{psl}_2\C$ as a real Lie algebra.
Then the complexification of the Lie algebra $\mathfrak{psl}_2\C$ is isomorphic to $\mathfrak{psl}_2\C \oplus (\mathfrak{psl}_2\C)^\ast$  by the mapping given by  $(u, 0) \mapsto (u, I u)$ and $(0, v) \mapsto (v, -I v)$, where $(\mathfrak{psl}_2\C)^\ast$ is the complex conjugate of $\mathfrak{psl}_2\C$ and  $I$ is the complex multiplication of $\mathfrak{psl}_2\C$.
\end{proposition} 
As we discussed in \S \ref{sComplexBendingVarieties},  
we regard  $\PSL_2\C$ as a real Lie group, and we complexify $\PSL_2\C$ by
\begin{gather*}
\function{c}{\PSL_2\C}{ \PSL_2\C \times \PSL_2\C}{Y}{(Y, Y^\ast)},  \\
\end{gather*} where $Y^\ast$ denote the complex conjugate of $Y$, 
so that it corresponds to \Cref{ComplexifiedLieAlgebra}. 
Then $c$ is holomorphic in the first factor and anti-holomorphic in the second factor.
Thus $c$ is, in particular, a proper real-analytic embedding of $\PSL(2,\C)$ into the complex Lie group $\PSL_2\C \times \PSL_2\C$.

\subsection{Bending framed representations}\Label{sBendingFramedRepresentataions}
We first define a complex bending of representations framed along a single loop. 
Let $\ell$ be a loop on $S$, and we fix a weight $w > 0$ of $\ell$. 
Fix $\alpha \in \pi_1(S)$ representing $\ell$.
Let 
$[(\rho, u, v)] \in X_\ell$, where  $\rho\col \pi_1(S) \to \PSL_2\C$ and $u, v$ are distinct
 fixed points of $\rho(\alpha)$. 
Let $(\rho, \rho)\col \pi_1(S) \to \PSL_2\C \times \PSL_2\C$  denote the diagonal representation given by $\gam \mapsto (\rho(\gam), \rho(\gam))$.  

Recall that $(u, v)$ generates a  $\rho$-equivariant  framing $f$ along $\ell$ and $\Lam_\ell$ denotes the subset of $\pi_1(S)$ corresponding to $\ell$. 
That is,  for every element $\gam \in \Lam_\ell$, an ordered pair $(u_\gam, v_\gam) \in  \CP^1 \times \CP^1$ of different fixed points of $\rho(\gam)$ is assigned $\rho$-equivariantly. 
Consider the oriented geodesic $g_\gam = (u_\gam, v_\gam)$ in $\H^3$ connecting $u_\gam$ to $v_\gam$ for all $\gam \in \Lam_\ell$. 
Those equivariant geodesics $\{g_\gam\}$ will be the axes of the bending.

First, we define the bending of $[(\rho, u,   v)] \in X_\ell$ alogn $\ell$. 
Pick a $\rho$-equivariant surface $\Sigma\col \ti{S} \to \H^3$ which is piecewise-embedding.
For instance, pick a triangulation of $S$, which induces a $\pi_1(S)$-invariant triangulation of $\ti{S}$; first construct a $\rho$-equivariant map $\Sigma^1$ on the one-skeleton of the triangulation on $\ti{S}$, such that each edge of the triangulation is embedded in $\H^3$ as a geodessic segment.
 Then equivariantly extend $\Sigma^1$ to the interiors of the triangles as totally geodesic triangles in $\H^3$.

 Let $\ti\ell$ be the lift of $\ell$ to the universal cover $\ti{S}$ preserved by $\alpha \in \Lam_\ell $. 
 Then we may, in addition, assume that
 \begin{enumerate}
 \item $\Sigma$ is an embedding at all points on $\ti\ell$ (possibly) except at the locally finitely many points  (corresponding to finitely many points on $\ell$);
 \item  $\Sigma$ takes $\ti{\ell}$ into the bi-infinite geodesic $(u, v)$ connecting $u$ to $v$;
 \item there is a segment $s$ in $\ti\ell$ embedded by $\Sigma$ such that the orientation of the restriction of $\Sigma$ to $s$ agrees with the orientation of the oriented geodesic $(u, v)$.
\end{enumerate}

Indeed, in the construction of $\Sigma$ above,  we can start with a triangulation of $S$ such that $\ell$ is embedded in the one-skeleton of the triangulation, and we can first construct $\Sigma^1$ which takes $\ti\ell$ to the oriented geodesic $(u, v)$. 
Moreover, one can modify this mapping $\Sigma^1$ so that it satisfies the conditions (1) and (3), if necessary, by refining the triangulation and making $\Sigma^1 \vert \ti\ell$ non-injective and making adjacent trigngles along an edge segment in $\ti\ell$ injectively map into the different sides of the geodesic $(u, v)$ in a hyperbolic plane in $\H^3$.

\begin{figure}
\begin{overpic}[scale=.03
] {PositivelyBendingDirection} 
    \put(10 , 24 ){$\H^3$}  
    \put(11.5 , 14.5 ){$\Sigma(s)$}  
    \put(40 , 24 ){$\H^3$}  
    \put(73 , 24 ){$\H^3$}  
            \put(55 ,  25 ){$u$}  
         \put(43 , 0 ){$v$}  
        \put(88 ,  25 ){$u$}  
         \put(75 , 0 ){$v$}  
         \put(79 , 20 ){\contour{white}{\textcolor{darkgray}{Exterior}} } 
              \put(8 ,  10){$H$}  
      \end{overpic}
\caption{Illustration of the bending direction. Left: The oriented hyperbolic plane containing the oriented geodesic $(u, v)$. Middle: The normal direction of the hyperbolic plane. Right: A positive bending deformation of the oriented hyperbolic plane along $(u, v)$.}\Label{fBendingDirection}  \end{figure}

 For every  $\theta \in \R$, we can equivariantly bend  $\Sigma$ along the equivariant oriented axes $\{g_\gam\}$ by angle $\theta$. 
 Then we can accordingly bend the representation $\rho\col \pi_1(S) \to \PSL_2\C$ so that the bent surface is equivariant via the bent representation.

We describe this bending construction more precisely.
 The (positive) bending direction given by the orientation of the hyperbolic three-space $\H^3$, the orientation of the loop $\ell$, and the oriented geodesic $(u, v)$. 

By the second condition, the orientation of a segment $s$ of $\ti\ell$ matches with the orientation of the oriented geodesic $(u, v)$  by $\Sigma$.
Let $H$ be the oriented hyperbolic plane $H$ in $\H^3$ containing the geodesic segment $\Sigma(s)$ (and thus the oriented geodesic $(u, v)$) so that the oriented tangent plane of $\Sigma$ at a point in $s$ is identified with $H$  (\Cref{fBendingDirection}, Right). 
The normal direction of $H$ is given by the orientation of $\H^3$ and the orientation of $H$ (\Cref{fBendingDirection}, Middle).
The bending of an oriented hyperbolic plane $H$ along the oriented geodesic $(u, v)$ is in the {\it positive direction}, if the normal direction of the hyperbolic plane is in the exterior side of (small) bending (\Cref{fBendingDirection}, Right).

 We can equivariantly bend $\Sigma\col \ti{S} \to \H^3$ along every lift of $\ell$ to $\ti{S}$ by angle $w$ as follows. 
Let $\Phi\col \ti{S} \to S$ denote the universal covering map.
Then  the surface $\Sigma'$ obtained by bending $\Sigma$ along $\ell$ by angle $w$ is the mapping $\ti{S} \to \H^3$ such that
\begin{itemize}
 \item $\Sigma'$ coincides with  $\Sigma$ on each component of $\ti{S} \minus \Phi^{-1} (\ell)$ up to an element of $\PSL_2\C = \Isom^+ \H^3$, and
\item for every lift $\ti\ell$ of $\ell$ to $\ti{S}$, letting  $P, Q$ be adjacent components of  $\ti{S} \minus \Phi^{-1} (\ell)$ accross $\ti\ell$, if we notmalize $\Sigma'$ so that $\Sigma = \Sigma'$ on $P$ by $\PSL_2\C$, then $\Sigma' \vert Q$ is obtained from $\Sigma \vert Q$ by postcomposing the elliptic element about the oriented geodesic containing $\Sigma(\ti\ell)$ by angle $w$, where the positive direction is given as above by the support tangent plain of $\Sigma$ at a point in a approprite segment $s$ of $\ti\ell$ satisfying (2). 
\end{itemize}
This bent surface $\Sigma'\col \ti{S} \to \H'^3$ is uniquely defined up to postcomposing with an element in $\PSL(2, \C)$ (as $\Sigma$ is). 

This bending construction of $\Sigma'$ form $\Sigma$ along $(\ell, w)$ induces a compatible representation $\rho' \coloneqq b_{\ell, w} (\rho, u, v)$ by ``bending'' the framed representation $[\rho, u, v]$ along $\ell$ by angle $w$, so that $\Sigma'$ is $\rho'$-equivariant and $\rho'$ is a unique up to conjugation by $\PSL_2\C$.
Namely, the {\it bent representation} $\rho'$ is a representation satisfying the following:
\begin{enumerate}[label=(\roman*)]
\item For every component $F$ of $S \minus \ell$, the restriction of $\rho$ to $\pi_1(F)$ coincides with the restreiction of $\rho'$ to $\pi_1(F)$ up to $\PSL(2, \C)$. 
\item For every lift $\ti\ell$ of $\ell$, letting $P, Q$  be adjacent components of  $\ti{S} \minus \Phi^{-1} (\ell)$ accross $\ti\ell$ and letting $\Gamma_P$ and $\Gam_Q$ denote the subgroup of $\pi_1(S)$ preserving $P$ and $Q$, respectively, 
if we normalize so that $\rho = \rho'$ on $\Gamma_P$, then $\rho' \vert \Gamma_Q$ is obtained from $\rho \vert \Gamma_Q$ by conjugating by the ellipeitc element of angle $w$, where the direction of the rotation coincides with the bending direction describe above.
\end{enumerate}

Next, we show that this bent representation $b_{\ell, w} (\rho, u, v)\col \pi_1(S) \to \PSL_2\C$ is independent of the choice of the $\rho$-equivariant map $\Sigma$, in order to provide the uniqueness.  
 In fact, in the case that the restriction of $\rho$ to $\pi_1(F)$ is irreducible for every component $F$ of $S\minus \ell$ (only one or two components), we can define the bending of $(\rho, u, v)$ along $\ell$ without picking a $\rho$-equivariant surface $\Sigma$. 
In order to define the bending direction, we can simply identify the $\ti\ell$ with the oriented geodesic $(u, v)$ respecting the orientation. 
We embed a regular neighborhood of $\ti\ell$ into $\H^3$, extending this identification. 
We can define the same positive bending direction via this (simpler) identification. 
 Indeed, for such generic representations, the same bent representation $\rho'\col \pi_1(S) \to \PSL_2\C$  is defined only with the conditions (i), (ii) without a $\rho$-equivariant mapping $\Sigma$, since the stabilizer of the restriction $\rho$ to $\pi_1(F)$ is trivial for every component $F$ fo $S \minus \ell$.  

 For general $[(\rho, u, v)] \in X_\ell$, we first pick a $\rho$-equivariant surface $\Sigma\col \ti{S} \to \H^3$ that satisfies the properties (1), (2), (3). 
Then, pick a sequence  $[(\rho^k, u^k, v^k)] \in X_\ell$, which converges to $[(\rho, u, v)] \in X_\ell$ as $k \to \infty$ such that the restriction of $\rho^k$ to $\pi_1(F)$ is irreducible for every connected component  $F$ of $S \minus \ell$.
 In addition, for each $k = 1,2, \dots$,  we can take a piecewise-smooth $\rho^k$-equivariant surface $\Sigma^k\col \ti{S} \to \H^3$ that satisfies the properties (1), (2), (3), such that $\Sigma^k$ converges to $\Sigma$ as $k \to \infty$. 
By continuity, the uniqueness of the bending of $[(\rho^k, u^k, v^k)]$ along $(\ell, w)$ implies the uniqunnes of bending of  $[(\rho, u, v)]$ along $(\ell, w)$, as desired.

  We denote the representation $\pi_1(S) \to \PSL_2\C$ obtained by bending $[\rho, u, v]$ along $\ell$ by angle $\theta$ by $b_{\ell, \theta}(\rho, u, v)$.    
 Note that, if we reverse the order of $u$ and $v$ of the framing $(u, v)$, then the positive bending direction is reversed.

We now define a complex bending map $B_\ell \col X_\ell \to \rchi \times \rchi$ by $B_{\ell, w}(\rho, u, v) = ( b_{\ell, w}(\rho, u, v),   b_{\ell, -w}(\rho, u, v) ) $.
Note that, in the first factor and the second factor, the bending $\rho$ is equivariantly done along the same axes and by the same angle,  but in opposite directions (\Cref{fBendingOppositely}).
\begin{figure}
\begin{overpic}[scale=.04
] {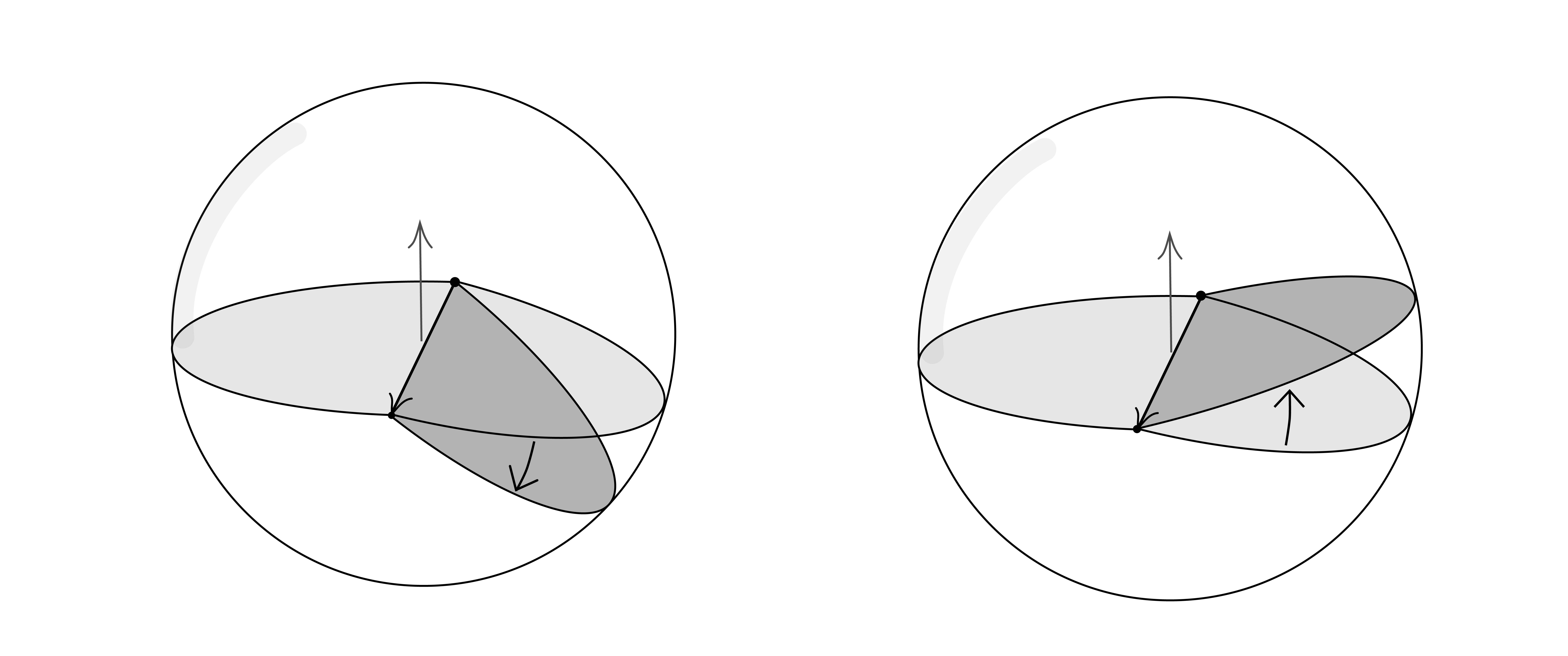} 
   \put(25 ,30 ){$\H^3$}  
   \put(20 ,18){$\H^2$}  
      \put(65 ,18){$\H^2$}  
   \put(70 , 29 ){$\H^3$}  
 \put(23 , 13 ){$v$}  
 \put(29 , 25 ){$u$}  
 \put(71 , 12 ){$v$}  
 \put(75.5 , 24 ){$u$}  
        \put(15,2){\contour{white}{Positive bending}}
\put(60,2){\contour{white}{Negative bending}}

      \end{overpic}
\caption{The complex bending map bends a representation in opposite directions in different factors.}\Label{fBendingOppositely}
\end{figure}

The bent representation is well-defined up to conjugation by an element of $\PSL_2\C \times \PSL_2\C$, and thus  $B_\ell (\rho, u, v) \in \rchi \times \rchi$ is well-defined.
We remark that, if $\rho\col \pi_1(S) \to \PSL_2\C$ is Fuchsian, then the representation of $B_\ell (\rho, u, v)$ in the first factor $\rchi$ is the complex conjugate of that in the second factor (given by complex conjugate of matrix entries as explained in \S\ref{sComplexBendingVarieties}).

For a weighted multiloop $M$ on $S$, we can similarly define the complex bending map $B_M \col X_M \to \rchi \times \rchi$ as follows.
Let $m_1, \dots, m_n$ be the weighted loops of $M$.
Pick $\gam_i \in  \pi_1(S)$ reprersenting $m_i$. 
Let $\ti{m}_i$ be a $\gam_i$ invariant lift of $m_i$ to the universal cover $\ti{S}$.
Let  $[\rho, (u_i, v_i)_{i = 1}^n] \in X_M$, where $(u_i, v_i)$ be the fixed point of $\rho(\gam_i)$. 
Then the oriented geodesic $g_i$ connecting  $u_i$ to $v_i$, equivariantly extends to a system of bending axes corresponding to all lifts of $m_i$ to $\ti{S}$. 
Find a $\rho$-equivariant surface $\Sigma\col \ti{S} \to \H^3$ such that $\ti{m}_i$ maps into its corresponding oriented axes $g_i$.

Let $\theta_1, \dots, \theta_n$ be real numbers. 
We can similarly bend the $\rho$-equivariant surface $\Sigma\col \ti{S} \to \H^3$ along the oriented geodesics $g_1, \dots, g_n$ and their orbit geodesics by angles $\theta_1, \dots, \theta_n$, respectively, in the positive bending direction (defined by the orientation of $\H^3$ and the orientations of the geodesics).
Since we bend $\Sigma$ in an equivariant manner,  the new bent surface $\Sigma^+\col \ti{S} \to \H^3$ is also equivariant via a unique representation.
We denote the bent representation by $$b_{(m_i, \theta_i)}(\rho, (u_i, v_i)_{i = 1}^n) = b_M^+(\rho, (u_i, v_i)_{i = 1}^n) \col  \pi_1(S) \to \PSL_2\C.$$
Similarly, we can bend $\Sigma$ along the same axes by the same angles but in opposite directions, and we obtain another bent surface $\Sigma^-\col \ti{S} \to \H^3$.
Then $\Sigma^-$is also equivariant via a unique representation $$b_M^-(\rho, (u_i, v_i)_{i = 1}^n) = b_{(m_i, -w_i)}(\rho, (u_i, v_i)_{i = 1}^n)\col \pi_1(S) \to \PSL_2\C.$$
By combining those two bending of framed representations, we obtain the bending map $B_M \col X_M \to \rchi \times \rchi$ by $$B_M(\rho, (u_i, v_i)_{i = 1}^n) = (b_{(m_i, w_i)}(\rho, (u_i, v_i)_{i = 1}^n), b_{(m_i, -w_i)_i}(\rho, (u_i, v_i)_{i = 1}^n)).$$
Then the mapping $\ti{S} \to \H^3 \times \H^3$ defined by $x \mapsto (\Sigma^+(x), \Sigma^-(x))$ is equivariant via $B_M(\rho, (u_i, v_i)_{i = 1}^n)\col \pi_1(S) \to \PSL_2 \times \PSL_2\C$.

\subsection{Equivariant property}

\begin{lemma}\Label{EquivariantCBending}
Let $M$ be a weighted multiloop on $S$.
Let $G_M$ be the subgroup of the mapping class group $\MCG(S)$, which preserves $M$. 
Then $B_M\col X_M \to \rchi \times \rchi$ is $G_M$-equivariant. 
\end{lemma}
\begin{proof}
Recall that $G_M$ acts on $X_M$ by marking change. 
Therefore $b_{(m_i, w_i)}\col X_M \to \rchi$ and $b_{(m_i, -w_i)}\col X_M \to \rchi$ are both $G_M$-equivariant, since the equivariant construction of those mappings respect the $G_M$-action. 
Hence $B_M$ is also $G_M$-equivariant. 
\end{proof}

\subsection{Support planes and spaces}\Label{sSupportSpaces}
For a marked hyperbolic surface $\tau$  homeomorphic to $S$ and a measured lamination $L$ on $\tau$, 
we have a $\pi_1(S)$-equivariant bending map $\beta_{\tau, L}\col \H^2 \to \H^3$ which is ``locally convex''. 
Letting $\ti{L}$ be the $\pi_1(S)$-invariant measured lamination on the universal cover $\H^2$ of $\tau$.
Then, for each component $P$ of $\H^2 \minus \ti{L}$, the mapping $\beta_{\tau, L}$ embeds $P$ into a totally geodesic hyperbolic plane $P$ in $\H^3$. 
Such a hyperbolic plane is a {\sf support plane} for $\beta_{L, \tau}$.  (See \cite{Epstein-Marden-87} for more general support planes.)
On the other hand,  this equivariant system $\{H_P\}_P$ of totally geodesic hyperbolic planes, indexed by the components, determines the original bending map $b_{\tau, L}\col \H^2 \to \H^3$. 

 In \S \ref{sBendingFramedRepresentataions}, we bend framed representations $\eta = [\rho, (u_i, v_i)]$ in $X_M$ along  a weighted multiloop $M$, and obtained a representation $\pi_1(S) \to \PSL_2\C \times \PSL_2\C$.
 As the symmetric space associated with $\PSL_2\C \times \PSL_2\C$ is the product $\H^3 \times \H^3$,   we consider a system of {\sf supporting hyperbolic three-spaces} in the product $\H^3 \times \H^3$ as follows. 
For every component $P$ of $\ti{S} \minus \ti{M}$, the restriction of  $\Sigma^+ $ to $P$ coincides with  the restriction of  $\Sigma^- $ to $P$ composed with an element $\gam$ of $\PSL_2\C$.
Therefore, the restriction of the surface $(\Sigma^+, \Sigma^-)\col \til{S} \to \H^3 \times \H^3$ to $P$ is contained in a totally geodesic copy $H_P$ of $\H^3$ given by $ \{ (x, \gam x ) \mid x \in \H^3 \} \subset \H^3 \times \H^3$. 

 Hence, we obtain an equivariant collection of supporting hyperbolic 3-spaces  $H_P$ for all components $P$ of $\ti{S} \minus \ti{M}$.
 We call this collection $\{H_P\}_P$ the {\sf (equivariant) bending support system} of $B_M$ at $\eta$.
 Let $G_P$ denote the subgroup of $\pi_1(S)$ consisting of the elements preserving the $P$. 
 Then $H_P$ is preserved by the restriction of the bent representation $$B_M(\rho, (u_i, v_i)_{i = 1}^n)\col \pi_1(S) \to \PSL_2\C \times \PSL_2\C$$ to the subgroup $G_P$.

Suppose that $P$ and $P'$ are adjacent components of $\ti{S} \minus \ti{M}$ across a lift $\ti{m}$ of a loop $m$ of $M$. 
Let $w$ be the weight of $m$. 
Then, in $\H^3 \times \H^3$, the support spaces $H_P$ and $H_{P'}$ intersect in a geodesic at angle $w$  ({\sf complex bending axis}), which corresponds to the bending axis in $\H^3$ induced by the framing in the definition of $B_M$ (\Cref{fSupportSpacesIntersecting}). 
 In particular, if the weight of $m$ is a multiple of $\pi$, then $H_P = H_{P'}$.
Indeed, for an elliptic element $e \in \PSL_2\C$ with rotation angle $\pi$, we have
 $$\{ (x, x) \in \H^3 \times \H^3 \mid x \in \H^3 \} = \{ (e x, e^{-1}x) \in \H^3 \times \H^3\,  \vert\, x \in \H^3)\}.$$

 \begin{figure}
\begin{overpic}[scale=.12, 
] {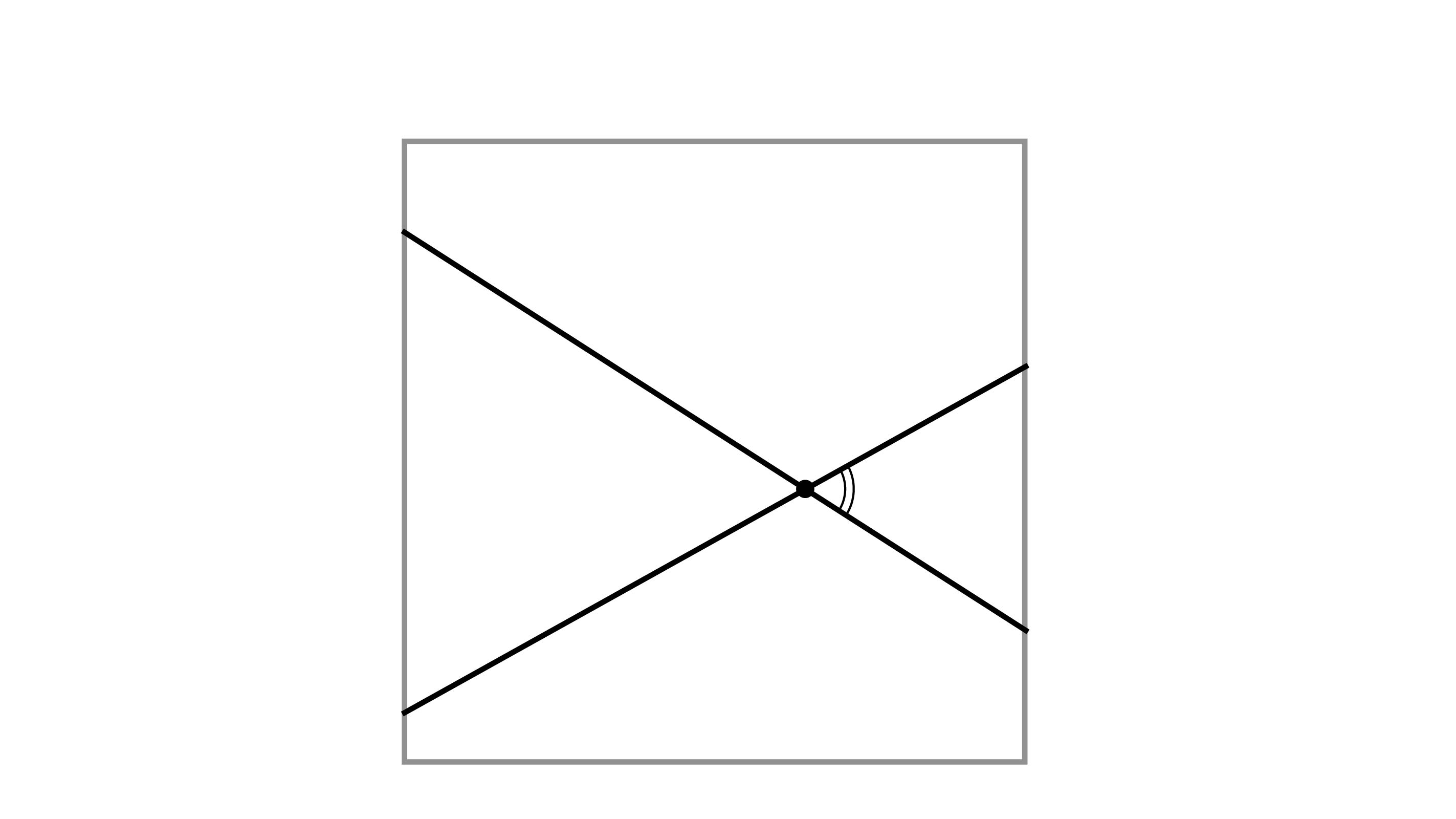} 
\put(45 , 50){\textcolor{darkgray}{$\H^3$}}  
\put(17 , 22){\textcolor{darkgray}{$\H^3$}}  
  \put(40 , 35 ){\textcolor{black}{$H_P$}}  
  \put(40 , 8 ){\textcolor{black}{$H_{P'}$}}  
   \put(60.5 , 20 ){\textcolor{black}{$w$}}  
      \end{overpic}
\caption{The intersection angle $w$ of totally geodesic copies $H_P, H_{P'}$ of $\H^3$ in $\H^3 \times \H^3$.}\Label{fSupportSpacesIntersecting}
\end{figure}

\begin{definition}
Let $\xi\col \pi_1(S) \to \PSL_2\C \times \PSL_2\C$ be a representation. 
A {\sf support system of}  $\xi$ with respect to $M$ is
an equivariant collection of  totally geodesic hyperbolic spaces $H_P$ in $\H^3 \times \H^3$ for all components $P$ of $\ti{S} \minus \ti{M}$ such that
the restriction of $\xi$ to $G_P$ preserves $H_P$ for all components $P$ of $\ti{S} \minus \ti{M}$.
\end{definition}
 
 In general, a representation $\pi_1(S) \to \PSL_2\C \times \PSL_2\C$ may have no support system or many support systems.
On the other hand, we will prove that the support system is uniquely determined by $B_M(\rho, (u_i, v_i)_{i = 1}^n)$ in most cases; see \Cref{UniqueSupportSpace}. 
   
\section{Complex bending maps are almost injective}\Label{sInjectivity}
In this section, we prove the injectivity of the complex bending map $B_M\col X_M \to \rchi \times \rchi$ when restricted to the complement of certain subvarieties. 

Let $M$ be a weighted multiloop on $S$, and let $n$ be the number of loops of $M$. 
Let $X_M^p$ be  the subset of $X_M$ consisting of the framed representations $(\rho, (u_i, v_i)_{i = 1}^n)$ such that $\tr^2 \rho(m) = 4$ for, at least, one loop $m$ of $M$, i.e. $\rho(m)$ is either a parabolic element or the identity. 
As it is an algebraic equation,  $X_M^p$ is an analytic subvariety of $X_M$. 

Let  $X_M^w$ be the subset consisting of the framed representations $(\rho, u, v)$ such that, for some component $F$ of $S \minus M$,  the restriction of $\rho$ to  $\pi_1(F)$ is {\it weakly reducible}, i.e. the image is, up to a finite index, reducible.  

 Given a complex Lie subgroup $G$ of $\PSL_2\C$, the set of all representations $\rho\col \pi_1(S)  \to G$ gives a subvariety of $\rchi$.  
 The reducible representations $\pi_1(S) \to \PSL_2\C$ form a subvariety of $\rchi$.
A representation $\rho\col \pi_1(S) \to \PSL_2\C$ is weakly reducible but not reducible, if and only if $\Im \rho$ preserves a pair of points on $\CP^1$ but it does not fix the pair pointwise. 
(If $\Im \rho$ preserves a triple of distinct points on $\CP^1$, then $\Im \rho$ is a finite group.)
Thus, the set of weakly reducible representations forms a subvariety of $\rchi$. 

Thus $X_M^w$ is also an analytic subset of $X_M$. 
We prove that the injectivity of the complex bending map holds in the complement of those analytic subsets. 
\begin{theorem}\Label{InjectivityAlomstEverywhere}
Let $M$ be a weighted oriented multiloop on $S$. 
Then, the complex bending map $B_M\col X_M \to \rchi \times \rchi$ is injective on $X_M \minus (X_M^p \cup X_M^w)$. 
\end{theorem}
We first show the uniqueness of the support systems of the complex bending.
\begin{lemma}\Label{UniqueSupportSpace}
Let $\eta \in X_M \minus (X_M^w \cup X_M^p)$.
Fix a representative $\xi\col \pi_1(S) \to \PSL_2\C \times \PSL_2\C$ of $B_M (\eta)$. 
Let $P$ be a component of $\ti{S} \minus \ti{M}$. 
Then, the support space  $H_P$ of $\xi$ is the unique totally geodesic copy of $\H^3$ in $\H^3 \times \H^3$ which contains the bending axes of the boundary components of $P$. 
\end{lemma}
\begin{proof}

As $\eta \notin X_M^P$, the bending axes of the boundary components of $P$ are uniquely determined by $\xi$. 
Let $G_P$ be the subgroup of $\pi_1(S)$ preserving $P$.  
Since $\eta \not\in X_M^\omega$, the restriction $\eta | G_P$ is  {\it strongly irreducible} (i.e. not weakly reducible).
Therefore one can prove that there is a unique totally geodesic copy of $\H^3$ in $\H^3 \times \H^3$, containing those bending axes, as follows: 

Set $B_M (\eta) = (\eta_1, \eta_2)\col \pi_1(S) \to \PSL_2\C \times \PSL_2\C$, where $\eta_1, \eta_2\col \pi_1(S) \to \PSL_2\C$.
Then, since $\eta_i | G_P$ is strongly irreducible, the $\PSL_2\C$-action on $\eta_i | G_P$ by conjugation has the trivial stabilizer for each $i = 1, 2$. 
 Suppose that there are two totally geodesic copies $\{ (x, \gam_1 x) \mid x \in \H^3 \}$ and   $\{ (x, \gam_2 x) \mid x \in \H^3\}$ of $\H^3$ in $\H^3 \times \H^3$ preserved by $\eta_i (G_P)$, where $\gam_1, \gam_2 \in \PSL_2 \C$.
 By the definition of the complexified bending map $B_M$, we have $\gam_1 \eta_1 \gam_1^{-1} = \eta_2$ and  $\gam_2 \eta_1 \gam_2^{-1} = \eta_2$ when restricted to $G_P$.
 Combining those equations, we have $ \gam_2^{-1}  \gam_1 \eta_1 \gam_1^{-1} \gam_2  = \eta_1$ on $G_P$. 
 Hence  $\gam_1 = \gam_2$, and the two copies of $\H^3$ coincide. 
\end{proof}
\Cref{UniqueSupportSpace} immediately implies the following. 
\begin{corollary}\Label{HolonomyDeterminesSupportSystem}
Suppose that $\eta_1, \eta_2 \in X_M \minus (X_M^p \cup X_M^w)$ satisfy $B_M (\eta_1) = B_M (\eta_2) \in \rchi \times \rchi$. 
Let $\xi\col \pi_1(S) \to \PSL_2 \C \times \PSL_2 \C$ be a representative of $B_M (\eta_1) = B_M (\eta_2)$. 
Then, the $\xi$-equivariant bending support system of $B_M$ at $\eta_1$ equivariantly coincides with that at $\eta_2$.
\end{corollary} 

\proof[Proof of Theorem \ref{InjectivityAlomstEverywhere}]

Suppose that $\eta_1, \eta_2 \in X_M \minus (X_M^p \cup X_M^w)$ map to the same representation $\pi_1(S) \to \PSL_2\C \times \PSL_2\C$ in $\rchi \times \rchi$ by $B_M$. 
Then, let $\xi\col \pi_1 (S) \to \PSL_2\C \times \PSL_2\C$ be a representative of their image. 

By 
\Cref{HolonomyDeterminesSupportSystem}, the support system of the bending of $\eta_1$ equivariantly coincides with that of $\eta_2$. 
Therefore $\eta_1$ and $\eta_2$ are obtained by unbending $\xi$ exactly in the same manner, and we obtain  $\eta_1 = \eta_2$ as follows:

    Let $\{H_P\}_P$ denote the support planes of $\xi$, where $P$ varies over all components $P$ of $\ti{S} \minus \ti{M}$.
    Recall that, if $P$ and $P'$ are adjacent components of $\ti{S} \minus \ti{M}$ along a lift of a loop $m$ of $M$, then $H_P$ and $H_P'$ intersect in a geodesic by the angle equal to the weight of $m$. 
    Take an abstract union $\cup_P H_P$ of the support 3-spaces $H_P$ obtained by gluing adjacent support spaces along the bending geodesic axes. 
    Then we have an $\xi$-equivariant mapping $\sigma \col \cup_P H_P \to \H^3 \times \H^3$ by the inclusions $H_P \sub \H^3 \times \H^3$. 
    Note that, letting $G_P$ be the subgroup of $\pi_1(S)$ preserving $P$ in $\ti{S}$, clearly $\xi (G_P)$ preserves $H_P$.
    
   Set   $\eta_1 = (\rho_1, (u_{1, i}, v_{1, i}))$ and   $\eta_2 = (\rho_1, (u_{2, i}, v_{2, i}))$.
Then, unbending $\sigma$ by $-M$, we have an equivariant mapping $\sigma(-M)\col \cup_P  H_P \to \H^3$, and $\xi$ is deformed to an representation of $\pi_1(S) \to \PSL_2\C.$
 This unbent representation must coincide with $\rho_1$ and $\rho_2$ up to conjugation by $\PSL_2\C$, due to the definition of $B_M$. 
 Moreover, since the endpoints of the bending axes correspond to the framing,  we see that $\eta_1 = \eta_2$. 
\Qed{InjectivityAlomstEverywhere}

\subsection{A non-injective example}\Label{sNonInjectiveExmaple}
We shall see, in an example, the non-injectivity of a complex bending map.
Let $m$ be a separating loop on $S$ with some positive weight. 
Pick a connected subsurface $F$ of $S$ bounded by $m$. 
Fix a homomorphism $\rho\col \pi_1(S) \to \PSL_2\C$ such that $\rho \vert \pi_1 F$ is the trivial representation. 
Then, as $\rho(m)$ is in particular the identity, any pair $(u, v) \in \CP^1 \times \CP^1$ is a framing of $\rho$ along $m$. 
\begin{lemma}\Label{NoninjectiveNonproper}
Fix an arbitrary orientation of $m$ and an arbitrary weight on $m$. 
Then 
 $B_m(\rho, (u,  v)) = (\rho, \rho) \in \rchi \times \rchi$ for all  framings $(u, v)$ along $m$. 
 In particular, $B_m$ is {\it not} injective and not proper. 
\end{lemma}
\begin{proof}
Pick a loop  $\ell$ on $S$ which essentially intersects $m$ exactly in two points (see Figure \ref{fNoninjectiveExmaple}). 
We can assume, without loss of generality, that the base point of $\pi_1(S)$ is on  $m$.
Let $\gam$ be an element of $\pi_1(S)$ corresponding to $\ell$. 
Then homotope $\ell$ so that $\ell$ is a composition of a loop $\ell_1$ on $S \minus F$ and a loop $\ell_2$ on $F$. 
Since  $\rho \vert \pi_1(F)$ is trivial, we have $B_m \eta (\gam_\ell) = B_m \eta (\gam_{\ell_1})$. 
We can take a generating set of $\pi_1(S)$ consisting of loops in $S \minus F$ and loops in $F$. 
Therefore $B_m (\rho, (u, v)) = (\rho, \rho)$ in $\rchi \times \rchi$. 

In particular, as $(u, v)$ may leave every compact set in $(\CP^1)^2$ minus the diagonal, $B_m (\rho, (u, v)) = (\rho, \rho)$ remains true. 
Therefore $B_\ell$ is non-proper. 
\end{proof}

\begin{figure}
\begin{overpic}[scale=.03
] {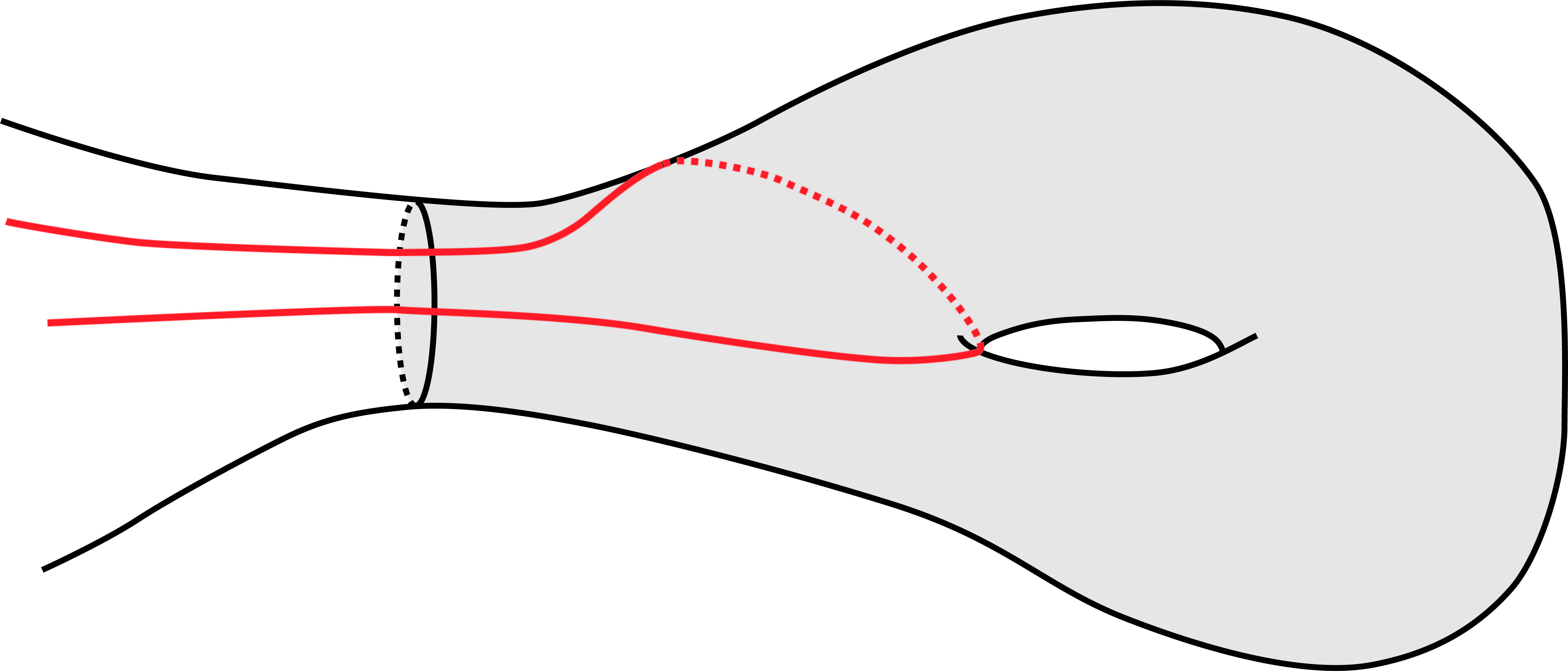} 
   \put( 70 ,28 ){\textcolor{black}{$F$}}  
   \put( 25 ,31.5 ){\textcolor{black}{$m$}}  
   \put( 45 ,22 ){\textcolor{RedOrange}{$\ell$}} 
      \end{overpic}
\caption{A loop $\ell$ essentially intersects a bending loop $m$ in two points.}\Label{fNoninjectiveExmaple}
\end{figure}
 
\section{Complex bending maps are almost proper}\Label{sAlmostProperComplexBendingMap}
 In this section, we prove the properness of the complex bending map, similarly to the injectivity in \S\ref{sInjectivity},  in the complement of certain proper subvarieties.  
 Similarly to $X_M^p$, 
 we let $\rchi_M^p$ be the subvariety of the $\PSL_2\C$-character variety $\rchi$  consisting of representations $\xi \col \pi_1(S) \to \PSL_2\C$ such that, for, at least, one loop $m$ of $M$, its holonomy  $\xi (m)$ is parabolic or the identity in $\PSL_2\C$. 
 Similarly to $X_M^w$, 
we let $\rchi_M^w$ be the subvariety of $\rchi$ such that, consisting of representations $\xi \col \pi_1(S) \to \PSL_2\C$ such that, for at least one component $F$ of $S \minus M$, $\xi | F$ is weakly reducible.      
     
\begin{theorem}\Label{PropernessAwayFromSubvariety}
The restriction of $B_M$ to $X_M \minus (X_M^p \cup X_M^w)$ is a proper mapping to $(\rchi \minus (\rchi_M^p \cup \rchi_M^w))^2$.
\end{theorem}

\proof
Let $\eta_i \in X_M \minus (X_M^P \cup X_M^w)$ be a sequence such that $B_M(\eta_i)$ converges to a representation in  $(\rchi \minus (\rchi_M^p \cap \rchi_M^w))^2$ as $i \to \infty$.
It suffices to show that $\eta_i$ also converges in $X_M \minus (X_M^P \cup X_M^w)$.

Pick a representative $\xi_i\col \pi_1(S) \to \PSL_2\C \times \PSL_2\C$ of $B_N(\eta_i)$ so that $\xi_i$ converges to $\xi\col \pi_1(S) \to \PSL_2\C \times \PSL_2\C$, so that its equivalence class $\xi$ is in $(\rchi \minus (\rchi_M^p \cap \rchi_M^w))^2$. 
Let $\{H_{i, P}\}$ be the $\xi_i$-equivariant bending support system of the complex bending of $\eta_i$ along $M$, where $P$ varies over all connected components of $\ti{S} \minus \ti{M}$.
By the hypothesis, the restriction of $\xi$ to each component of $S \minus M$ is strongly irreducible. 
Therefore, by \Cref{UniqueSupportSpace}, the $\xi_i$-equivariant  support system $\{H_{i, P}\}$ converges to a unique support system $\{H_P\}$ of $\xi$ as $i \to \infty$. 

We also show that the bending axes converge. 
\begin{claim}
The $\xi_i$-equivariant axis system for bending $\eta_i$ along $M$ in $\H^3 \times \H^3$ converges to a $\xi$-equivariant axis system as $i \to \infty$. 
\end{claim}
\begin{proof}
Let $m$ be a loop of $M$, and let $\ti{m}$ be a component of $\ti{M}$ which descends to $m$.
Let $\alpha \in \pi_1(S)$ denote the element preserving $\ti{m}$ such that the free homotopy class of $\alpha$ is  $m$.
Let $P, Q$ denote the adjacent components of $\ti{S} \minus \ti{M}$ separated by $\ti{m}$. 
Then $H_{i, P} \cap H_{i, Q}$ is the complex bending axis $g_{i, \ti{m}}$ for $\ti{m}$ in $\H^3 \times \H^3$, and also the axis of $\xi_i (\alpha)$.
The angle of the intersection of $H_{i, P}$ and $H_{i, Q}$ along the axis is equal to the weight of $m$.
As $\xi_i(m)$ converges to a non-parabolic element $\xi(m)$, the axis $H_{i, P} \cap H_{i, Q}$ converges to the axis of $\xi(\alpha)$ as $i \to \infty$.
\end{proof}

For each $i = 1,2, \dots$,
 let $\{g_{i, \ti{m}}\}$ denote the $\xi_i$-equivariant bending axis system in $\H^3 \times \H^3$ of $B_M$ at $\eta_i$.
Note that $\eta_i$ is obtained by unbending $\xi_i$ along the axes $g_{i, \ti{m}}$ by the angles given by the weights $M$. 
By the convergence, similarly unbending the limit $\xi$ in $(\rchi \minus (\rchi_M^p \cup \rchi_M^w))^2$ along the limit bending axis system by the angles given by $M$, we obtain the limit of $\eta_i$ as $i \to \infty$. 
As $\xi$ is in $(\rchi \minus (\rchi_M^p \cup \rchi_M^w))^2$, thus $\lim_{i \to \infty} \eta_i$ is contained in $X_M \minus (X_M^p \cup X_M^w)$.
\Qed{PropernessAwayFromSubvariety}

 \section{Analyticity of complex bending maps}\Label{sAnalyticity}
  
  \begin{theorem}\Label{BendingIsAnalytic}
  For every weighted oriented multiloop $M$ on $S$, the bending map
$B_M \col X_M \to \rchi \times \rchi$ is complex analytic. 
\end{theorem}
\proof
Recall that $X_M^p$ is the subvariety of the complex-analytic variety $X_M$ consisting of representations such that at least one loop of $M$ is parabolic, and also that $X_M^w$ is the subset of $X_M$ consisting of representations $\eta$ such that the restriction of $\eta$ to a component of $S \minus M$ is weakly reducible. 
 We have shown that the restriction of $B_M$ to $X_M \minus X_M^p \cup X_M^w$ is injective (\Cref{InjectivityAlomstEverywhere}).
 We first prove the assertion of \Cref{BendingIsAnalytic} for almost everywhere.
  \begin{lemma}\Label{AnalyticAlmostEverywhere}
  The restriction of $B_M$ to  $X_M \minus (X_M^p \cup X_M^w)$ is complex analytic.  
\end{lemma} 
\begin{proof}
Recall that $R_M$ is the space of representations framed along $M$, and that $R_M \sslash \PSL_2\C = X_M$. 
Let $R_M^p$ be the subset of $R_M$ consisting of framed representations, such that at least one loop of $M$ is parabolic (or the identity). 
Let $\eta= (\rho, (u_i, v_i)_{i = 1}^n)$ be an arbitrary framed representation in $R_M \minus (R_M^p \cup R_M^r)$, where $n$ is the number of the loops of $M$.
As the closed subvariety $R_M^p \cup R_M^r$ is $\PSL_2\C$-invariant, 
 we can take a $\PSL_2\C$-invariant open neighborhood $U$ of $\eta$ in $R_M \minus (R_M^p \cup R_M^r)$.
Then, for every framed representation $\zeta \in U$,  the stabilizer of $\zeta$ in $\PSL_2\C$ is a discrete group, since $\zeta$ is {\it not} in $R_M^r$,
Thus, if we take $U$ appropriately, $U$ is holomorphically a product of $\PSL_2\C$ and an open disk $D$.
Let $W$ be the image of $U$ in $X_M$. 
Then, we can biholomorphically identify $W$ in $X_M$ with $D$ in $U$ and take a holomorphic section $s\col W \to U$.

Pick any component of $Q$ of $\ti{S} \minus \ti{M}$, where $\ti{M}$ is the inverse image of $M$ in $\ti{S}$.
Let $G_Q$ be the stabilizer of $Q$ in $\pi_1(S)$. 
By $\C$-bending along $M$ (normalizing so that the restriction to $G_Q$ is unchanged), we obtain a holomorphic mapping $s(W) \to  (\RR \minus \RR_M^p \cup \RR_M^r)^2$ which is a lift of the restriction of $B_M$ to $W$.
Then, for every $\zeta \in s(W)$, its image by this mapping is a pair of strongly irreducible representations in $\RR$. 
Since $W$ is isomorphic to $s(W)$ and the quotient map from $\RR \times \RR$ to $\rchi \times \rchi$ is algebraic, the analyticity of $s(W) \to  (\RR \minus \RR_M^p \cup \RR_M^r)^2$ implies the analyticity of $B_M$ at the equivalent class of $\eta$ in $X_M$. 
\end{proof}

  By \Cref{AnalyticAlmostEverywhere},  $X_M \minus (X_M^p \cup X_M^w) \to (\rchi \minus \rchi^p_M \cup \rchi_M^w) \times (\rchi \minus \rchi^p_M \cup \rchi_M^w)$ is an injective analytic mapping. 
 Since $X_M^p \cup X_M^w$ is an analytic subvariety of $X_M$,  by the Removable Singularity Theorem (\Cref{RemovableSingularity}), the mapping $B_M \col X_M \to \rchi \times \rchi$ is analytic. 
\Qed{BendingIsAnalytic}

\section{The real-bending map sits in the complex-bending map}\Label{sCoplexification}
In this section, we observe that the complex-analytic bending map $B_M \col X_M \to \rchi \times \rchi$ is a natural extension of the real-analytic bending map $b_M\col \TT \to \rchi$.
Recall, from \S \ref{sComplexBendingVarieties}, that $\Delta^\ast$ is the twisted diagonal $\{ (\rho, \rho^\ast) \mid  \rho \in \rchi )\}$ and  $\psi \col \rchi \to  \Delta^\ast \sub \rchi \times \rchi$ is the embedding given by $\rho \mapsto (\rho, \rho^\ast ).$ \note{}

The forgetful map $X_M \to \rchi$ restricts to an analytic covering map $X_M \minus X_M^p \to \rchi \minus \rchi_M^p$ of degree $2^n$, where $n$ is the number of loops of $M$.
As the base surface $S$ is oriented, we let $\TT$ be the Teichmüller space of $S$ identified with a unique component of the real slice of  $\rchi$.
Since each loop of $M$ is oriented, there is a unique lift of $\TT$ to $X_M$:
Namely,  given a discrete faithful representation $\rho\col \pi_1(S) \to \PSL(2, \R)$ in $\TT$, for each (oriented) loop $m$  of $M$,  assign the reprelling fixed point of $\rho(m) \in \PSL_2\C$ in $\CP^1$ and the attacting fixed point of $\rho(m)$ in this order  as its framing.  
Let $\iota_M\col \TT \to X_M$ denote this real-analytic embedding.

\begin{proposition}\Label{ComplexifyingTwistedDiagonal}
Let $M$ be a weighted multiloop on $S$. 
Then, the restriction of $B_M$ to $\TT$ is a real-analytic embedding into the twisted diagonal $\Delta^\ast$ of $\rchi \times \rchi$, such that   $B_M \circ \iota_M$ coincides with $\psi \circ b_M\col \TT \to \rchi \times \rchi$.
\end{proposition}
\begin{proof}
Let $b_M^\ast\col \TT \to \rchi$ denote the complex conjugate of $b_M \col \TT \to \rchi$, i.e. the Fuchsian representation $\rho\col \pi_1(S) \to \PSL_2\R$ maps to the mapping taking $\gam \in \pi_1(S)$ to $(b_M(\rho)(\gam))^\ast \in \PSL_2\C$.
When applying the complex bending $B_M$, a representation into $\PSL_2\C$ is bent in opposite directions in the first and the second factor of $\rchi \times \rchi$  (\S \ref{sBendingFramedRepresentataions}). 
Therefore, when applying $B_M$ to a Fuchsian representation,   the representation in the second factor is the complex conjugate of the representation in the first factor. 
Therefore $B_M \circ \iota_M (\rho)$ is $(b_M (\rho),  b_M^\ast(\rho) )$ for $ \rho  \in \TT$, as desired. 
The analyticity of the mapping was already proven in \Cref{BendingIsAnalytic}.
\end{proof}

\section{Properness of the complex bending map along a non-separating loop}\Label{sPropernessForNonseparatingLoop}

\begin{theorem}\Label{ProperCase}
Let $\ell$ be a non-separating oriented loop on $S$ with weight not equal to  $\pi$ modulo $2\pi$. 
Then, the complex bending map $B_\ell\col X_\ell \to \rchi \times \rchi$ is proper. 
\end{theorem}

\begin{corollary}\Label{CloseAnalyticSet}
The image of  $B_\ell$  is a closed analytic set in $\rchi \times \rchi$.
\end{corollary}
\begin{remark}
By Theorem \ref{BendingTeich},  the properness of the complex bending map $B_M$ fails if $M$ contains a loop $\ell$ of weight $\pi$ modulo $2\pi$.
In this case, $\Im B_M$ is not closed:  for every representation $\rho\col \pi_1(S) \to \PSL_2\C \times \PSL_2\C$ in $\Im B_M$, the holonomy of $\xi(\ell)$ is non-parabolic in both factars. This non-parabolic property fails for some representations in the closure of $\Im B_M$, as we have seen in the proof of \Cref{NonpropernessOfBendingTeichmullerSpaces}.

On the other hand, it is still plausible that the image of $B_M$ is a closed analytic subset of $\rchi \times \rchi$ for every weighted multiloop $M$ on $S$ as long as the weight of each loop is not equal to $\pi$ modulo $2\pi$. 
\end{remark}
   
Pick $\theta  \in (0, \pi)$.
Let
$$E_\theta =  \{(\gam, e) \in \PSL_2\C \times \PSL_2\C  \mid e \text{ is elliptic of rotation angle $\theta$}   \}.$$
Clearly, for every $(\gam, e) \in \E_\theta$, $\tr^2 e$ is a fixed constant in $(0, 4)$ only depending on $\theta$. 
Thus $E_\theta$ is a smooth affine algebraic variety. 
Then $\PSL_2\C$ acts on $\E_\theta$ by conjugating both parameters $\gam$ and $e$ simultaneously. 
Let $\mathcal{E}_\theta$ be the GIT-quotient $E_\theta \sslash \PSL_2\C$.
Then $\mathcal{E}_\theta$ is an affine algebraic variety.
Then the following holds.
\begin{lemma}\Label{RankTwoFreeGroup}
The analytic mapping $E_\theta \sslash \PSL_2\C \to \C^2$ defined by $\phi\col (\gam, e) \mapsto (\tr^2 \gam, \tr^2 \gam e)$ is a proper mapping.
\end{lemma}
\begin{proof}
The map
$\SL_2\C \times \SL_2\C \sslash \SL_2\C \to \C^2$ given by $(\gam, e) \mapsto ( \tr \gam, \tr e, \tr \gam e)$ is a biholomorphic map (see for example,  \cite{Goldman09FrickeSpaces}). 

Let $(\alpha_i, e_i)$ be a sequence in $\E_\theta \sub \PSL_2\C \times \PSL_2\C \sslash \PSL_2\C$ which leaves every compact as $i \to \infty$. 
Pick any lift  $(\ti\alpha_i, \ti{e}_i) \in \SL_2\C \times \SL_2\C \sslash \SL_2\C$ of $(\alpha_i, e_i)$ for each $i$.
Then $(\ti\alpha_i, \ti{e}_i)$ also leaves every compact set as $i \to \infty$. 

By a basic trace identity,  we have  $\tr \ti\alpha_i \ti{e}_i + \tr \ti\alpha_i \ti{e}^{-1}_i = \tr \ti\alpha_i \tr \ti{e}_i.$
Therefore, since $\tr \ti{e}_i$ is a fixed non-zero constant, up to a subsequence, either $\tr \ti\alpha_i$ or $\tr \ti\alpha_i \ti{e}_i$ diverges to $\infty$ as $i \to \infty$. 
Thus the image $\phi (\alpha_i, e_i)$ leaves every compact in $\C^2$ as $i \to \infty$.
\end{proof}

Since $\ell$ is non-separating, we can
pick a generating set $\{\gam_1, \dots, \gam_{2g}\}$ of $\pi_1(S)$ such that $\gam_1, \dots, \gam_{2g}$ correspond to loops on $S$ intersecting $\ell$ exactly once.
Let $\eta_i  = [\rho_i, (u_i, v_i)] \in X_\ell$ be a sequence which leaves every compact in $X_\ell$. 

Let $w(\ell)$ denote the weight of $\ell$, and let $e_i \in \PSL_2\C$ be the elliptic element by angle $w(\ell)$ along the geodesic from $u_i$ to $v_i$. 
Then we can normalize $(\rho_i, (u_i, v_i))$, by an element of $\PSL_2\C$,  so that $e_i \in \PSL_2\C$ is independent of $i$;
let $e$ denote this elliptic element in $\PSL_2\C$.

As $\eta_i$ leaves every compact and $\gam_1, \dots, \gam_n $ form a generating set of $\pi_1(S)$, then there is $k \in \{1, \dots, n\}$ such that, up to a subsequence, $\rho_i (\gam_k)$ leaves every compact subset as $i \to \infty$. 

Then, since $\gam_k$ intersects $\ell$ exactly at once,  
by the properness of Lemma \ref{RankTwoFreeGroup}, the image $B_\ell (\eta_i) \gam_k$ also leaves every compact set as $i \to \infty$. 
(The hypothesis of the weight  of $\ell$ not being $\pi$ corresponding to the rotation angle of $e$ not being $\pi$ in Lemma \ref{RankTwoFreeGroup}.)
This immediately implies the properness of $B_\ell$, and completes the proof of Theorem \ref{ProperCase}.


\section{Symplectic property}\Label{sSymplectic}
In this section, we prove the symplectic property of the bending maps. 
Complex Fenchel-Nielsen coordinates on the quasi-Fuchsian space are introduced by \cite{Kourouniotis41ComplexLengthCoordinatesQuasiFuchsian} and \cite{Tan94ComplexFenchelNielsenForQuasiFuchsian}, and the coordinates holomorphically extend to most part of the character variety $\rchi$.
We explicitly explain the subset of $\rchi$ where the complex Fenchel-Nielsen coordinates are defined.

Let $M$ be a maximal multiloop on $S$. 
Then $M$ contains $3g - 3$ loops, where $g$ is the genus of $S$.  
 Let $\rchi_M^h$ be the  (Euclidean) open full-measure subset of $\rchi$ consisting of $\rho\col \pi_1(S) \to \PSL_2\C$ such that
 \begin{itemize}
 \item all loops of $M$ are hyperbolic, and
 \item for each component $P$ of $S \minus M$, the restriction of $\rho$ to $\pi_1(P)$ is irreducible. 
\end{itemize}
 
 Pick (real) Fenchel-Nielsen coordinates on the Teichmüller-Fricke space $\TT$ with respect to $M$ (see \cite{FarbMarglit12} for example). 
 Let $\C_{+} = \{ z \in \C  \mid {\rm Re}\, z > 0 \}.$
For each $\rho\col \pi_1(S) \to \PSL_2\C$ in $\rchi_M^h$,  
 let $\ell_i \in \C_+ / 2 \pi I \Z$ be the complex translation length of $\rho(m_i)$:
When we  $\ell_i = x_i + I y_i $ in real and imaginary coordinates, $x_i \in \R_{> 0}$ is the (real) translation length and the $y_i \in \R$ is the rotation angle of the screw motion of the hyperbolic element $\rho(m_i)$. 

Clearly, for real representations $\pi_1(S) \to \PSL_2\R$, their length parameters $\ell_1, \dots, \ell_{3g -3}$ are all real numbers. 
 Let $\tau_i \in \C / 2\pi I\Z$ be the twist coordinate along  $\ell_i$ which complexifies the Fenchel-Nielsen twist coordinate, so that the imaginary direction is the direction of bending deformation (where $I$ denotes the imaginary unit). 
Similarly, for real representations $\pi_1(S) \to \PSL_2\R$, their twist parameters $\tau_1, \dots, \tau_{3g-3}$ are all real numbers. 
 \begin{lemma}\Label{ComplexFenchelNiesen} 
 $\rchi_M^h$ is a (Zariski) open dense subset of $\rchi$ and biholomorphic to $(\C_+ / 2 \pi I \Z)^{3g - 3} \bigoplus (\C/ 2\pi I \Z)^{3g - 3}$ by $(\ell_1, \ell_2, \dots, \ell_{3g-3}, \linebreak[2] \tau_1, \tau_2, \dots, \tau_{3g-3})$.
\end{lemma}
\begin{proof}
The mapping $\rchi_M^h \to (\C_+ / 2 \pi I \Z)^{3g - 3} \bigoplus (\C/ 2\pi I \Z)^{3g - 3}$ is a holomorphic mapping, as the coordinates are given by the traces of loops. 

Given a pair of pants $P$, the irreducible representations $\pi_1(P)$ are algebraically parametrized by the holonomy traces of the three boundary components of $P$ (\cite{Vogt1889} \cite{Fricke1896}; see also \cite{Goldman09FrickeSpaces}). 
Now let $P$ be a component of $S \minus M$. 
Then  $\rho \in \rchi_M^h$, the $\rho \vert \pi_1(P)$ is parametrized by the complex length coordinates of the boundary components of $P$. 

For a loop $m_i$ of $M$, let $F$ be the component of $S \minus (M \minus m_i)$ which contains $M$. 
Then the representation on $\pi_1(F) \to \PSL_2\C$ is determined by the twisting parameter $\tau_i$ of $m_i$, and the length parameters $\ell_i$ of $m_i$ and the boundary loops of $F$. 

    We can take a finite generating set of $\pi_1(S)$ corresponding to loops in $M$ and loops intersecting only one loop of $M$.
    Then, the parameters  $(\ell_1, \ell_2, \dots, \ell_{3g-3}, \linebreak[2] \tau_1, \tau_2, \dots, \tau_{3g-3})$ determine a representation $\pi_1(S) \to \PSL_2\C$ up to conjuegation, since the image of the generating set are determined by the arugment above. 
   (See  \cite{Tan94ComplexFenchelNielsenForQuasiFuchsian, Kourouniotis41ComplexLengthCoordinatesQuasiFuchsian} for complex Fenchel-Nielsen coordinates.)
We see that the mapping is biholomorphic. 
\end{proof}

Due to Platis \cite{Platis01ComplexSymplecticGeometryOfQuasi-FuchsianSpace} and Goldman \cite{Goldman04}, the complex Fenchel-Nielsen coordinates yield Darboux coordinates for Goldman's complex symplectic structure.  
 $$w_G = \Sigma_{i = 1}^{3g-3} d \ell_{m_i}^\C \wedge d t_{m_i}^\C.$$ 
(see Loustau \cite{Loustau15ComplexSymplecticGeometry} for details).
To be concrete and self-contained, we first explain the Darboux coordinates on $\rchi_M^h$. 
\begin{lemma}\Label{Darboux}
Let $M = m_1 \sqcup m_2 \sqcup \dots \sqcup m_{3g-3}$ be a maximal multiloop on $S$. 
Then $w_G = \Sigma_{i = 1}^{3g-3} d \ell_{m_i}^\C \wedge d t_{m_i}^\C$ on $\rchi_M^h$. 
\end{lemma}

\begin{proof}
The symplectic structure $w_G$ is a complex symplectic structure, so that the 2-form changes holomorphically in $\rchi$. 
On the Fricke-Teichmüller space space, $w_G | \TT$ is given by $\Sigma d \ell_{m_i}^\R \wedge d t_{m_i}^\R$.
Therefore, since the complex Fenchel-Nielsen coordinates are holomorphic coordinates (\Cref{ComplexFenchelNiesen}),  $w_G = \Sigma d \ell_{m_i}^\C \wedge d t_{m_i}^\C$ on $\rchi_M^h$. 
\end{proof}
Then these Darboux coordinates on $\rchi_M^h$ give the symplectic property of the real bending map. 
\begin{proposition}\Label{MultiloopBendingIsSymplectic}
If $M$ is a weighted multiloop on $S$, then $b_M\col \TT \to \rchi$ is a symplectic embedding. 
\end{proposition}
\begin{proof}
As $M$ may {\it not} be maximal, we pick a maximal multiloop $M'$ on $S$ containing $M$. 
Set $m_1, m_2, \dots, m_{3g-3}$ to be the loops of $M'$.
Let $w_1, w_2, \dots w_{3g-3} \in \R_{\geq 0}$ be the weight of the loops of $M'$ (so that, if $\ell_i$ is {\it not} a loop of the original multiloop $M$, then $w_i = 0$).
The Teichmüller-Fricke space $\TT$ is a component of the real slice of $\rchi_M^h$.
In the Darboux coordinates of \Cref{Darboux}, the real bending map $b_M\col \TT \to \rchi$ extends to $\hat{b}_M \col \rchi_M^h \to \rchi_M^h$ by the translation 
$$(\ell_1, \dots, \ell_{3g-3}, \tau_1, \dots, \tau_{3g-3}) \mapsto (\ell_1, \dots, \ell_{3g-3}, \tau_1 + w_1 I, \dots, \tau_{3g-3} + w_{3g-3} I).$$ 
As it is a translation in the Darboux coordinates, $b_M\col \TT \to \rchi$ is clearly a symplectic embedding. 
\end{proof} 

By the limiting argument, all real bending maps are symplectic. 
\begin{theorem}\Label{RealSymplectic}
For every $L \in \ML$, $b_L\col \TT \to \rchi$ is a symplectic embedding w.r.t. Goldman's symplectic structure.
\end{theorem}

\begin{proof}
Let $\ell_i$ be a sequence of weighted loops which converges to $L$ in $\ML$ as $i \to \infty$. 
(Recall that $b_{\ell_i}\col \TT \to \rchi$ is a real-analytic embedding.)
For each $\tau \in \TT$,
the tangent space of $b_{\ell_i}$ at $\tau$ converges to the tangent space of $b_L$ at $\tau$.
By \Cref{MultiloopBendingIsSymplectic},  $b_{\ell_i}\col \TT \to \rchi$ is a symplectic embedding for each $i = 1, 2, \dots$. 
Therefore, by the continuity of the symplectic structure $w_G$,  the limit $b_L$ is also symplectic at  $\tau$. 
\end{proof}

\subsection{Symplectic property for complex bending map}
As  $X_M \minus X_M^p \to \rchi \minus \rchi_M^p$ is an analytic covering map,   $X_M \minus X_M^p$ has a pull-back symplectic structure.

A representation $\rho\col \pi_1(S) \to \PSL_2\C$ is {\sf reductive}, if the Zariski-closure of the image $\Im \rho \sub \PSL_2\C$ is reductive. 
(That is, the maximal normal unipotent subgroup of  $\Im \rho$ is the trivial group.)
Then a representation $\rho\col \pi_1(S) \to \PSL_2\C$ is non-reductive, if and only if $\Im \rho$ is conjugate to a subgroup consisting of upper triangular matrices which contains at least one (non-identity) parabolic element. 
Let $X^r_M$ be the set of framed representations $\eta = [\rho, (u_i, v_i)]$ of $X_M$ such that $\rho$ is a reductive representation other than the trivial representation. 
\begin{theorem}\Label{Symp}
The restriction of $B_M$ to $X_M^r \minus X_M^p$ is a complex symplectic map.  
\end{theorem}
\begin{proof}
We show that the restriction of $b_M^{\pm}\col X^r_M \to \rchi$ is symplectic on $\rchi_M^h$.
  For every framed representation in  $R_M^r$, its $\PSL_2\C$-orbit is a closed subset of $R_M$ and biholomorphic to $\PSL_2\C$.
Therefore, the reductive part $X_M^r$ is contained in the smooth part of the framed character variety $X_M$. 

Recall that $\rchi_M^h$ is the subset of character variety $\rchi$ consisting of $\pi_1( S) \to \PSL_2\C$ such that every loop of $M$ maps to a hyperbolic element by $\rho$ and for every component $F$ of $S \minus M$,   the restriction of $\rho$ to the fundamental group of $F$ is irreducible. 

 Let $X_M^h$ denote the full-measure subset of $X_M^r$ consisting of framed representations whose representations are in $\rchi_M^h$. 
Then $X_M^h$ is a (Euclidean) open dense full-measure subset of $X_M$. 
The complex bending map $B_M$ is symplectic on $X_M^h$ by \Cref{Darboux}.
Therefore,  by continuity, $B_M$ is symplectic on  $X_M^r \minus X_M^p$.
  \end{proof}

\begin{figure}

\begin{tikzpicture}[
  every node/.style={inner sep=5pt},
  arrow/.style={->},
    hookarrow/.style={{Hooks[right]}->},
]

\node (X) at (0, 0) {$X_\ell \minus X_\ell^p$};
\node (ccp) at (2.5, 0) {$\rchi \minus \rchi^p $};
\node (c) at (3.6, 0) {$\subset \rchi$};
\node (cc) at (2.5, -1.5) {$\rchi \times \rchi$};

\draw[arrow] (X) -- (ccp) node[midway, above] {\tiny $b_M$};
\draw[hookarrow ] (X) -- (cc) node[midway, left] {\tiny $B_M$};
\draw[hookarrow] (ccp) -- (cc) node[midway, right] {\tiny $b_M^- \times b_M^+$};
\end{tikzpicture}

 \caption{A local commutative diagram describing the complexification of the real bending map.}\label{fSymplecticEmbedding}
\end{figure}


\section{The general complex bending variety}\Label{sComplexBendingVariety}
Let $L$ be a non-empty measured lamination on $S$. 
Let $\ell_i$ be a sequence of non-separating weighted oriented loops on $S$ converging to $L$ as $i \to \infty$. 
By \Cref{CloseAnalyticSet}, the image of $B_{\ell_i} \col X_{\ell_i} \to \rchi \times \rchi$ is a closed complex-analytic subset of $\rchi \times \rchi$. 

\begin{theorem}\Label{ComplexBendingVaraiety}
The analytic set $\Im B_{\ell_i}$ converges,  up to a subsequence,  to a closed complex-analytic subset of $\rchi \times \rchi$ as $i \to \infty$. 
\end{theorem}

\proof
By \Cref{Symp}, the bending maps  $b_{\ell_i}^{\pm}\col X_{\ell_i} \to \rchi$ is a complex symplectic mapping on $X^r_{\ell_i} \minus X_M^p \to \rchi$.
 \begin{claim}
 Let $\ell$ be an essential simple closed curve with weight $w$ not equal to $\pi$ modulo $2\pi$. 
 Then 
 $b_\ell^{\pm}\col X_\ell \to \rchi$ is two-to-one mapping on $X^r_{\ell_i} \minus X^p_{\ell_i}$.  
\end{claim}
 \begin{proof}
 Let $\rho\col \pi_1(S) \to \PSL_2\C$ be a representation in $\rchi_{\ell}^r \minus \rchi_\ell^p$. 
 Let $\alpha \in \pi_1(S)$ be an element representing $\ell$. 
  As $\rho(\alpha)$ is {\it not} a parabolic element or the identity, pick a framing $(u, v)$ of $\ell$, where $u, v$ are the fixed points of $\rho(\alpha)$.
 
 Since $b^+_\ell$ and $b^-_\ell$ bend each representation in opposite directions, they are inverse to each other, when the framing is kept:
 Namely  $b^+_\ell$ takes  $b^-(\rho, (u, v)) \in \rchi$ with the framing $(u, v)$ back to $(\rho, (u, v)) \in X_\ell$. 
  Similarly $b^-_\ell$ takes  $b^+(\rho, (u, v)) \in \rchi$ with the framing $(u, v)$ back to $(\rho, (u, v)) \in X_\ell$. 
 Therefore the inverse image $(b^+_\ell)^{-1} (\rho)$ consists of $( b^-(\rho, (u, v)), (u, v) )$ and $( b^+(\rho, (u, v)), (v, u) )$. 
Moreover, the above inverse relation of $b^+_\ell$ and $b^-_\ell$ implies that there are no other framed representations mapping to $\rho$ by $b^+_\ell$.
Hence $b^+_\ell$ is a two-to-one mapping on  $X^r_{\ell_i} \minus X^p_{\ell_i}$.

 One can similarly prove that $b^-_\ell$ is a two-to-one mapping on  $X^r_{\ell_i} \minus X^p_{\ell_i}$.
\end{proof}

Let $\omega$ denote the Goldman symplectic structure on $\chi$.     
Then the complex volume form of $\chi$ is given by $\omega^n$.
With respect to the Darboux coordinates  $d z_1 \wedge d z_2 \wedge \dots \wedge d z_n$.
As $B_{\ell_i}$ is complex symplectic almost everywhere,  
it preserves the complex volume (i.e. Jacobian is one).
Therefore, since $B_{\ell_i}$ is a two-to-one mapping almost everywhere (see \S\ref{sForgetfulMap}),   the volume of the analytic set $\Im B_{\ell_i}$ is locally finite in $\rchi \times \rchi$ and uniformly bounded in $i$.
Hence, up to a subsequence, the closed $\C$-analytic set $\Im B_{\ell_i}$ converges to a closed $\C$-analytic set in $\rchi \times \rchi$ as $i \to \infty$ by Bishop's theorem \cite[Corollary in p205]{Chirka89ComplexAnalyticSets}). 
\Qed{ComplexBendingVaraiety}

   \begin{remark}
Since $\Im B_{\ell_i}$ is symplectic in the smooth part, the closed $\C$-analytic set in the limit is also $\C$-symplectic in the smooth part. 
 \end{remark}
 
Let $\QF$ be the quasi-Fuchsian space, which contains the Fricke space $\TT$. 
Then,  the domain $X_{\ell_i}$ of $B_{\ell_i}$ contains $\QF$ for all $i = 1, 2, \dots$.

Let $\QF_i$ be the open subset of $X_{\ell_i}$ so that the restriction of $b^+_{\ell_i}$ to the Fuchsian space $\TT$ in $\QF_i$ is the real bending map $b_{\ell_i}$.
Moreover,  $L$ is realizable for all quasi-Fuchsian representations, i.e. there is a $\rho$-equivariant pleated surface $\ti{S} \to \H^3$ whose pleating lamination contains the geodesic lamination supporting $L$. 
Therefore the $\R$-analytic bending map $b_L\col \TT \to \rchi$ extends to a holomorphic mapping $b_L\col \QF \to \rchi$.
Similarly to the complex bending map $B_M\col X_M \to \rchi \times \rchi$ for a weighted multiloop,  we can define $B_L\col \QF \to \rchi \times \rchi$ by bending $\rho \col \QF \to \rchi \times \rchi$ by $L$ and by $-L$, 
$$\rho \mapsto (b_L (\rho), b_{-L}(\rho)).$$
Then $B_L$ complex analytically embeds $\QF$ into $\rchi \times \rchi$. 
Therefore $B_{\ell_i} | \QF_i$ converges to $B_L \vert \QF$ as $i \to \infty$.
The limit of \Cref{ComplexBendingVaraiety} has the canonical irreducible component $\BB_L$ containing $B_L(\QF)$.

Indeed this component is independent of the choices we made, since otherwise there are distinct irreducible $\C$-analytic sets containing $B_L(\QF)$ against the uniqueness of the decomposition of complex analytic sets into irreducible ones. 
\begin{corollary}
The irreducible component $\BB_L$ containing the real bending map image $\psi \circ b_L(\TT)$ is independent of the choice of the sequence $\ell_i$ converging to $L$ and the subsequence in \Cref{ComplexBendingVaraiety}. 
\end{corollary}

We obtained a unique irreducible closed complex-analytic set $\BB_L$ in $\rchi \times \rchi$ containing the real-analytic subvariety $\psi \circ b_L (\TT)$, and it is symplectic on the smooth part.
We finished the proof of \Cref{GeneralComplexBendingVarietiy}.

For a general measured lamination $L$, 
the complex bending variety $\BB_L$ is constructed as the limit as above.
Since $L$ is realizable for all quasi-Fuchsian representations, 
$\BB_L$ contains quasi-Fuchsian representations bent in opposite directions along $L$ analogous to the complex bending map along a weighted multiloop defined in \S \ref{sBendingFramedRepresentataions}.
 We can more generally hope that generic representations in $\BB_L$ have similar properties.

\begin{question}
Let $L$ be a measured lamination on $S$ (containing a non-periodic leaf). 
Let $\mathcal{B}_L$ be the complex bending variety of $L$ in $\rchi \times \rchi$. 
Let $\eta = (\rho_1, \rho_2)$ be a generic point in $\BB_L$.
Are there a $\rho_1$-equivariant pleated surface $\beta_1 \col \H^2 \to \H^3$ and a $\rho_2$-equivariant pleated surface $\beta_2 \col \H^2 \to \H^3$ both realizing $|L|$, such that $\rho_2$ is obtained by bending $\rho_1$ along the geodesic lamination $|L|$?\end{question}

\bibliography{Bending_Teichmuller_spaces_and_character_varieties_2026_5_17arXiv}

\bibliographystyle{alpha}

\end{document}